\documentclass[10pt,draftcls,onecolumn]{IEEEtran}

\IEEEoverridecommandlockouts                        
 \usepackage{graphicx}

\renewcommand{\baselinestretch}{1.25}

\usepackage{wrapfig} 
\usepackage{times} 
\usepackage{amsmath}
\usepackage{amsmath}
\usepackage{subfig}

\usepackage{amssymb}  
\usepackage{dsfont} 
\usepackage[dvipsnames,usenames]{color}
\usepackage[colorlinks,%
linkcolor=BrickRed,%
filecolor =BrickRed,%
citecolor=RoyalPurple,%
]{hyperref}

\usepackage{amsmath} 
\usepackage{amssymb}  
\usepackage{color}
\usepackage{cite}

\def\mbR{\mathbb{R}}

\newtheorem{lemma}{Lemma}
\newtheorem{definition}{Definition}
\newtheorem{assumption}{Assumption}
\newtheorem{proposition}{Proposition}
\newtheorem{corollary}{Corollary}

\newcommand{\an}[1]{{\color{black}#1}}

\newcommand{\us}[1]{{\color{black}#1}}

 \newcommand{\remove}[1]{}
\newcommand{\EXP}[1]{\mathsf{E}\!\left[#1\right] }

\newcommand{\fhat}[1]{\hat{f} }

\def\sF{\mathcal{F}}
\def\Kbar{\bar{K}}
\def\Real{\mathbb{R}}
\def\g{\gamma}

\def\e{\epsilon}
\def\a{\alpha}

\usepackage{ulem}
\author{Farzad~Yousefian,   
        Angelia~Nedi\'c, and   
		Uday V.~Shanbhag \thanks{The authors are with the Department of Industrial and Enterprise
Systems Engineering, University of Illinois, Urbana, IL 61801, USA,
{\tt\small \{yousefi1,angelia,udaybag\}@illinois.edu}. Nedi\'{c} and
Shanbhag gratefully acknowledge the support of the NSF through the award
NSF CMMI 0948905 ARRA.}}

\title{\LARGE \bf On Stochastic Gradient and Subgradient Methods \\
				with Adaptive Steplength Sequences}
\begin{document}
\maketitle
\thispagestyle{empty}
\pagestyle{plain}
\vspace{-0.5in}
\begin{abstract}
Traditionally, stochastic approximation (SA) schemes have been
popular choices for solving stochastic optimization problems.
However, the performance of standard SA implementations can vary
significantly based on the choice of the steplength sequence, and in
general, little guidance is provided about good choices. 
Motivated by this gap, in the first part of the paper,
we present two adaptive steplength schemes for strongly convex differentiable 
stochastic optimization problems, equipped with convergence theory, 
that aim to overcome some of the reliance on user-specific parameters.  
Of these, the first scheme,
referred to as a recursive steplength stochastic approximation (RSA)
scheme, optimizes the error bounds to derive a rule that expresses the
steplength at a given iteration as a simple function of the steplength
at the previous iteration and certain problem parameters.  
The second scheme, termed as a cascading
steplength stochastic approximation (CSA) scheme, maintains the
steplength sequence as a piecewise-constant decreasing function with the
reduction in the steplength occurring when a suitable error threshold is met. 

In the second part of the paper, we allow for nondifferentiable
objectives but with bounded subgradients over a certain domain. In such
a regime, we propose a local smoothing technique, based on random local
perturbations of the objective function, that leads to a differentiable
approximation of the function.  Assuming a uniform distribution on the
local randomness, we establish a Lipschitzian property for the gradient
of the approximation and prove that the obtained Lipschitz bound grows
at a modest rate with problem size. This facilitates the development of
an adaptive steplength stochastic approximation framework, which now
requires sampling in the product space of the original measure and the
artificially introduced distribution. The resulting adaptive steplength
schemes are applied \an{to three } stochastic optimization
problems. In particular, we observe that both schemes perform well in
practice and display markedly less reliance on user-defined parameters.
\end{abstract} \maketitle

\section{Introduction}\label{sec:introduction}
The use of stochastic gradient and subgradient schemes for the solution
of stochastic convex optimization problems has a long tradition,
beginning with an iterative scheme, first proposed by Robbins and
Monro~\cite{robbins51sa}, that relied primarily on noisy gradient
observations.  Research by Ermoliev and his
coauthors~\cite{Ermoliev69,Ermoliev76,Ermoliev83,Ermoliev88} focused
largely on quasigradient (subgradient) methods and considered a host of
stochastic programming problems, amongst them being two-period
recourse-based problems (see~\cite{birge97introduction}). To accelerate
the convergence of stochastic subgradient methods, ergodic sequences,
arising from the averaging of  iterates, have been employed
in~\cite{ruz83averaging,Polyak92,Polyak01,Nemirovski09}. Often gradient
computations are either costly or unavailable; in such instances, a
finite-difference approximation of the gradient can be constructed as
first observed by Kiefer and Wolfowitz~\cite{Kiefer52}. While standard
finite-difference techniques  perturb one direction at a time to obtain
gradient estimates, simultaneous perturbation stochastic approximation
techniques simultaneously perturb all directions and general require
fewer function evaluations~\cite{Spall1992,Fu03}. More recently, there
has been a significant interest in the application of ODE-based methods
for investigating the stability and convergence of the associated
stochastic approximation schemes~\cite{borkar00ode,Borkar08}. An elegant
exposition of these methods may be found in the monographs by 
Polyak~\cite{Polyak87}, Kushner and Yin~\cite{Kush03}, and 
Borkar~\cite{Borkar08}.

Sample-average approximation
(SAA) techniques~\cite{shap03sampling} are often viewed as an
alternative to stochastic approximation techniques and are particularly 
attractive when
approximate solutions to the problem are desired in an offline manner.
This approach relies on using a sample from the underlying distribution
to construct a deterministic sample-average problem, which can be
subsequently solved via standard nonlinear programming solvers, as seen
in~\cite{linderoth02empirical}.  In~\cite{Nemirovski09}, the
authors demonstrate that stochastic approximation schemes are shown to
be competitive with SAA techniques. Importantly
in~\cite{Nemirovski09}, Nemirovski et al.\ develop a robust SA
scheme that determines an optimal constant steplength for minimizing
the theoretical error over a pre-specified number of steps.
Mirror-descent generalizations of SA, that rely on a suitably defined
prox-mapping, are also presented in~\cite{Nemirovski09} (also
		see~\cite{nemirovski78cesaro}), while validation analysis is 
provided in~\cite{Lan09}.

Stochastic gradient algorithms have also been found to be effective in
solving large deterministic problems such as convex feasibility
problems~\cite{Polyak01,Alamo09,nedic10cdc}, 
\an{feasibility problems arising in control~\cite{Calafiore00,PolyakTempo01}} 
and some specially structured large-scale convex problems
in~\cite{Nedic01,Nedic01a,Nedic01b}.  
Distributed consensus-based
stochastic subgradient methods for minimizing a convex objective over a
network have been recently developed and studied
in~\cite{Sundhar08,Sundhar09,Sundhar09a}.  The success of gradient-based
methods in solving monotone variational
inequalities~\cite{facchinei02finite} has prompted the study of similar
techniques for contending with stochastic variational inequalities. In
fact, Jiang and Xu~\cite{Houyuan08} develop precisely such a scheme for
the solution of strongly monotone stochastic variational inequalities
and regularized variants were presented in~\cite{koshal10single} to
allow for application to monotone stochastic variational inequalities.
Finally, stochastic generalizations of the   mirror-prox schemes were
examined in~\cite{Nem08} and allowed for the solution of monotone
variational inequalities. 

While stochastic approximation schemes have proved successful,
other avenues exist for addressing stochastic programs. For instance, an
alternate approach lies in using sample-average approximation methods,
that obtain estimators to the optimal value and solution of the problem
through the solution of deterministic problem in which the expectation
is replaced by a sample-average.  Convergence theory for the obtained
estimators is examined by Shapiro~\cite{shap03sampling}. Decomposition
schemes, that leverage cutting-plane methods, have also been
particularly successful in addressing two-period stochastic
linear~\cite{vanslyke69lshaped}, convex~\cite{ruz04decomp} and nonconvex
programs~\cite{kulkarni10benders} while  a scalable matrix-splitting
decomposition scheme is presented in~\cite{shanbhag06forward} for
two-period stochastic Nash games. 

In this paper, we consider adaptive stochastic gradient and subgradient
methods for solving constrained stochastic convex optimization problems.
The novelty of our work can be categorized as follows: (1) the
development and analysis of two adaptive stepsize rules; and (2) the
development of a local function smoothing technique. Next, we provide
some motivation and a more elaborate description of each. 

In stochastic gradient methods (cf.~\cite{Ermoliev69,Ermoliev76,Ermoliev83,
 Ermoliev88,Polyak87, Bertsekas00,borkar00ode,Kush03,Borkar08}), 
the almost-sure convergence
of such methods is guaranteed assuming that the stepsize is diminishing
but not too rapidly, i.e., the stepsize is proportional to
$\frac{1}{k^a}$ with $\frac{1}{2}< a\le 1$.  Typically, there is no
guidance on the specific choice of the sequence and problem parameters
play little role in refining this choice. In contrast, in this paper, we
propose specific (adaptive) rules for the stepsize values that exploit
the information about the objective function.  Accordingly, our {\em
first} goal lies in examining whether one can construct a convergent
scheme under an adaptive stepsize rule that is more reflective of the
problem setting. Through out this part of the paper, we assume that the
integrand of the expectation is a random convex differentiable function.  
Under a Lipschitzian assumption on the gradient, we propose two different
adaptive stepsize rules:
\begin{enumerate}
\item[(a)] {\em Recursive stepsize rule:} In attempting to minimize the
bound on the expected error, we develop a recursive scheme for
specifying the stepsize that requires only the steplength at the
previous parameter and some problem parameters. 
Global convergence and rate estimates for this
scheme are developed. 
\item[(b)] {\em Cascading stepsize rule:} It is well-known that under
suitable assumptions, fixed-stepsize schemes are guaranteed to converge
to a compact region containing the solution set of the original problem.
We consider a modified version of such a scheme where the trajectory
moves to successively smaller compact regions containing the solution
sets. Furthermore, as soon as the trajectory of iterates reaches within
a bound of the solution set, the steplength is updated allowing the
sequence to make further progress. Effectively, we consider a method in
which the steplength sequence can be viewed as one where the stepsize is
maintained constant with drops or cascades in stepsize occurring at
particular epochs.  While the scheme has intuitive appeal, we provide a
theoretical support for the convergence of such an algorithmic
framework. \end{enumerate}

When the random integrands arising in such stochastic problems are nonsmooth, 
direct application of known SA schemes is impossible. Contending with 
nonsmoothness in mathematical programming
is often managed through avenues that rely on the solution of a sequence
of smoothed problems (cf.~\cite{jiang00smooth,facchinei96smoothing}). In
a stochastic regime, an approach for addressing such problems is through a
technique of {\em global smoothing}, as considered in~\cite{Gupal79} and 
more recently in~\cite{DeFarias08}.\footnote{See~\cite{Rusz04} for a scheme 
that develops an approximation method for addressing a class of separable 
piecewise-linear stochastic optimization
problems with integer breakpoints.} This involves modifying the original
problem by adding a random variable with possibly unbounded support.
However, such a technique is not feasible in when the  objective is
defined over a restricted domain. We present  a local smoothing
technique which leads to a globally differentiable approximation of the
original function with Lipschitz continuous gradients. Furthermore,
through such a smoothing, we derive a Lipschitz constant for the
gradients and show that the constant grows at the rate of $\sqrt{n}$ where 
$n$ is the dimensionality of the problem space. Importantly, this Lipschitzian 
property facilitates the construction of a stochastic approximation framework. 
Consequently, the second part of the paper focuses on  computing
solutions to approximations with smoothed integrands whose gradients are
shown to be provably Lipschitz continuous.

The remainder of the paper is organized as follows. In
Section~\ref{sec:formulation}, we establish the almost-sure convergence
of the classical stochastic approximation algorithm for a constrained
problem with a differentiable convex function with Lipschitz gradients.
In Sections~\ref{sec:adap} and~\ref{sec:cascading}, for a strongly
convex function, we propose and analyze two different stepsize rules,
each motivated by a minimization of an estimate on the expected error
per iteration of the method.  In Section~\ref{sec:localrand}, we
introduce a local randomized smoothing technique for nondifferentiable
convex optimization, and derive its  approximation properties as well as
a bound on the Lipschitz constant of the gradients.  In
Section~\ref{sec:numerical}, we report some numerical results obtained
by applying our proposed stepsize rules and the smoothing technique to
\an{three} test problems and conclude with a discussion in
Section~\ref{sec:conclusions}.

{\it Notation and basic terminology:} We view vectors as columns, and 
write $x^T$ to
denote the transpose of a vector $x$. We use $\|x\|$ to denote the
Euclidean vector norm, i.e., $\|x\|=\sqrt{x^Tx}$. We write $\Pi_X(x)$ 
to denote the Euclidean projection of a vector $x$ on a set $X$.
i.e., $\|x-\Pi_X(x)\|=\min_{y\in X}\|x-y\|$. 
For a convex function $f$ with domain ${\rm dom} f$, 
a vector $g$ is a {\it subgradient of $f$ at $\bar x\in{\rm dom }f$} 
if the following relation holds\footnote{For a differentiable convex $f$, 
the inequality holds with $g=\nabla f(\bar x)$.}:
\[f(\bar x)+ g^T(x-\bar x)\le f(x)\qquad\hbox{for all $x\in {\rm dom} f$}.\]
The subdifferential set of $f$ at $x=\bar x$, denoted by $\partial
f(\bar x)$, is the set of all subgradients of $f$ at $x=\bar x$.
Finally, we write {\it a.s.} for ``almost surely'', and use
$\hbox{Prob}(\mathcal{Z})$ and $\EXP{Z}$ to denote the probability of an
event $\mathcal{Z}$ and the expectation of a random variable $Z$,
	  respectively.
\section{Problem Formulation \us{and Background}}\label{sec:formulation}
\us{In this section, we begin by describing the problem and iterative
scheme of interest (Section~\ref{sec:prob}). This is followed by
Section~\ref{sec:adap-sa} where we provide a short description on
various adaptive schemes in the realm of stochastic approximation.}

\subsection{Problem Formulation}\label{sec:prob}
We consider the following stochastic optimization problem
\begin{equation}\label{eqn:problem}
\min_{x \in X} f(x)= \EXP{F(x,\xi)},
\end{equation}
where $F: \mathcal{D}\times\Omega \to\mathbb{R}$ is a function, the
set $\mathcal{D}\subseteq\mathbb{R}^n$ is open, and the set $X$ is
nonempty with $X\subset\mathcal{D}$. The vector {$\xi:\Omega
\rightarrow \Real^d$} is {a} random vector {with} a
probability distribution on a set $\Omega \subseteq \mathbb{R}^d$,
while the expectation $\EXP{F(x,\xi)}$ is taken with respect to $\xi$.
We use $X^*$ to denote the optimal set of
problem~\eqref{eqn:problem} and $f^*$ to denote its optimal value.
We assume the following:

\begin{assumption}\label{assum:convex}
The set $X\subset \mathcal{D}$ is convex and closed. The function
$F(\cdot,\xi)$ is convex on $\mathcal{D}$ for every {$\xi \in
\Omega$}, and the expected value $\EXP{F(x,\xi)}$ is finite for
every $x \in \mathcal{D}$.
\end{assumption}

Under Assumption~\ref{assum:convex}, the function $f$ is convex over
$X$ and the following relation holds
\begin{equation}\label{eqn:expdiffer}
\partial{f(x)} = \EXP{\partial_x F(x,\xi)}\qquad\hbox{for all }x\in\mathcal{D},
\end{equation}
where $\partial_x F(x,\xi)$ denotes the set of all subgradients of
$F(x,\xi)$ with respect to the variable $x$
(see~\cite{Bertsekas72,Bertsekas73}\footnote{In both of these articles,
 the analysis is for a function defined over $\mbR^n\times\Omega$, but
 {can be extended} to the case of a function defined over
 $\mathcal{D}\times\Omega$ for an open convex set
 $\mathcal{D}\subseteq\mbR^n$.}).

First, we will consider problem~(\ref{eqn:problem}) where $f$ is a
differentiable function with Lipschitz gradients.  Later, we will
allow the function $f$ to be nondifferentiable and we will consider a
local smoothing technique yielding a differentiable function that
approximates $f$ over $X$.  For this reason, we start our discussion by
focusing on a differentiable problem~(\ref{eqn:problem}) and the
following iterative algorithm:
\begin{equation}\label{eqn:algorithm}
\begin{split}
x_{k+1} & =\Pi_{X}\left(x_k-\g_k (\nabla f(x_k)+w_k)\right)
\qquad\hbox{for all }k\ge0,\cr w_k &= \nabla_x F(x_k,\xi_k)-\nabla
f(x_k).
\end{split}
\end{equation}
Here, $x_0\in X$ is a random initial point, $\gamma_k>0$ is a (deterministic) 
stepsize, and $w_k$ is the random vector given by 
the difference between the sampled gradient $\nabla_x F(x,\xi_k)$
and its expectation $\EXP{\nabla_x F(x,\xi)}$ evaluated at $x=x_k$.
Throughout the paper, we assume that $\EXP{\|x_0\|^2}<\infty$.

We let $\sF_k$ denote the history of the method up to time $k$, i.e., 
$\sF_k=\{x_0,\xi_0,\xi_1,\ldots,\xi_{k-1}\}$ for $k\ge 1$ and $\sF_0=\{x_0\}$. 
By Assumption~\ref{assum:convex} and
relation~\eqref{eqn:expdiffer}, it follows that $\nabla
f(x_k)=\EXP{\nabla_x F(x_k,\xi)}$ for a differentiable $F$, implying
that $w_k$ has zero-mean, i.e.,
\begin{equation}\label{eqn:zeromean}
\EXP{w_k\mid\sF_k}=0\qquad\hbox{for all } k\ge0.\end{equation}
Next, we state some additional assumptions on the stochastic gradient
error $w_k$ and the stepsize $\g_k$. 

\begin{assumption}\label{assum:step_error}
The stepsize is such that $\gamma_k>0$ for all $k$. Furthermore, the following 
hold:
\begin{enumerate}
 \item[(a)]
$\sum_{k=0}^\infty \gamma_k = \infty$ and $\sum_{k=0}^\infty
\gamma_k^2 < \infty$.
  \item[(b)] The stochastic errors $w_k$ satisfy
$\sum_{k=0}^\infty \gamma_k^2 \EXP{\|w_k\|^2\mid \sF_k}<\infty$ almost surely.
\end{enumerate}
\end{assumption}
Assumption~\ref{assum:step_error}(b) is satisfied, for example, when
$\sum_{k=0}^\infty \gamma_k^2 < \infty$ and the error $w_k$ is bounded almost 
surely, i.e., $\|w_k\|\le c$ for all $k$ and some scalar $c$ almost surely.

We use the following Lemma in establishing the convergence of
method~(\ref{eqn:algorithm}) (see~\cite{Polyak87}, page 50).

\begin{lemma}\label{lemma:supermartingale}
(Robbins-Siegmund) Let $v_k,$ $u_k,$ $\alpha_k,$ and  $\beta_k$ be
nonnegative random variables, and let the
following relations hold almost surely: 
\[\EXP{v_{k+1}\mid {\tilde \sF_k}} 
\le (1+\alpha_k)v_k - u_k + \beta_k \quad\hbox{ for all } k,\qquad
\sum_{k=0}^\infty \alpha_k < \infty,\qquad
\sum_{k=0}^\infty \beta_k < \infty,\] 
where $\tilde \sF_k$ denotes the collection $v_0,\ldots,v_k$, $u_0,\ldots,u_k$,
$\alpha_0,\ldots,\alpha_k$, $\beta_0,\ldots,\beta_k$. 
Then, almost surely we have 
\[\lim_{k\to\infty}v_k = v, \qquad \sum_{k=0}^\infty u_k < \infty,\]
where $v \geq 0$ is some random variable.
\end{lemma}

We also make use of the following result, which can be found
in~\cite{Polyak87} (see Lemma 11 in page 50).

\begin{lemma}\label{lemma:probabilistic_bound_polyak}
Let $\{v_k\}$ be a sequence of nonnegative random variables, where 
$\EXP{v_0} < \infty$, and let $\{\a_k\}$ and $\{\beta_k\}$
be deterministic scalar sequences such that:
\begin{align*}
& \EXP{v_{k+1}|v_0,\ldots, v_k} \leq (1-\alpha_k)v_k+\beta_k
\qquad a.s \ \hbox{for all }k\ge0, \cr
& 0 \leq \alpha_k \leq 1, \quad\ \beta_k \geq 0, 
\quad\ \sum_{k=0}^\infty \alpha_k =\infty, 
\quad\ \sum_{k=0}^\infty \beta_k < \infty, 
\quad\ \lim_{k\to\infty}\,\frac{\beta_k}{\alpha_k} = 0. 
\end{align*}
Then, $v_k \rightarrow 0$ almost surely, 
$\lim_{k\to\infty}\EXP{v_k}= 0$, and for any $\epsilon >0$,
\[\hbox{Prob}(v_j \leq \e \hbox{ for all }j \geq k) 
\geq 1 - \frac{1}{\e}\left(\EXP{v_k}+\sum_{i=k}^\infty \beta_i\right)
\qquad\hbox{for all $k>0$}.\]
\end{lemma} 

In Sections~\ref{sec:adap} and~\ref{sec:cascading}, we examine 
adaptive steplength schemes for a
strongly  convex function $f$ whose gradients $\nabla f$ are
	Lipschitz continuous over $X$ with constant $L$.
%
%
$X$ is defined as 

\subsection{Adaptive Stochastic Approximation Schemes}\label{sec:adap-sa}

Robbins and Monro~\cite{robbins51sa} proposed the first stochastic
approximation algorithm in 1951 while Kiefer and
Wolfowitz~\cite{Kiefer52} proposed a variant of this scheme in which
finite differences were employed to estimate the gradient.  Asymptotic
distributions of the Robbins-Monro scheme were first examined by
Chung~\cite{chung54stochastic}, leading to an asymptotic normality
result in the one-dimensional regime while generalizations were
subsequently studied by Sacks~\cite{sacks58asymptotic}.

A potential challenge in developing efficient implementations of
stochastic approximation implementations lies in choosing an appropriate
steplength sequence. Kesten~\cite{kesten58accelerated}, in 1957, suggested a 
technique where the steplength sequence adapts to the observed data, which was
further extended by Kushner and Gavin~\cite{kushner73extensions} to 
the multi-dimensional regime, \an{while its accelerations were studied 
in~\cite{delyon93accelerated}}. 
Sacks~\cite{sacks58asymptotic} proved that, under
suitable conditions, a choice of the form $a/k$ 
(where $k$ is the iterate index) 
is optimal from the standpoint of minimizing the asymptotic variance. Yet,
the challenge lies in estimating the ``optimal'' $a$.   Subsequently,
Ventner~\cite{ventner67extension} in what is possibly amongst 
the first {\em adaptive} steplength SA schemes, considered 
sequences  of the form $a_k/k$ where $a_k$ is updated by leveraging past 
information. Notably, Chung~\cite{chung54stochastic} also examined 
the asymptotic variance properties of SA
	when steplength choices of the form $a/k^{1-\alpha}$ with $\alpha <
	{1\over 2}$ are used. In related work on adaptive schemes, 
Lai and Robbins~\cite{lai79adaptive} considered
schemes of the form $a_k/k$ where $a_k$ is a strongly consistent
estimator of $\nabla f(x)$ in a stochastic root-finding problem. One
choice for obtaining $a_k$ is through the use of least-squares
estimators. Multivariate generalizations of this analysis were suggested
by Wei~\cite{wei87multivariate} in 1987 and again, it was observed that the 
Jacobian of the vector function assumes relevance in constructing efficient
steplength sequences.

An alternative to using a single sample was suggested by  
Spall~\cite{Spall1992} and relied on obtaining gradient estimates through a 
{\em simultaneous
	perturbation} of all the parameters.  An adaptive generalization
	of this scheme, proposed by the same
	author~\cite{spall00adaptive,spall09feedback}, employed an additional
	recursion to the standard projected gradient step that attempted to
	estimate the Jacobian in root finding problems or the Hessian in
	optimization problems. Accordingly, the modified update rule is of
	the form 
\begin{equation}\label{eqn:algorithm2}
x_{k+1}  =\Pi_{X}\left(x_k-\g_k  H_k^{-1}(\nabla f(x_k)+w_k)\right),
\end{equation}
where $H_k$ is an estimate of the Hessian matrix of the objective.
Clearly, this also falls under the regime of an {\em adaptive
steplength} scheme. Related adaptive schemes may also be found in the
work by Bhatnagar~\cite{Bhatnagar05adaptivemultivariate,
Bhatnagar07adaptivenewton-based}.

A final remark is in order regarding the key difference  between our
proposed schemes and past work.  A majority of the adaptive schemes in
the literature employ past information to update the steplength. One
such avenue involves developing estimates of the Hessian which is
subsequently used in scaling the gradient step appropriately. In the
sections to appear, we consider two very different approaches that are
linked by a crucial property: they rely on using algorithm and problem
parameters, and not sample points, to develop adaptive steplength
schemes.

\subsection{Smoothing Techniques}\label{sec:smooth}
One of the goals of this paper is to address stochastic optimization
	problems with nonsmooth integrands. Here, we provide some background
for accommodating nonsmoothness in optimization problems.  In
	deterministic regimes, subgradient methods and their incremental
		variants have proved popular
		(see~\cite{Nedic01,Nedic01a,bertsekas73descent}), as have 
bundle
		 methods~\cite{kiwiel}, amongst others. One approach for
		 managing nonsmoothness is through smoothing approaches.  For
		 instance, such avenues have allowed for the solution of
		 variational inequalities and complementarity
		 problems~\cite{facchinei02finite} as well as mathematical
		 programs with equilibrium
		 constraints~\cite{luo96mathematical}.
	
In this paper, we also adopt a smoothing technique which bears
little similarity to such approaches. We adopt a framework that can be
traced back to a class of {\em averaged} functions introduced by
Steklov~\cite{steklov1,steklov2} in 1907.  A general definition of such
an averaging over possibly discontinuous functions is provided
next~\cite{norkin93optimization}.  \begin{definition} Given a locally
integrable function $f:\Real^n\to \Real$ and a family of mollifiers
$\{p_{\epsilon}:\Real^n \to \Real_+, \epsilon > 0\}$ that satisfy $$
\int_{\Real^n} p_{\epsilon}(z) dz = 1, \qquad
\textrm{supp}(p_{\epsilon}) := \{
z \in \Real^n: p_{\epsilon}(z) > 0\} \subset \rho_{\epsilon} \mathbb{B}
\mbox{ with } \rho_{\epsilon} \downarrow 0 \hbox{ as } \epsilon \downarrow
0, $$ where $\mathbb{B}$ is a unit ball in $\Real^n$. Then  the
associated family $\{ \hat{f}_{\epsilon}, \epsilon > 0\}$ of averaged
functions  is defined by $$ \hat{f}_{\epsilon} := \int_{\Real^n} f(x+z)
p_{\epsilon} (z) dz = \int_{\Real^n} f(z) p_{\epsilon} (x-z) dz. $$
\end{definition} In effect, the {\em mollifier} is a probability density
function and the family of smoothed approximations, denoted by
$\{\hat{f}_{\epsilon},\epsilon>0\}$ must possess a host of convergence 
properties with respect to $f$ as $\epsilon \to 0$. For instance, if $f$ is a
continuous function then $\hat{f}_{\epsilon}$ converges uniformly to $f$
on every bounded subset of $\Real^n$. In the absence of continuity, this
cannot be guaranteed; yet, we may draw on epi-convergence
results~\cite{rockafellar98variational} for this class of functions may
be employed in an effort to establish convergence of the infima/minima.
These averaging functions have allowed for solving convex
nondifferentiable optimization
problems~\cite{ermoliev73limit,gaivoronski78nonstationary} and
discontinuous optimization problems~\cite{gupal77algorithm}, by
minimizing a sequence of averaged or smoothed functions. 

We pursue an alternative to solving a sequence of smoothed problems and
obtain an approximate solution by solving a single smoothed problem with
a fixed $\epsilon$ akin to that employed by Lakshmanan and
Farias~\cite{DeFarias08}. However, since we intend to leverage
stochastic approximation schemes of the form described earlier in this
paper, Lipschitz constants associated with the gradients are a requiem.
In~\cite{DeFarias08}, the authors obtain Lipschitz constants assuming
that the averaging is achieved through a normal distribution that
requires the function be defined everywhere. Instead of ``globally
smoothing'' the function, we employ a uniform distribution, referred to
as ``local smoothing.''

\section{A recursive steplength stochastic approximation scheme}
\label{sec:adap}
In this section, we introduce an adaptive stochastic approximation
scheme that overcomes certain challenges associated with implementing
standard diminishing steplength schemes and relies on the use of a
recursive rule for prescribing steplengths. We begin by examining the
standard stochastic gradient method for problem~\eqref{eqn:problem} in
Section~\ref{sec:preliminaries}.  In general, the convergence of this
scheme is guaranteed under the requirement that $\sum_{k = 1}^{\infty}
\gamma_k = \infty$ and $\sum_{k=0}^\infty \gamma_k^2 < \infty.$ A host
of choices exists with one possible choice being $\gamma_k = \theta/k.$
Yet, the choice of the appropriate $\theta$ can have a significant
impact on the performance of the algorithm.  Motivated by the desire to
minimize the ``expected error,'' we develop a recursive stochastic
approximation algorithm (referred to as the RSA scheme) in which the
steplength at a particular iteration is a function of the steplength at
the previous iteration and some problem parameters.  In
Section~\ref{sec:adaptive-self}, we motivate and introduce such a scheme
and proceed to develop the associated convergence theory in
Section~\ref{sec:convergence-self}.

\subsection{Preliminaries}~\label{sec:preliminaries}
We consider method~\eqref{eqn:algorithm} as applied to
problem~\eqref{eqn:problem} where $f$ has Lipschitz gradients.  The
method generates a sequence of iterates that converge to an optimal
solution almost-surely, as shown in the forthcoming proposition.  This
result is a straightforward extension of Theorem 1
in~\cite[Pg.~51]{Polyak87} which pertains to an unconstrained problem.

\begin{proposition}[Almost-sure convergence] \label{prop:lipschitzgrad}
Let Assumptions~\ref{assum:convex}--\ref{assum:step_error} hold, and
let $f$ be differentiable over the set $X$ with Lipschitz gradients.
Assume that the optimal set $X^*$ of problem~\eqref{eqn:problem} is nonempty. 
Then, the sequence $\{x_k\}$ generated by \eqref{eqn:algorithm} 
converges almost surely to some random point in $X^*$.
\end{proposition}
\begin{proof}
By definition of the method and the nonexpansive property of the
projection operation, we obtain for any $x^*\in X^*$ and $k\ge0$,
\begin{align*}
\|x_{k+1}-x^*\|^2 & \leq \|x_k-x^*-\g_k(\nabla f(x_k)+w_k)\|^2 \cr
&= \|x_k-x^*\|^2-2\g_k(\nabla f(x_k)+w_k)^T(x_k-x^*)+\g_k^2\|\nabla
f(x_k)+w_k\|^2.
\end{align*}
By the convexity of $f$ and the gradient inequality, we have
\begin{align*}
\|x_{k+1}-x^*\|^2 & \leq  \|x_k-x^*\|^2-2\g_k(f(x_k)-f(x^*))-2\g_k
w_k^T(x_k-x^*) +\g_k^2\|\nabla f(x_k)+w_k\|^2.
\end{align*}
Since $\|a+b\|^2 \le 2\|a\|^2 +2\|b\|^2$ for any $a,b \in
\mathbb{R}^n$, by using $f^*=f(x^*)$, and by adding and subtracting
$\nabla f(x^*)$ in the last term, we obtain
\begin{align*}
\|x_{k+1}-x^*\|^2 & \leq  \|x_k-x^*\|^2-2\g_k(f(x_k)-f^*)-2\g_k
w_k^T(x_k-x^*) +2\g_k^2\|\nabla f(x_k)-\nabla f(x^*)\|^2 \cr
&+2\g_k^2\|\nabla f(x^*)+w_k\|^2.
\end{align*}
Taking the conditional expectation given $\sF_k$, using
$\EXP{w_k\mid \sF_k}=0$ (see Eq.~\eqref{eqn:zeromean}) and 
the Lipschitzian property of the gradient, we have
\begin{align*}
\EXP{\|x_{k+1}-x^*\|^2 \mid \sF_k}& \leq
(1+2L^2\g_k^2)\|x_k-x^*\|^2-2\g_k(f(x_k)-f^*)\cr
&+2\g_k^2\left(\|\nabla f(x^*)\|^2+\EXP{\|w_k\|^2\mid \sF_k}\right).
\end{align*}

Under Assumption~\ref{assum:step_error}, the conditions of
Lemma~\ref{lemma:supermartingale} are satisfied. Therefore, almost
surely, the  sequence $\{\|x_{k+1}-x^*\|\}$ is convergent for any
$x^*\in X^*$ and $\sum_{k=0}^\infty\g_k(f(x_k)-f^*)<\infty$. The
former relation implies that $\{x_k\}$ is bounded a.s., while the
latter implies $\liminf_{k\to\infty} f(x_k)=f^*$ a.s.\ in view of the condition
$\sum_{k=0}^\infty \g_k=\infty$.  Since the set $X$ is closed, 
all accumulation points of $\{x_k\}$ lie in $X$. Furthermore, 
since $f(x_k)\to f^*$ along a subsequence a.s., by continuity of $f$ 
it follows that $\{x_k\}$ has a subsequence converging to some random point 
in $X^*$ a.s. Moreover, since $\{\|x_{k+1}-x^*\|\}$ is convergent 
for any $x^*\in X^*$ a.s.,
the entire sequence $\{x_k\}$ converges to some random point in $X^*$ a.s.
\end{proof}

Under the Lipschitz continuity of the gradient and the strong convexity
of the objective, an expected error bound may also be provided for the
method.  During the development of the error bound, the following
intermediate result assumes relevance.

\begin{lemma}\label{lemma:Lipschitz_key}
Let Assumption \ref{assum:convex} hold, and let $f$ be differentiable
over the set $X$ with Lipschitz gradients with constant $L>0$. Also,
assume that the optimal set $X^*$ of problem~\eqref{eqn:problem} is
nonempty. Let the sequence $\{x_k\}$ be generated by algorithm
\eqref{eqn:algorithm} with any (deterministic) stepsize $\g_k>0$. Then,
for any $x^*\in X^*$ and any $k\ge0$, the following holds almost surely:
\begin{align*}
\EXP{\|x_{k+1}-x^*\|^2\mid\sF_k} \leq
 \|x_{k}-x^*\|^2 +\g_k^2 \EXP{\|w_k\|^2\mid \sF_k}
-\g_k(2-\g_k L) (x_k-x^*)^T(\nabla f(x_k)-\nabla f(x^*)).
\end{align*}
\end{lemma}
\begin{proof} 
By the first-order optimality conditions, a vector $x^*$ is
optimal for the problem if and only if $x^*$ satisfies
\[x^*=\Pi_X(x^*-\g \nabla f(x^*))\qquad\hbox{for any }\g>0.\]
By the definition of the method and the nonexpansive property of the
projection operation, we obtain for all $k\ge0$,
\begin{equation*}
\begin{split}
\|x_{k+1}-x^*\|^2 
&=\|\Pi_X(x_k-\g_k(\nabla f(x_k)+w_k))-\Pi_X(x^*-\g_k \nabla f(x^*))\|^2 \\
&\le \|x_k-x^*-\g_k(\nabla f(x_k)+w_k-\nabla f(x^*))\|^2.
\end{split}
\end{equation*}
Taking the expectation conditioned on the past, and using
$\EXP{w_k\mid \sF_k}=0$ (cf.\ Eq.\ \eqref{eqn:zeromean}), we have
\begin{equation*}
\begin{split}
\EXP{\|x_{k+1}-x^*\|^2\mid\sF_k} &\le \|x_{k}-x^*\|^2 +
\gamma_k^2\|\nabla f(x_k)-\nabla f(x^*)\|^2 +\g_k^2
\EXP{\|w_k\|^2\mid \sF_k}\cr &-2\gamma_k (x_k-x^*)^T(\nabla
f(x_k)-\nabla f(x^*)).
\end{split}
\end{equation*}
The Lipschitz gradient property for a convex function is equivalent
to co-coercivity of the gradient map with constant $1/L$, {(see
\cite[Pg.~24, Lemma 2]{Polyak87})}, 
i.e., for all $x,y\in X$,
\begin{equation*}\label{co-coercivity}
\frac{1}{L}\,\|\nabla f(x)-\nabla f(y)\|^2 \le (x-y)^T(\nabla
f(x)-\nabla f(y)).
\end{equation*}
Therefore, for any $x^*\in X^*$ and any $k\ge 0$,
\[\EXP{\|x_{k+1}-x^*\|^2\mid\sF_k} 
\leq \|x_{k}-x^*\|^2 +\g_k^2 \EXP{\|w_k\|^2\mid \sF_k} 
-\g_k (2-\g_k L) (x_k-x^*)^T(\nabla f(x_k)-\nabla f(x^*)).\]
\end{proof}

{In what follows, we will often use a stronger version of 
Assumption~\ref{assum:step_error}(b), given as follows.

\begin{assumption}\label{assum:bounded_error}
The errors $w_k$ are such that 
for some $\nu >0$,
\[\EXP{\|w_k\|^2| \sF_k} \le \nu^2 \quad a.s.\ \hbox{for all $k\ge0$}.\]
\end{assumption}

Next, we  provide an error bound for algorithm~\eqref{eqn:algorithm}
under the assumption  that $f(x)$ is a strongly convex function with
	Lipschitz gradients. 
Note that requiring that $f(x)$ is strongly
	convex over a set $K$ follows if $F(x,\xi(\omega))$ is a strongly
	convex function for $\omega \in \bar{\Omega}$, where $\bar{\Omega}$
	is a set of positive measure defined as 
$$\bar{\Omega} \triangleq \left\{\omega: \exists \eta > 0,
	(y-x)^T(\nabla F(y,\xi(\omega)) -\nabla F(x,\xi(\omega))) \geq
	\eta\|x-y\|^2 \quad
	\hbox{for all }x,y \in K\right\}.$$
Less formally, we merely require that $F(.,\xi)$ is a strongly
convex function with positive, but arbitrarily small, probability to
ensure that $f(x)$ is strongly convex over $K$
(see~\cite{ravat09characterization}).  

\begin{lemma}[Strongly convex function with Lipschitz gradients]
\label{lemma:rel_bound}
Let Assumptions~\ref{assum:convex}--\ref{assum:step_error} hold. Also,
let $f$ be differentiable over the set $X$ with Lipschitz gradients
with constant $L>0$ and strongly convex with constant $\eta >0$. 
Then, the sequence $\{x_k\}$
generated by algorithm (\ref{eqn:algorithm}) converges almost surely
to the unique optimal solution of problem (\ref{eqn:problem}).
Furthermore, if the stepsize satisfies $0 <\gamma_k \le \frac{2}{L}$ 
for all $k \ge0$, we then have:
\begin{itemize}
\item[(a)]
The following relation holds almost surely:
\[\EXP{\|x_{k+1}-x^*\|^2\mid\sF_k} \leq
 \left(1-\g_k(2-\g_k L)\right) 
\|x_{k}-x^*\|^2 +\g_k^2 \EXP{\|w_k\|^2\mid \sF_k}\qquad\hbox{for all }k\ge0.\]
\item[(b)] If Assumption~\ref{assum:step_error}(b) is replaced 
with Assumption~\ref{assum:bounded_error}, 
then the following relation holds almost surely:
\[\EXP{\|x_{k+1}-x^*\|^2} \le (1-\eta \g_k (2-\g_k
L))\EXP{\|x_k-x^*\|^2}+ \g_k^2\nu^2 \quad \hbox {for all }  k \geq 0.\]
Moreover,
$\lim_{k\to\infty}\EXP{\|x_k-x^*\|^2}= 0$, 
and for every $\epsilon >0$, 
\[\hbox{Prob}\left(\|x_j-x^*\|^2 \leq \e \hbox{ for all } j \geq k\right) 
\geq 1- \frac{1}{\e}\,
\left(\EXP{\|x_k-x^*\|^2}+\nu^2\sum_{i=k}^\infty \g_i^2\right)
\qquad\hbox{for all $k>0$}.\]
\end{itemize}
\end{lemma}

\begin{proof}
The existence and uniqueness of the optimal solution of problem
(\ref{eqn:problem}) is guaranteed by the strong convexity of $f(x)$. 
The convergence of the method follows by Proposition~\ref{prop:lipschitzgrad}.
To establish the relation in part (a) for the expected value 
$\EXP{\|x_{k+1}-x^*\|^2}$, we use Lemma~\ref{lemma:Lipschitz_key},
which implies for the optimal $x^*$ and all $k\ge0$,
\begin{align*}
\EXP{\|x_{k+1}-x^*\|^2\mid\sF_k} \leq
 \|x_{k}-x^*\|^2 +\g_k^2 \EXP{\|w_k\|^2\mid \sF_k}
-\g_k(2-\g_k L) (x_k-x^*)^T(\nabla f(x_k)-\nabla f(x^*)).
\end{align*}
By the strong convexity of $f(x)$, we have   
$(x-y)^T(\nabla f(x)-\nabla f(y)) \ge \eta\|x-y\|^2$ for all $x,y \in X$,
which when combined with the preceding relation implies for all $k\ge0$,
\begin{align}\label{eqn:oneo}
\EXP{\|x_{k+1}-x^*\|^2\mid\sF_k} \leq
 \left(1-\g_k\eta (2-\g_k L)\right) 
\|x_{k}-x^*\|^2 +\g_k^2 \EXP{\|w_k\|^2\mid \sF_k},
\end{align}
thus showing the relation in part (a).

The relation in part (b), follows from inequality~\eqref{eqn:oneo} by
using Assumption~\ref{assum:bounded_error} and by
taking the total expectation. 
To show the other results in part (b), we apply 
Lemma~\ref{lemma:probabilistic_bound_polyak}. For this, we need 
to verify that all the conditions of 
Lemma~\ref{lemma:probabilistic_bound_polyak} hold. 
Since $0 <\gamma_k \le \frac{2}{L}$, it follows $\eta \g_k (2-\g_k L)\ge0$. 
Also, in view of $\eta \leq L$, we have $\eta \g_k (2-\g_k L) \leq
1$.
Obviously, $\nu^2\g_k^2 \geq 0$ for all $k$. Since
Assumption~\ref{assum:step_error}(a) holds, we have
$\sum_{k=0}^\infty\eta \g_k (2-\g_k L)=\infty$ and
$\sum_{k=0}^\infty \eta \g_k^2<\infty$. Furthermore, since $\g_k\to0$,
we have
\[\lim_{k\to\infty} \frac{\nu^2\g_k^2}{\eta \g_k (2-\g_k L)} 
=\lim_{k\to\infty}\frac{\nu^2\g_k}{\eta(2-\g_k L)}= 0.\] 
Hence, the conditions of Lemma~\ref{lemma:probabilistic_bound_polyak} hold 
and the stated results follow.
\end{proof}

Lemma~\ref{lemma:rel_bound} will be employed in developing our adaptive
stepsize schemes. Before proceeding, we make the following comment
regarding the lemma. 

{\em Remark 1}: The result in part (a) of Lemma~\ref{lemma:rel_bound}
is similar to a result in~\cite{Nemirovski09},
which was derived by requiring  only the strong convexity of the
function $f$.  Here, we make the additional assumption that the
gradients are Lipschitz continuous and this assumption gains relevance
when we employ local random smoothing in Section~\ref{sec:localrand}.
Furthermore, our result depends on the expectation of gradient errors,
$\EXP{\|w_k\|^2}$, with $w_k$ defined in~\eqref{eqn:algorithm}.  Note
that, in contrast, the result in~\cite{Nemirovski09} depends on the
expectation of the subgradient norms, $\EXP{\|G(x,\xi)\|^2}$, where
$G(x,\xi)\in\partial_x F(x,\xi)$.

\subsection{A recursive steplength scheme}~\label{sec:adaptive-self}
A challenge associated with the implementation of diminishing steplength
schemes lies in determining an appropriate sequence $\{\g_k\}$. 
The key result of this section is the motivation and introduction of a scheme 
that {\em adaptively} optimizes the steplength from iteration to iteration.  
Our adaptive scheme relies on the minimization of a suitably defined error 
function at each step. We start with the relation in part (b) of  
Lemma~\ref{lemma:rel_bound}: 
\begin{equation}\label{rate_nu}
\EXP{\|x_{k+1}-x^*\|^2} \le (1-\eta \g_k (2-\g_k
L))\EXP{\|x_k-x^*\|^2}+ \g_k^2\nu^2 \quad \hbox {for all }  k \geq 0.
\end{equation}
When the stepsize is further restricted so that $0 < \g_k
\le\frac{1}{L}$, we have
\[1-\eta \g_k (2-\g_k L) \leq 1-\eta  \g_k.\]
Thus, for  $0 < \g_k \le\frac{1}{L}$, inequality (\ref{rate_nu}) yields
 \begin{equation}\label{rate_nu_est}
\EXP{\|x_{k+1}-x^*\|^2} 
\le (1-\eta \g_k)\EXP{\|x_k-x^*\|^2}+ \g_k^2\nu^2\qquad\hbox{for all }k \geq 0.
\end{equation}

We now use relation~\eqref{rate_nu_est} to develop an adaptive stepsize
procedure. Loosely speaking for the moment, let us view the quantity
$\EXP{\|x_{k+1}-x^*\|^2}$ as an error $e_{k+1}$ of the method arising
from the use of the stepsize values $\g_0,\g_1,\ldots, \g_k$. Also,
	 consider the worst case error which is the case when
	 ~\eqref{rate_nu_est} holds with equality.  Thus, in the worst case,
	 the error satisfies the following recursive relation:
\[e_{k+1}(\g_0,\ldots,\g_k) = (1-\eta \gamma_k) e_k(\g_0,\ldots,\g_{k-1}) 
+ \gamma_k^2 \nu^2.\]
Then, it seems natural to investigate if the stepsizes
$\g_0,\g_1\ldots,\g_k$ can be selected so as to minimize the error
$e_{k+1}$. It turns out that this can indeed be achieved and minimizing
the error $e_{k+2}$ at the next iteration can also be done by selecting
$\g_{k+1}$ as a function of only the most recent stepsize $\g_k$.
To formalize the above discussion, 
we define real-valued error functions $e_k(\g_0,\ldots,\g_{k-1})$ as follows:
\begin{align}\label{def:e_k}
  e_{k}(\g_0,\ldots,\g_{k-1})
& \triangleq (1-\eta \g_{k-1})e_{k-1}(\g_0,\ldots,\g_{k-2}) +\g_{k-1}^2\nu^2
\qquad\hbox{for $k\ge 1$},
\end{align}
where $e_0$ is a positive scalar, $\eta$ is the strong convexity parameter 
and $\nu^2$ is the upper bound for the second moments of the error 
norms $\|w_k\|$.

In what follows, we consider the sequence $\{\g^*_k\}$ given by 
\begin{align}
& \gamma_0^*=\frac{\eta}{2\nu^2}\,e_0\label{eqn:optimal_self_adaptive_2}
\\
& \gamma_{k}^*=\gamma_{k-1}^*\left(1-\frac{\eta}{2}\gamma_{k-1}^*\right) 
\quad \hbox{for all }k\ge 1.\label{eqn:optimal_self_adaptive}
\end{align}
We often abbreviate $e_k(\g_0,\ldots,\g_{k-1})$ by $e_k$
whenever this is unambiguous.  We show that the stepsizes $\g_i$,
$i=0,\ldots,k-1,$ minimize the errors
$e_k$ over an $(0,\frac{1}{L}]^{k}$, where $L$ is the Lipschitz constant.  In
particular, we have the following result.

\begin{proposition}\label{prop:rec_results}
Let $e_k(\g_0,\ldots,\g_{k-1})$ be defined as in~\eqref{def:e_k},
where $e_0>0$ is such that $\frac{\eta}{2\nu^2}\,e_0\leq \frac{1}{L}$,
with $L$ being the Lipschitz constant for the gradients of $f$.
Let the sequence $\{\g^*_k\}$ be given by 
\eqref{eqn:optimal_self_adaptive_2}--\eqref{eqn:optimal_self_adaptive}.
Then, the following hold:
\begin{itemize}
\item [(a)] The error $e_k$  satisfies
\begin{align*}
e_k(\g_0^*,\ldots,\g_{k-1}^*) = \frac{2\nu^2}{\eta}\,\g_k^*
\qquad\hbox{for all $k\ge0$}.
\end{align*}
\item [(b)] For each $k\ge 1$, the vector 
$(\gamma_0^*, \gamma_1^*,\ldots,\gamma_{k-1}^*)$ 
is the minimizer of the function $e_k(\g_0,\ldots,\g_{k-1})$ over the set
$$\mathbb{G}_k 
\triangleq \left\{\alpha \in \Real^k : 0 < \alpha_j \leq \frac{1}{L}
\hbox{ for }j =1,\ldots, k\right\}.$$
More precisely, for any $k\ge 1$ and any 
$(\g_0,\ldots,\g_{k-1})\in \mathbb{G}_k$, we have
	\begin{align*}
	e_{k}(\g_0,\ldots,\g_{k-1}) 
         - e_{k}(\g_0^*,\ldots,\g_{k-1}^*) \geq \nu^2(\g_{k-1}-\g_{k-1}^*)^2.
	\end{align*}	
\item[(c)] 
The vector $\gamma^*=(\gamma_0^*, \gamma_1^*,\ldots,\gamma_{k-1}^*)$ 
is a stationary point of function $e_k(\gamma_0,
			\gamma_1,\ldots,\gamma_{k-1})$ over the set 
$\mathbb{G}_k$.	
\end{itemize}
\end{proposition} 
\begin{proof}
(a) \ We use induction on $k$ to prove our result. 
Note that the result holds trivially for 
$k=0$ from \eqref{eqn:optimal_self_adaptive_2}.  Next, assume that we have
$e_k(\g_0^*,\ldots,\g_{k-1}^*) = \frac{2\nu^2}{\eta}\,\g_k^*$ for some $k$,
and consider the case for $k+1$. By the definition of the error $e_k$ 
in~\eqref{def:e_k}, we have
\[e_{k+1}(\g_0^*,\ldots,\g_{k}^*) = (1-\eta\g^*_{k})
e_{k}(\g^*_0,\ldots,\g_{k-1}^*)+\g_k^*\nu^2
=(1-\eta\g^*_{k})\frac{2\nu^2}{\eta}\,\g_k^*+\g_k^*\nu^2,\]
where the second equality follows by the inductive hypothesis.
Hence,
\[e_{k+1}(\g_0^*,\ldots,\g_{k}^*) =\frac{2\nu^2}{\eta}\,\g_k^*
\left(1-\eta \g^*_k+\frac{\eta}{2}\,\g^*_k\right)
=\frac{2\nu^2}{\eta}\,\g_k^*\left(1-\frac{\eta}{2}\,\g^*_k\right)
=\frac{2\nu^2}{\eta}\,\g_{k+1}^*,\]
where the last equality follows by the definition of $\g^*_{k+1}$ in
\eqref{eqn:optimal_self_adaptive}.

\noindent (b) \ We now show that $(\gamma^*_0,\g_1^*,\ldots,\g_{k-1}^*)$  
minimizes the error $e_k$ for all $k\ge1$. 
We again use mathematical induction on $k$. 
By the definition of the error $e_1$ and  the relation
$e_1(\g_0^*)=\frac{2\nu^2}{\eta}\g_1^*$ shown in part (a), we have
\[e_1(\g_0) -e_1(\g_0^*) =(1-\eta \g_0)e_0+\nu^2 \g_0^2
-\frac{2\nu^2}{\eta}\g_1^*.\]
Using $\g_1^*=\g_0^*\left(1-\frac{\eta}{2}\g_0^*\right)$, we obtain
\[e_1(\g_0) -e_1(\g_0^*) =(1-\eta \g_0)e_0+\nu^2 \g_0^2
-\frac{2\nu^2}{\eta}\g_0^*\left(1-\frac{\eta}{2}\g_0^*\right)
=(1-\eta \g_0)\frac{2\nu^2}{\eta}\,\g_0^*+\nu^2 \g_0^2
-\frac{2\nu^2}{\eta}\g_0^*\left(1-\frac{\eta}{2}\g_0^*\right),\]
where the last equality follows from $e_0=\frac{2\nu^2}{\eta}\,\g_0^*.$
Thus, we have
\[e_1(\g_0) -e_1(\g_0^*)=-2\nu^2\g_0\g_0^*+\nu^2 \g_0^2
+\nu^2\,(\g_0^*)^2=\nu^2\left(\g_0-\g_0^*\right)^2.\]

Now suppose that $e_k(\g_0,\ldots,\g_{k-1}) 
\geq e_k(\g_0^*,\ldots,\g_{k-1}^*)$ holds for some $k$ and
any $(\g_0,\ldots,\g_{k-1}) \in\mathbb{G}_k$. We want to show that 
$e_{k+1}(\g_0,\ldots,\g_{k}) \geq e_{k+1}(\g_0^*,\ldots,\g_{k}^*)$ 
holds as well for all $(\g_0,\ldots,\g_{k}) \in\mathbb{G}_{k+1}$. 
To simplify the notation
we use $e_{k+1}^*$ to denote the error $e_{k+1}$ evaluated at 
$(\g^*_0,\g_1^*,\ldots,\g_{k}^*)$, and $e_{k+1}$ when 
evaluating at an arbitrary vector 
$(\g_0,\g_1,\ldots,\g_{k})\in\mathbb{G}_{k+1}$.
Using (\ref{def:e_k})  and part (a), we have
\begin{align*}
e_{k+1}- e_{k+1}^*= (1-\eta \gamma_{k})e_k
+\nu^2\gamma_{k}^2 -\frac{2\nu^2}{\eta} \gamma_{k+1}^*.
\end{align*}
Under the inductive hypothesis, we have $e_k\ge e_k^*$.
Using this, the relation $e_k^*=\frac{2\nu^2}{\eta}\g_k^*$
of part (a) and  the definition of $\g_{k+1}^*$, we obtain
\begin{align*}
e_{k+1} - e_{k+1}^* 
 \geq (1-\eta \gamma_{k})\frac{2\nu^2}{\eta}\g_k^*+\nu^2\gamma_{k}^2 
-\frac{2\nu^2}{\eta} \gamma_{k}^*\left(1-\frac{\eta}{2}\g_k^*\right) 
 = \nu^2(\g_k-\g_k^*)^2.
\end{align*}
Hence, we have $e_k(\g_0,\ldots,\g_{k-1})-e_k(\g^*_0,\ldots,\g_{k-1}^*)\ge
\nu^2(\g_k-\g_k^*)^2$ for all $k\ge1$ and all 
$(\g_0,\ldots,\g_{k-1})\in\mathbb{G}_k$. Therefore, for all $k\ge1$, the vector
$(\g_0,\ldots,\g_{k-1})\in\mathbb{G}_k$ is a minimizer of the error $e_k$.

\noindent (c) \ By the choice of $e_0$, we have $0<\g_0^*< \frac{1}{L}$.
Observe that since $\eta\le L$, it follows that $0< \g^*_1\le\g^*_0$, 
and by induction we can see that $0< \g^*_k\le\g^*_{k-1}$ for all $k\ge1$. 
Thus, $(\g^*_0,\ldots,\g_{k-1}^*)\in\mathbb{G}_k$ for all $k\ge1$. 

Now, we proceed by induction on
$k$.  For $k=1$, we have 
\[\frac{\partial e_1}{\partial \gamma_0}=-\eta e_0+2\g_0\nu^2.\]
Thus, the derivative of $e_1$ vanishes at $\g_0^*=\frac{\eta}{2\nu}\, e_0$,
which satisfies $0<\g_0^*\leq \frac{1}{L}$ by the choice of $e_0$.
Furthermore, note that the function $e_1(\g_0)$ is convex in $\g_0$.
Hence, $\g_0^*=\frac{\eta}{2\nu}\, e_0$ is the stationary point of $e_1$ 
over the entire real line. Suppose now that for $k\ge1$, the vector 
$(\g_0,\ldots,\g_{k-1})$ is the minimizer of $e_k$ over the set $G_k$. 
Let us now consider the case of $k+1$. The partial derivative of $e_{k+1}$ 
with respect to $\g_\ell$ is given by 
\begin{align*}
\frac{\partial e_{k+1}}{\partial \gamma_\ell}
& =-\eta e_0\prod_{i=0,i\neq
	\ell}^{k}(1-\eta \gamma_i)-\eta \nu^2
	\sum_{i=0}^{\ell-1}
\left(\gamma_i^2\prod_{j=i+1, j\neq \ell}^{k}(1-\eta
				\gamma_j)\right)+2\nu^2\gamma_{\ell}
\prod_{i=\ell+1}^{k}(1-\eta \gamma_i),  
\end{align*}
where $0\le \ell\le k-1.$ By factoring out the common term 
$\prod_{i=\ell+1}^{k}(1-\eta\gamma_i)$, 
we obtain
\begin{align}\label{eqn:one}
\frac{\partial e_{k+1}}{\partial \gamma_\ell}  
=\left[-\eta \left( e_0 \prod_{i=0}^{\ell-1}(1-\eta \gamma_i) + \nu^2
\sum_{i=0}^{\ell-2}\left(\gamma_i^2\prod_{j=i+1}^{\ell-1}(1-\eta
\gamma_j)\right) +\nu^2\gamma_{\ell-1}^2 \right)
+2\nu^2\gamma_{\ell}\right]\prod_{i=\ell+1}^{k}(1-\eta\gamma_i).
\end{align}
From the  definition of $e_k$ in~\eqref{def:e_k} we can see that  
\begin{align}\label{eqn:recursive}
e_{k+1}(\g_0,\ldots,\g_{k}) = e_0 \prod_{i=0}^{k}(1-\eta \gamma_i)+\nu^2
\sum_{i=0}^{k-1}\left(\gamma_i^2\prod_{j=i+1}^{k}(1-\eta
	\gamma_j)\right)+\nu^2\gamma_{k}^2.
\end{align}
By combining relations~\eqref{eqn:one} and~\eqref{eqn:recursive}, we obtain
for all $\ell=0,\ldots,k-1,$
\[\frac{\partial e_{k+1}}{\partial \gamma_\ell}  
=\left(-\eta e_\ell(\g_0,\ldots,\g_{\ell-1}\right) 
+2\nu^2\gamma_{\ell})\prod_{i=\ell+1}^{k}(1-\eta\gamma_i),\]
where for $\ell=0$, we have $e_\ell(\g_0,\ldots,\g_{\ell-1})=e_0$.
By part (a), there holds $-\eta e_\ell(\g_0^*,\ldots,\g_{\ell-1}^*)
+2\nu^2\gamma_{\ell}^*=0$, thus showing that 
$\frac{\partial e_{k+1}}{\partial \gamma_\ell}$ vanishes at
$(\g^*_0,\ldots,\g_k^*)\in\mathbb{G}_{k+1}$ for all $\ell=0,\ldots,k-1$.

Finally, we consider the partial derivative of $e_{k+1}$ 
with respect to $\g_{k}$, for which we have 
\begin{align*}
\frac{\partial e_{k+1}}{\partial \gamma_k} 
& =-\eta e_0\prod_{i=0}^{k-1}(1-\eta \gamma_i)-\eta \nu^2 
\sum_{i=0}^{k-2}\left(\gamma_i^2\prod_{j=i+1 }^{k-1}(1-\eta \gamma_j)\right)
-\eta \nu^2\gamma_{k-1}^2+2\nu^2\gamma_{k}. 
\end{align*}
Using relation~\eqref{eqn:recursive}, we obtain
\[\frac{\partial e_{k+1}}{\partial \gamma_{k}} 
=-\eta e_{k}(\g_0,\ldots,\g_{k-1})+2\nu^2\gamma_{k}.\]
By part (a), we have 
$-\eta e_k(\g_0^*,\ldots,\g_{k-1}^*)+2\nu^2\gamma_{k}^*=0$, thus showing that 
$\frac{\partial e_{k+1}}{\partial \gamma_k}$ vanishes at
$(\g^*_0,\ldots,\g_k^*)\in\mathbb{G}_{k+1}$. 
Thus, by induction we have that 
$(\g^*_0,\ldots,\g_k^*)$ is a stationary point
of $e_{k+1}$ in the set $\mathbb{G}_{k+1}$.
\end{proof}

We observe that in Proposition~\us{\ref{prop:rec_results}},
the minimizer $(\g_0^*,\ldots,\g_{k-1}^*)$ of the function $e_k$
over the set $\mathbb{G}_k$ is unique up to scaling  by a factor $\beta<1$.
Specifically, the solution $(\g_0^*,\ldots,\g_{k-1}^*)$ is obtained
for an initial error $e_0>0$ satisfying $e_0<\frac{2\nu^2}{\eta L}.$ 
Suppose that in the definition of the sequence $\{\g_k^*\}$ instead of
$e_0$ we use $\beta e_0$ for some $\beta\in(0,1)$. Then it can be seen
(by following the proof) that, for the resulting sequence,
Proposition~\ref{prop:rec_results} would still hold.

\subsection{Convergence theory}~\label{sec:convergence-self}
We next show that the proposed RSA approximation scheme  discussed in
Section~\ref{sec:adaptive-self} leads to a convergent algorithm. We
prove this in a more general setting for a stepsize with a form similar
to that seen in constructing the optimal choice.  
The following proposition holds for any stepsize of a form  similar to 
the optimal scheme of~\eqref{eqn:optimal_self_adaptive}. 

\begin{proposition}[Global convergence of RSA scheme]
\label{prop:self_adaptive}\label{prop:rec_convergence}
Let Assumptions~\ref{assum:convex} and~\ref{assum:bounded_error} hold.
Let the function $f$ be differentiable over the set $X$ with Lipschitz
gradients and the optimal solution set of problem~(\ref{eqn:problem}) be
nonempty.  Assume that the stepsize sequence $\{\g_k\}$ is generated by
the following self-adaptive scheme:
\begin{align}\label{eqn:self_adaptive}
\gamma_{k}=\gamma_{k-1}(1-c\gamma_{k-1}) 
\qquad \hbox{for all }k \geq 1,
\end{align}
where $c>0$ is a scalar and the initial stepsize is such that
$0<\g_0<\frac{1}{c}$. Then, the sequence $\{x_k\}$ generated by
algorithm (\ref{eqn:algorithm}) converges almost surely to a random
point that belongs to the optimal set.
\end{proposition}
\begin{proof} We employ Proposition~\ref{prop:lipschitzgrad}.  
To apply this proposition, it suffices to verify that 
Assumption~\ref{assum:step_error} holds. First we show that 
$\sum_{i=0}^{\infty}{\gamma_i}=\infty$. From (\ref{eqn:self_adaptive})
we obtain
\[\prod_{\ell=1}^{k+1}\gamma_{\ell}=\left(\prod_{i=0}^{k}\gamma_i\right)
\prod_{i=0}^{k}(1-c\gamma_i).\]
By dividing both sides by $\left(\prod_{i=1}^{k}\gamma_i\right)$, it follows 
that
\begin{align}\label{eqn:rec_prod} 
\gamma_{k+1}=\gamma_0\prod_{i=0}^{k}(1-c\gamma_i).
\end{align}
Since $\g_0\in(0,\frac{1}{c})$, from (\ref{eqn:self_adaptive})
it follows that $\{\gamma_k\}$ is positive nonincreasing sequence. 
Therefore, the limit $\lim_{k\to\infty} \gamma_k$ exists and 
it is less than $\frac{1}{c}$. Thus, by taking the limit 
in~(\ref{eqn:self_adaptive}),  we obtain $\lim_{k\to\infty}\g_k =0$. 
Then, by taking limits in~(\ref{eqn:rec_prod}), we further 
obtain \[\lim_{k\rightarrow \infty}\prod_{i=0}^{k}(1-c\gamma_i) =0.\]

To arrive at a contradiction suppose that 
$\sum_{i=0}^{\infty}{\gamma_i}<\infty$. Then, there is an $\e\in (0,1)$ 
such that for $j$ sufficiently large, we have
\[c\sum_{i=j}^{k} \gamma_i\le\e \quad \hbox{for all }k \geq j.\]
Since $\prod_{i=j}^{k}(1-c\gamma_i) \geq 1 - c\sum_{i=j}^{k} \gamma_i $
for all $j<k$, by letting $k \rightarrow \infty$, we obtain
for all $j$ sufficiently large,
\[\prod_{i=j}^{\infty}(1-c\gamma_i) 
\geq 1 - c\sum_{i=j}^{\infty} \gamma_i\ge 1-\e>0.\]
This, however, contradicts the fact 
$\lim_{k\rightarrow \infty}\prod_{i=0}^{k}(1-c\gamma_i) =0.$
Therefore, we conclude that $\sum_{i=0}^{\infty}{\gamma_i}=\infty$.

Now we show that $\sum_{i=0}^{\infty}{\gamma_i}^2 < \infty$. 
From (\ref{eqn:self_adaptive}) we have
\begin{align*}
\gamma_{k}=\gamma_{k-1}-c\gamma_{k-1}^2 \quad \hbox{ for all } k \geq 1.
\end{align*}
Summing the preceding relations, we obtain
\begin{align*}
 \gamma_{k}=\gamma_0-c\sum_{i=0}^{k-1}{\gamma_i}^2 
\quad \hbox{ for all } k \geq 1.
\end{align*}
By taking limits and recalling that $\lim_{k \to \infty} \gamma_k = 0$, 
we obtain
\begin{align*}
\sum_{i=0}^{\infty}{\gamma_i}^2=\frac{\gamma_0}{c}< \infty.
\end{align*} 
Assumption \ref{assum:bounded_error} and relation
$\sum_{i=0}^{\infty}{\gamma_i}^2 < \infty$ yield
$\sum_{k=0}^\infty \gamma_k^2 \EXP{\|w_k\|^2\mid \sF_k}<\infty$. Hence, 
Assumption~\ref{assum:step_error} holds.
\end{proof}

Note that Proposition~\ref{prop:self_adaptive} applies to algorithm
(\ref{eqn:algorithm}) with the stepsize sequence $\{\g_k^*\}$ generated by 
the recursive scheme (\ref{eqn:optimal_self_adaptive}). 
Thus, we immediately have the following corollary.

\begin{corollary}[Convergence of \us{RSA} scheme]
\label{prop:rate_self_adaptive}
Let Assumptions~\ref{assum:convex} and \ref{assum:bounded_error} hold.
Let the function $f$ be differentiable over the set $X$ with Lipschitz 
gradients with constant $L>0$ and strongly convex with parameter $\eta>0$. 
Let the stepsize sequence $\{\g_k^*\}$ be  generated by the recursive scheme
(\ref{eqn:optimal_self_adaptive}) with $e_0=\EXP{\|x_0-x^*\|^2}$. 
If $\frac{\eta}{2\nu^2}\EXP{\|x_0-x^*\|^2}<\frac{1}{L},$
then the sequence $\{x_k\}$ generated by algorithm (\ref{eqn:algorithm}) 
converges almost surely to the unique optimal solution $x^*$ of 
problem~(\ref{eqn:problem}). 
\end{corollary}

\begin{proof} The existence and uniqueness of 
the optimal solution follows by the strong convexity assumption. 
Almost sure convergence follows by 
Proposition~\ref{prop:self_adaptive}. 
\end{proof}

Note that when the set $X$ is bounded, in 
Proposition~\ref{prop:rate_self_adaptive} we may use 
$e_0=\max_{x,y\in X}\|x-y\|^2$ and the results will hold as long as 
$\frac{\eta}{2\nu^2} \max_{x,y\in X}\|x-y\|^2<\frac{1}{L}.$

In the following, \an{we discuss a recursive stepsize for
algorithm~(\ref{eqn:algorithm}) as applied to} 
a nonsmooth but strongly convex function
$f(x)=\EXP{F(x,\xi)}$. Let  $G(x,\xi)$ be a subgradient vector of 
$F(x,\xi)$ with respect to $x$, i.e.,  
$G(x,\xi)\in \partial F(x,\xi)$. 
Assume that there is a positive number $M$ such that 
\[\EXP{\|G(x,\xi)\|^2}\leq M^2 \quad \hbox{for all }x \in X.\] 
\an{We have the following convergence result,
which obviously also holds for smooth problems}.

\begin{proposition}[Convergence of RSA with a nonsmooth objective]
\label{prop:recursive_scheme}
Consider problem (\ref{eqn:problem}) and let Assumption~\ref{assum:convex}
hold. Also, let the set $X$ be compact and the function $f$ be strongly convex 
over $X$ with constant $\eta$. Assume that there is a scalar $M>0$
such that $\EXP{\|G(x,\xi)\|^2}\leq M^2$ for all $x\in X$.
Consider the following algorithm:
\begin{equation}\label{eqn:algorithm_shapiro}
\begin{split}
x_{k+1} & =\Pi_{X}\left(x_k-\g_k G(x_k,\xi_k)\right),
\end{split}
\end{equation}
where $x_0\in X$ is a random initial point independent of $\{\xi_k\}$
and $\gamma_k$ is a (deterministic) stepsize. Consider the self-adaptive stepsize
sequence $\{\g_k^*\}$ defined by 
\begin{align*}
& \gamma_0^*=\frac{\eta}{M^2}D^2,\cr
& \gamma_{k}^*=\gamma_{k-1}^*(1-\eta\gamma_{k-1}^*) 
\qquad \hbox{for all }k\ge1,
\end{align*}
where $D=\max_{x,y\in X}\|x-y\|$. 
Assuming that $\frac{\eta D^2}{M^2} < \frac{1}{2}$, we have
\begin{align*}
\EXP{\|x_k-x^*\|^2} \leq \frac{M^2}{\eta}\g_k^* 
\qquad \hbox{for all }k \geq 1.
\end{align*}
\end{proposition}
\begin{proof}
The proof is based on verifying that, for the algorithm 
in~\eqref{eqn:algorithm_shapiro}, Proposition~\ref{prop:rec_results} 
holds, where $2\nu^2$ is replaced by $M^2$ and $e_0=D^2$. 
Then, the rest of the proof is similar to 
that of Proposition~\ref{prop:rate_self_adaptive}.
\end{proof}

\section{A cascading steplength stochastic approximation scheme}
\label{sec:cascading}
In Section~\ref{sec:adap}, we presented a stochastic approximation
scheme in which the sequence of steplengths is determined via a
recursion that relies on optimizing the error estimates. A key benefit
of such a recursion is that the steplength choice is not left to the
user. In this section, we introduce an alternate avenue for specifying
steplengths that also considers a diminishing steplength framework but
uses a markedly different approach for determining the steplength.  In
particular, the scheme relies on reducing the steplength at a set of
epochs while the steplengths are maintained as constant between these
epochs. The details of this stochastic approximation scheme (called the
cascading steplength stochastic approximation (CSA) scheme) are
presented in Section~\ref{casc:des} while convergence theory is provided
in Section~\ref{casc:conv}. 

\subsection{A cascading steplength scheme}\label{casc:des}
Our technique is based on the properties derived from problems
possessing strongly convex objectives. Specifically, we obtain the
following result from the inequality in Lemma \ref{lemma:rel_bound} when
the stepsize is maintained as constant.

\begin{proposition}\label{prop:casc_inequ}
Let Assumptions~\ref{assum:convex} and~\ref{assum:bounded_error} hold.
Also, let $f$ be differentiable over the set $X$ with Lipschitz
gradients with constant $L>0$ and strongly convex with constant $\eta
>0$.  Let the sequence $\{x_k\}$ be generated by 
\eqref{eqn:algorithm} with constant stepsize $\g_k=\g$ for all $k \geq
0$, where $\gamma\in (0,\frac{2}{L})$. Then, we have
\begin{equation}\label{constant-rate}
\EXP{\|x_k-x^*\|^2} \le q(\g)^k \EXP{\|x_0-x^*\|^2}+
\left(\frac{1-q(\g)^k}{1-q(\g)}\right)\g^2\nu^2,
\end{equation}
where $q(\g)=1-\eta \g (2-\g L)$ and $x^*$ is the optimal solution of 
problem~\eqref{eqn:problem}.
\end{proposition}
\begin{proof}
Follows from the {inequality in part (b)} of Lemma \ref{lemma:rel_bound}. 
\end{proof}

From inequality \eqref{constant-rate}, we obtain {the following relation}
\begin{equation}\label{equ:reform2}
\EXP{\|x_k-x^*\|^2} \le
\underbrace{q(\g)^k\EXP{\|x_0-x^*\|^2}}_{\rm Transient \,error }+
\underbrace{\frac{\g^2\nu^2}{1-q(\g)}}_{\rm Persistent \, error}
\quad \hbox{for all }k \ge 1,
\end{equation}
where the {expected distance $\EXP{\|x_k-x^*\|^2}$ }
is bounded by the sum of two error terms:
\begin{enumerate}
\item[(1)] {\em Transient error:} The transient error, given by $q(\g)^k
\EXP{\|x_0-x^*\|^2}$, decays to zero as $k \to \infty.$ In effect, the
contractive nature of {this error}, as arising from $q(\g) < 1$, ensures
that the transient error can be reduced to an arbitrarily small level. 
\item[(2)]{\em Persistent error:} The persistent error, given by
${\frac{\g^2\nu^2}{1-q(\g)}}$, is invariant to increasing the number of
iterations, denoted by $k$.  Its reduction, as we proceed to show,
necessitates reducing $\gamma.$
\end{enumerate}

Our cascading steplength scheme basically requires specifying a rule
for deciding at what iteration  to decrease the steplength and to what
extent it should be decreased. The iterations during which the stepsize
is kept fixed is referred to as a {\it constant steplength regime}} or
just a {\em regime}.  Given the two error terms, our scheme can be
loosely represented as an infinite {sequence} of regimes of finite
duration.  In fact, we proceed to show that the duration of the regimes
is an increasing function. Entering a new regime is marked by a
reduction in {the} steplength. In fact, since a finite reduction in the
steplength occurs between consecutive regimes, the steplength sequence
would naturally converge to zero if there {is} an infinite number of the
regimes.  Suppose one is at the beginning of the $t$th regime, where the
steplength  is $\g_t$ and the current iteration number is $K$. The
steplength $\gamma_t$ is maintained constant during regime $t$.
Furthermore, suppose that at the beginning of the $t$th regime, the
transient error is greater than the persistent error for $\gamma_t$,
i.e., $\EXP{\|x_K-x^*\|^2} > \frac{\g_t^2\nu^2}{1-q(\g_t)}$. 
Since $0<q(\g_t)<1$, 
$\EXP{\|x_K-x^*\|^2}$ decreases when multiplied with $q(\g_t)^k$ for
$k \geq 0$. The larger $k$, the smaller $q(\g_t)^k\EXP{\|x_K-x^*\|^2}$,
so there exists $k>0$ for which $q(\g_t)^k\EXP{\|x_K-x^*\|^2}$
will drop and remain below the persistent error 
$\frac{\g_t^2\nu^2}{1-q(\g_t)}$. 
We let $K_t$ be the index $k$ just before this drop takes place,
i.e., $K_t$ is the largest $k$ for which the following
inequality holds:  
$$ q(\g_t)^k\EXP{\|x_K-x^*\|^2} > \frac{\g_t^2\nu^2}{1-q(\g_t)}.$$
Therefore, $K_t$ specifies the duration of regime $t$,
during which the stepsize is fixed at $\g_t$.

The next question is how one should go about reducing the persistent
error. We observe through the next result that by reducing $\gamma_t$, the
persistent error does indeed reduce.
\begin{lemma}\label{lemma:decrease} 
Consider the persistent error given by $P(\g)=\frac{ \gamma^2
	\nu^2}{1-q(\gamma)}$, {where $q(\g)=1-\eta \g (2-\g L)$ and
		$\g\in(0,\frac{2}{L})$.} Then, this error is an increasing
		function of $\gamma.$
\end{lemma}
\begin{proof}
{By using $q(\gamma) = 1-\eta\gamma(2-\gamma L)$, for the persistent
	error we obtain $P(\g)=\frac{ \gamma\nu^2}{\eta (2-\gamma L)}$.
		Therefore, the derivative of the persistent error with respect
		to $\g$ is given by 
                  $P'(\g) =\frac{\nu^2}{\eta}\, \frac{2}{(2-\g
				L)^2}> 0$ for all $\g\ne\frac{2}{L}$.}
\end{proof}

Therefore, when $\gamma_t$ is reduced to $\gamma_{t+1}$, the persistent
error does indeed reduce. This drop in steplength is referred to as 
the {\em cascading step} and marks the commencement  of  a new regime. 
As earlier, in this regime, the persistent error will be smaller than 
the transient error and the process of determining $K_{t+1}$ can be repeated.
Therefore, we may view the scheme as a diminishing steplength scheme
where the steplength is reduced at a sequence of time epochs and between
these epochs, it is maintained constant. 

We now proceed to describe the scheme more formally. It can be viewed as
having two stages, of which the second stage repeats infinitely
often in a consecutive fashion. The first of these is an initialization
phase. {We assume throughout that the constraint set $X$ is bounded,
so that $\EXP{\|x_0-x^*\|^2}\le D^2$ with $D=\max_{x,y\in X}\|x-y\|$.
Next, we describe each of the stages in cascading scheme in some} detail.

\noindent {\bf \underline{Cascading steplength stochastic approximation
	(CSA) scheme}:}

\noindent {\em Initialization phase (Phase I):} A requirement to begin
making gradient steps, is that the persistent error has to be smaller
than $D^2$. If this were not the case, then $\gamma$ would have to be
reduced {until} the persistent error is smaller than $D^2$. {
More specifically, given a parameter
$\theta\in(0,1)$, we determine the integer $\ell$ such that}
\begin{align}
\ell \triangleq \min_{j} \left\{ D^2> \frac{\g^2
	\theta^{2j}\nu^2}{1-q(\gamma \theta^j)} 
	\right \}, 
\end{align}
where $q(\gamma)=1-\eta(2-L\gamma)$ and $0<\gamma<\frac{2}{L}.$ 
We define $\gamma_0$ as $\g_0 \triangleq  
\gamma \theta^\ell$, $q_0=q(\g_0)$, and 
\begin{equation}\label{epoch:K0}
K_0=\max_{k}\left\{ k \in \mathbb{Z}_+: q_0^k D^2 >
 \frac{\g_0^2\nu^2}{1-q_0} \right\}.\end{equation}
Finally, we exit this phase by defining $\Kbar_{-1}=0$, 
setting $t=0$, and going to Phase II$_t$.

 \begin{figure}[htb]
\includegraphics[width=2.16in]{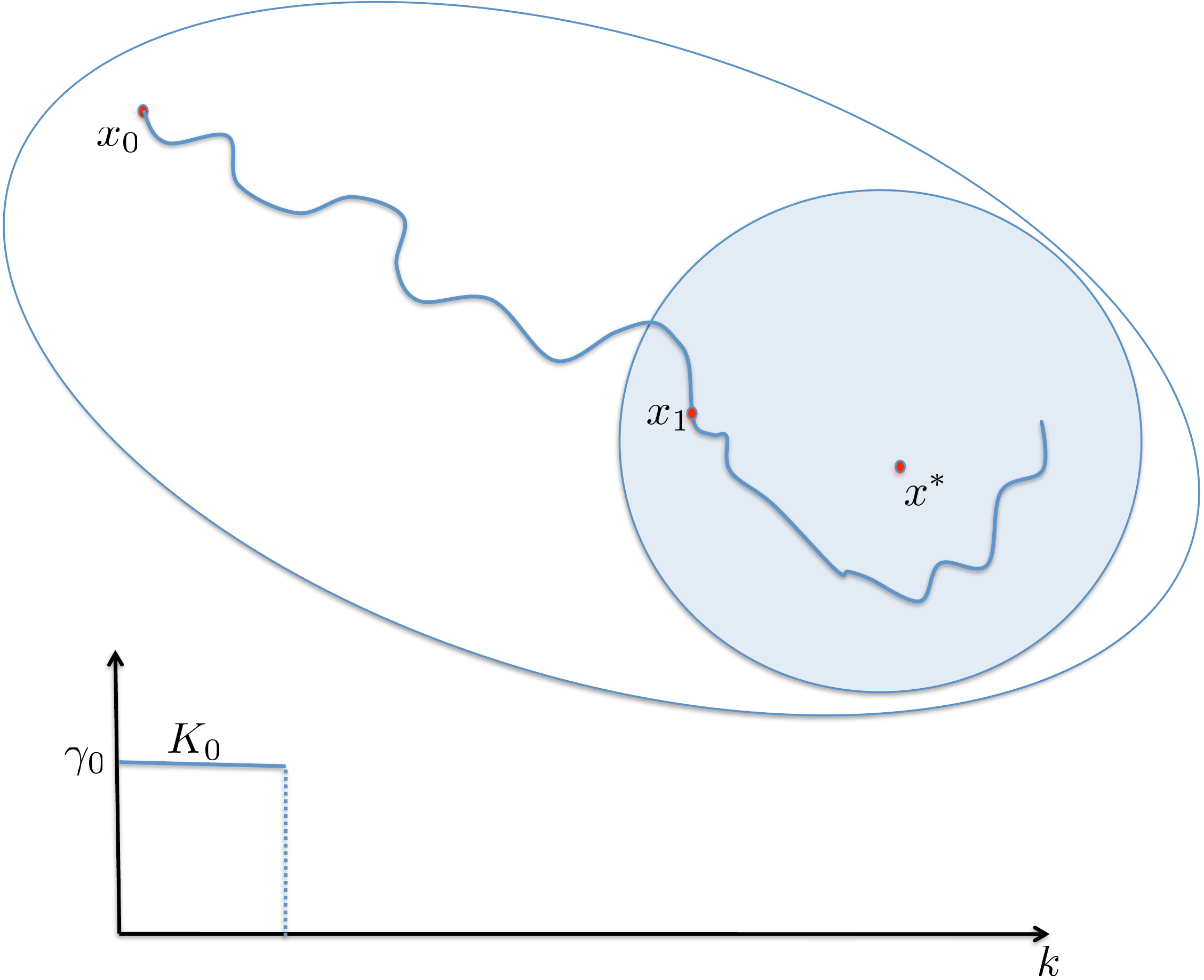}
\includegraphics[width=2.16in]{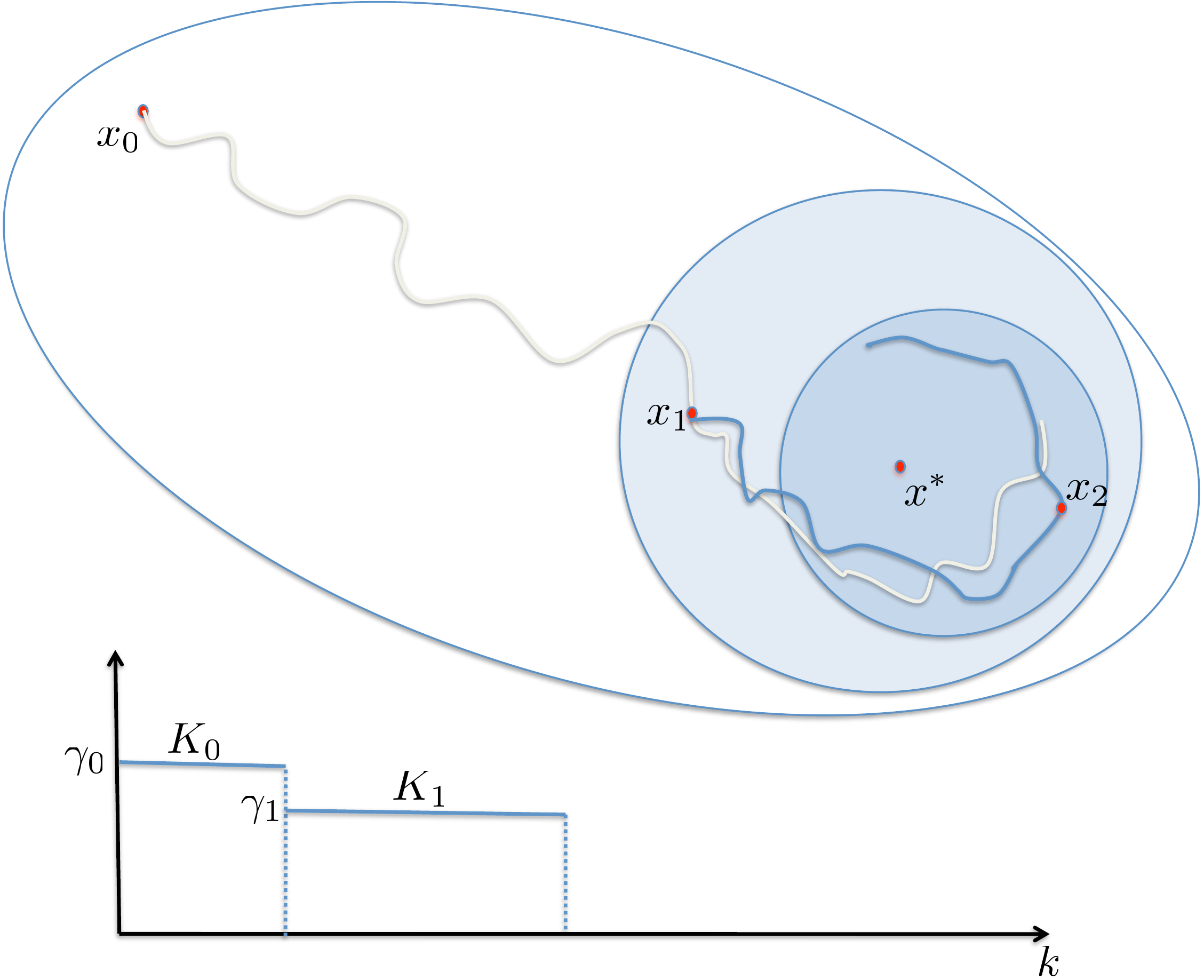}
\includegraphics[width=2.16in]{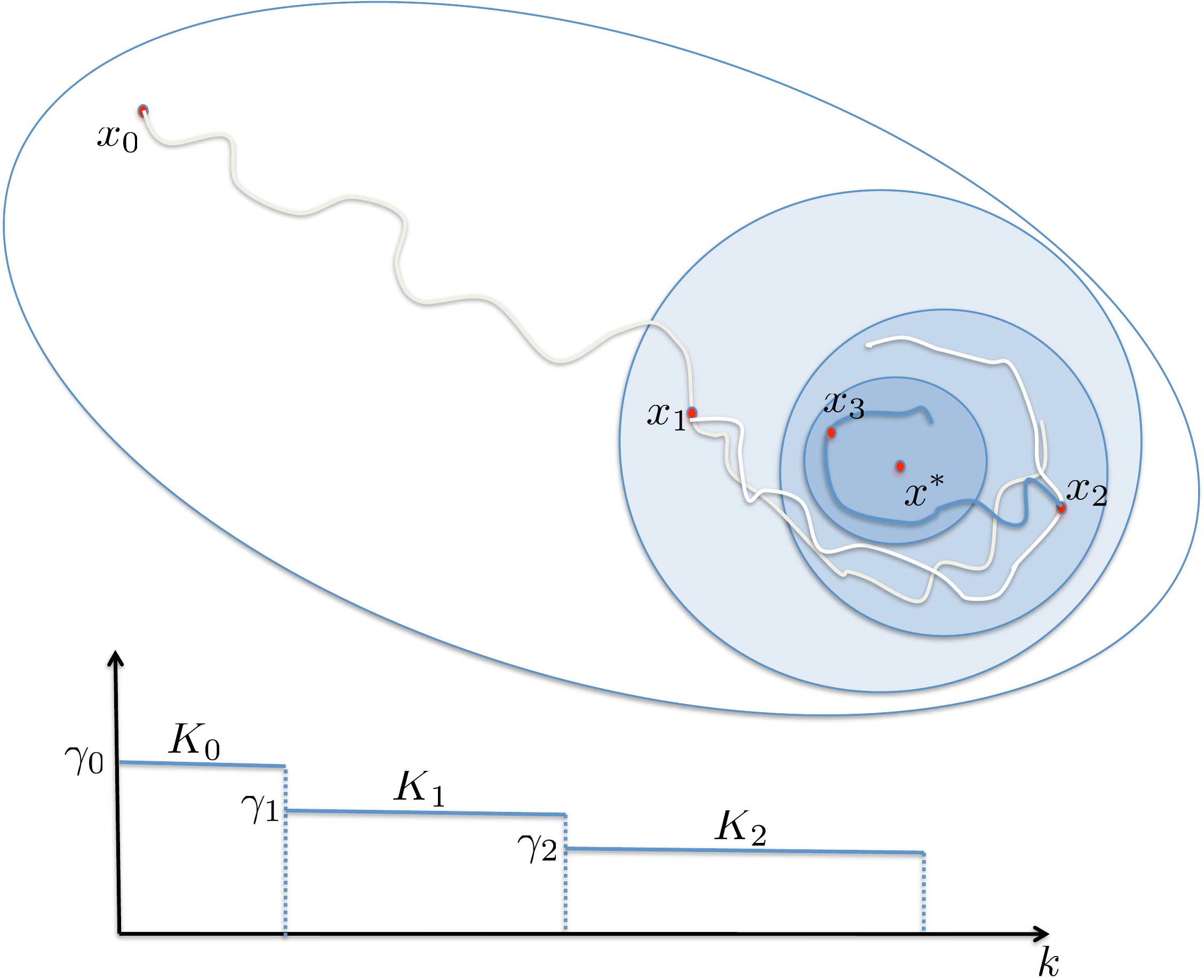}
 \caption{The cascading scheme with phases II$_0$ (left), II$_1$ (center) and
	 II$_2$ (right).}
 \label{fig:casc_algo} 
 \end{figure}
 
\noindent {\em Constant steplength phase (Phase II$_t$):}
Define $ \Kbar_t = \sum_{j=0}^t K_t$.  For the iteration indices $k$
with $k \in \{\Kbar_{t-1}+1, \hdots, \Kbar_t\}$, the stepsize is kept
constant and equal to $\g_t$, i.e.,
\[\g_k=\g_t\qquad\hbox{for all }k=\Kbar_{t-1}+1, \hdots, \Kbar_t.\] 
Then, we  increase $t$ by setting $t=t+1$, 
reduce  the stepsize by letting  $\g_t \triangleq \gamma_{t-1} \theta$, compute
$q_t=q(\g_t)$ and determine the integer $K_t$ as follows:
\begin{align} 
K_t & \triangleq \max_k \left\{ k \in \mathbb{Z}_+:
q_{t}^k2^{t}\left(\prod_{j=0}^{t-1}q_j^{K_j}\right) D^2 > 
\frac{\g_t^2\nu^2}{1-q_{t}}\right\}. 
\label{epoch:Kt} 
\end{align}
We then repeat phase II$_{t}$ until the number $k$ of iterations (i.e.,
gradient steps) exceeds a pre-specified threshold, in case of which the
algorithm terminates.  

We provide a graphical representation of these phases in 
Figure~\ref{fig:casc_algo} where the circles around $x^*$ represent thresholds 
beyond which the transient error is less than the persistent error. 
For instance, in Figure~\ref{fig:casc_algo} (plot to the left), 
phase II$_0$ requires $K_0$ steps to 
reach the first circle. Once, the steplength is reduced  by a factor $\theta$, 
the phase II$_1$ commences and requires $K_1$
steps to reach an analogous error threshold where the transient error is
equal to the persistent error; this is illustrated in 
Figure~\ref{fig:casc_algo} (plot in the center). Finally, phase II$_2$ requires
$K_2$ to reach an even smaller level of persistent error, as depicted in 
Figure~\ref{fig:casc_algo} (plot to the right). Note that whenever 
the steplength is reduced, the persistent error is immediately reduced 
(Lemma~\ref{lemma:decrease}). Thus, the stepsize is essentially a piecewise 
constant decreasing function of the iteration index $k$. 

The next result establishes the correctness of the cascading scheme by
showing that $K_t$ in Phase II$_t$ is finite, so the scheme is well
defined.
\begin{proposition} 
Let Assumptions~\ref{assum:convex} and~\ref{assum:bounded_error} hold. Also,
let $f$ be differentiable over the set $X$ with Lipschitz gradients
with constant $L>0$ and strongly convex with constant $\eta >0$. 
Assume that the set $X$ is compact and let 
$D=\max_{x,y\in X}\|x-y\|$. Then, $K_t$ is finite for all $t\ge0$.
\end{proposition}
\begin{proof}
We use induction on $t$ to show that $K_t$ is well defined and 
for all $t\ge0$,
\begin{equation}\label{eqn:jedan}
\EXP{\|x_{\bar K_t}-x^*\|^2} < 2^{t+1}
\left(\prod_{j=0}^{t} q_j^{K_j}\right)D^2.
\end{equation}
First note that, since $\g_0\in(0,\frac{2}{L})$ and the steplength $\g_k$
is non-increasing in $k$, we have $q(\gamma_t)\in(0,1)$ for all $t\ge0$.

For $t=0$, from Proposition~\ref{prop:casc_inequ}
 and the boundedness of the set $X$ we have 
\begin{equation}\label{eqn:constant-rateo}
\EXP{\|x_k-x^*\|^2} \le q_0^k D^2+\frac{\g_0^2\nu^2}{1-q_0}
\qquad\hbox{for all $k\ge0$},
\end{equation}
where $q_0 = q(\g_0)=1-\eta \g_0 (2-\g_0 L)$ and $\g_0$ is as given in 
the initialization phase of the cascading scheme. Since $\g_0$ is selected in 
the initialization phase so that $D^2>\frac{\g_0^2\nu^2}{1-q_0}$ and 
$q_0^k D^2$ is decreasing as $k$ increases, there exists
an integer $\tilde K\ge1$ such that
$q_0^{\tilde K} D^2 \le \frac{\g_0^2\nu^2}{1-q_0}$. Note that 
$K_0=\tilde K-1$, thus $K_0$ is well defined. Furthermore,
since $q_0^k D^2 > \frac{\g_0^2\nu^2}{1-q_0}$ for $k=0,\ldots K_0$,
from~\eqref{eqn:constant-rateo} we have
\[\EXP{\|x_{\bar K_0}-x^*\|^2} < 2\g_0^{K_0} D^2,\]
where we use the fact $\bar K_0=K_0$ (see Phase II$_t$ for $t=0$).
\begin{figure}[htb]
 \centering
 \subfloat[Transient vs.
 Persistent]{\label{fig:casc_t_p}
\includegraphics[width=2.4in,height=2.4in]{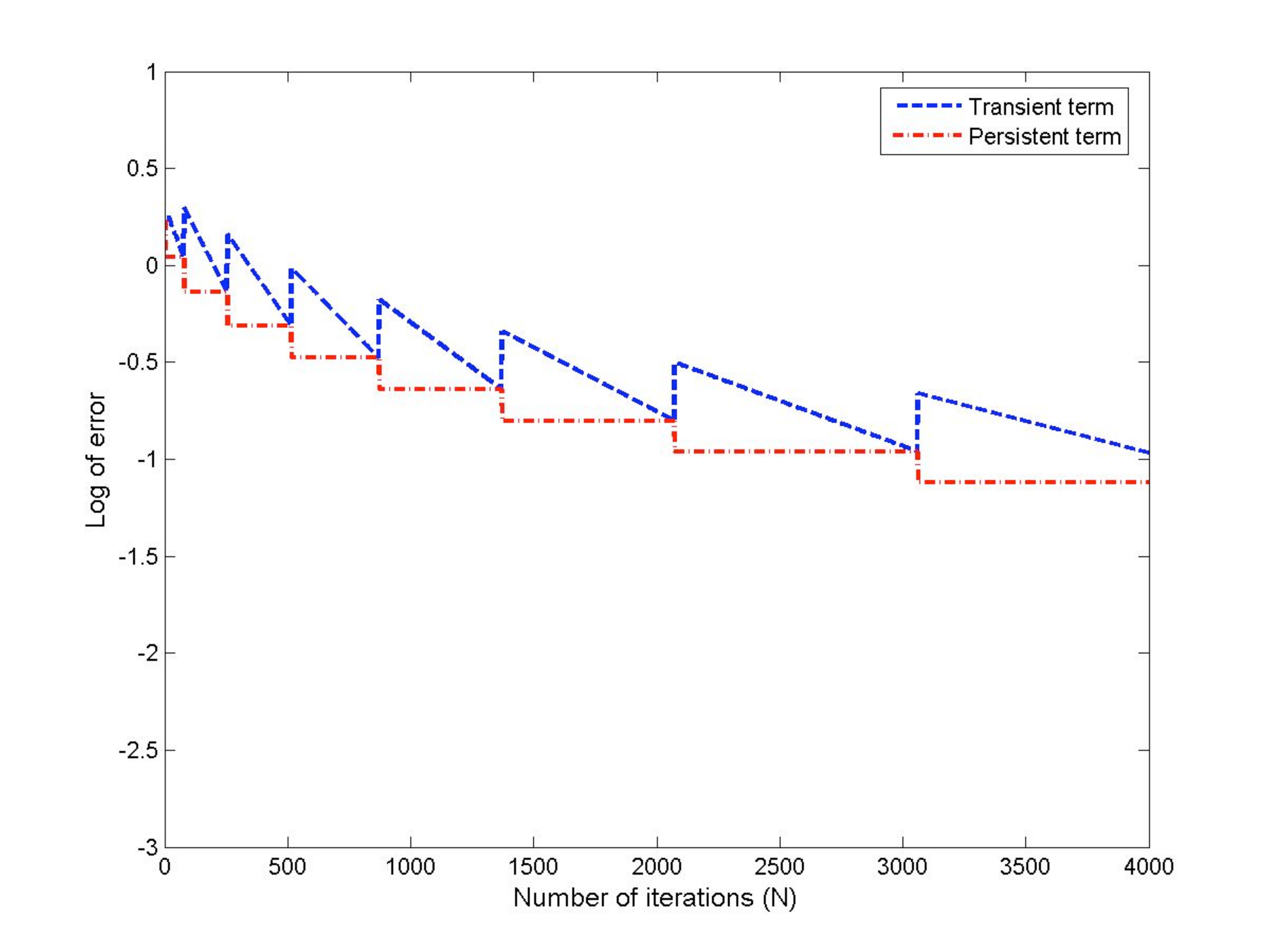}}
 \subfloat[Total]{\label{fig:casc_tot}
\includegraphics[width=2.4in,height=2.4in]{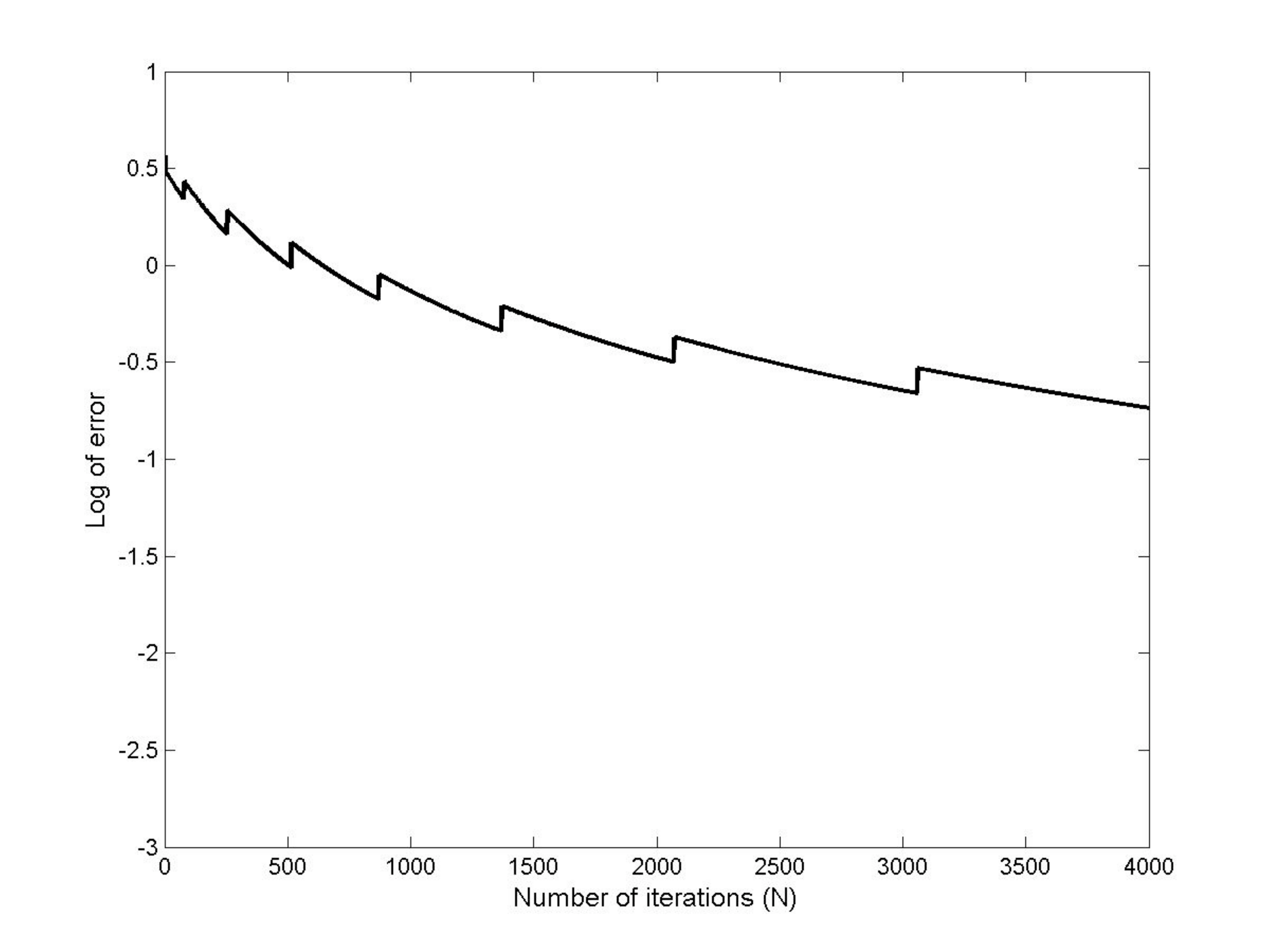}}
  \caption{Elements of cascading scheme for the stochastic utility problem.}
 \label{fig:casc_utility}
 \end{figure}

Now assume that $K_t$ is well defined and relation~\eqref{eqn:jedan} 
holds for $t$. We next show that $K_{t+1}$ is also well defined and 
relation~\eqref{eqn:jedan} holds for $t+1$. Note that the steplength
$\g_k=\g_{t+1}$ is used for $k\ge \bar K_t$. 
From Proposition~\ref{prop:casc_inequ}
where we replace
$x_0$ with $x_{\bar K_t}$, by replacing $\gamma$ by $\gamma_{t+1}$
letting $q_{t+1}=q(\gamma_{t+1})$, 
we have for $k\ge \bar K_t$,
\[\EXP{\|x_k-x^*\|^2} \le q_{t+1}^k \EXP{\|x_{\bar K_t}-x^*\|^2}+
\frac{\g_{t+1}^2\nu^2}{1-q_{t+1}}.\]
By inductive hypothesis relation~\eqref{eqn:jedan} holds, so it follows
\begin{equation}\label{eqn:last}
\EXP{\|x_k-x^*\|^2} <\underbrace{q_{t+1}^k 2^{t+1} 
\left(\prod_{j=0}^{t} q_j^{K_j}\right)}_{\bf Term\,1}
+\underbrace{\frac{\g_{t+1}^2\nu^2}{1-q_{t+1}}}_{\bf Term\,2}
\qquad\hbox{for all $k\ge \bar K_t$}.\end{equation}
Consequently, $K_{t+1}$ is defined  as the largest positive integer $k$ for
	which term 1 is strictly greater than term 2, i.e.,
\[K_{t+1} \triangleq\max_k \left\{ k \in \mathbb{Z}_+:
q_{t+1}^k2^{t+1}\left(\prod_{j=0}^{t}q_j^{K_j}\right) D^2 > 
\frac{\g_{t+1}^2\nu^2}{1-q_{t+1}}\right\}\]
(see the definition of $K_t$ in~\eqref{epoch:Kt}).
Noting that $\bar K_{t+1}= \bar K_t + K_{t+1}$ (see Phase II$_t$)
and $q_{t+1}^k2^{t+1}\left(\prod_{j=0}^{t}q_j^{K_j}\right) D^2 > 
\frac{\g_{t+1}^2\nu^2}{1-q_{t+1}}$
for $k=\bar K_t+1,\ldots, \bar K_{t+1}$,
from~\eqref{eqn:last} with $k=\bar K_{t+1}$, we obtain
\[\EXP{\|x_{\bar K_{t+1}}-x^*\|^2} \le 
2 q_{t+1}^{K_{t+1}}\, 2^{t+1} 
\left(\prod_{j=0}^{t} q_j^{K_j}\right) D^2
=2^{t+2} \left(\prod_{j=0}^{t+1} q_j^{K_j}\right) D^2,\]
thus showing relation~\eqref{eqn:jedan} for $t+1$ and
completing the proof.
\end{proof}

The transient and persistent error trajectories are illustrated in 
in Figure \ref{fig:casc_utility} for a  problem discussed later
in Section \ref{sec:utility_problem}. In Figure \ref{fig:casc_t_p}, the
transient and persistent terms of the error are plotted. The persistent
error, as expected, is a piecewise constant decreasing function of the
iteration count with the {\em jumps}  occurring whenever the
steplengths are reduced. The transient error is a plot of $q_{t}^k 2^{t}
\prod_{j=0}^{t-1} q_j^{K_j} D^2$ with respect to $k$. This function is a
decreasing function when $k \in \{\Kbar_{t-1},\ldots, \Kbar_t-1\}.$ As
soon as $k =\Kbar_t$, {in the transient error the factor $2^{t}$ is
replaced with $2^{t+1}$}, leading to the increase in transient error at that
juncture. The total error, which is the summation of two terms, is
showed in Figure \ref{fig:casc_tot}.  

{\bf Remark on choice of $\theta$:}  Recall that $\theta$ specifies the rate 
at which the
	steplength is dropped over consecutive steps in the cascading
		scheme. It can be readily observed from the bounds derived on
		$K_t$ that if $\theta \to 1$, then $K_t \to 0$ 
thus implying that the steplength is kept constant for a very short period. 
This is intuitive since a conservative drop in steplengths would imply
			 that these drops have to occur more frequently to 
ensure that the sequence is driven to zero. Conversely, if $\theta
			 \to 0$, then $K_t$ can grow to be quite large.

\subsection{Global convergence theory}\label{casc:conv}
In this section, we prove that algorithm (\ref{eqn:algorithm}) using the
cascading steplength scheme is indeed convergent  to the optimal
solution of problem (\ref{eqn:problem}).
\begin{lemma}\label{conv-ln}
Let $q(\g) = 1-2\eta \g + \eta L \g^2$ and let $\eta < L$. 
Then,  we have 
\begin{align*}
0 & < \frac{-\ln(q(\g))}{\g}  \qquad
\hbox{for $\g\in \left(0,\frac{2}{L}\right)$},\\
\frac{-\ln(q(\g))}{\g} & \leq \frac{2\eta L}{L-\eta}\quad 
\hbox{for $\g\in\left(0,\frac{2}{L}\right)$
}.
\end{align*}
Furthermore  
	$$ \lim_{\g \to 0}\, \frac{-\ln(q(\g))}{\g} =2\eta.$$
\end{lemma}
\begin{proof}
Let $r(\g)=\frac{-\ln(q(\g))}{\g}$. Note that
the function $q(\g) = 1-2\eta \g + \eta L \g^2$ is nonnegative for all $\g$
since $0<\eta\le L$. Furthermore $q(\g)< 1$ for $\g<\frac{2}{L}$.
Thus, $r(\g)>0$ for $0<\g<\frac{2}{L}$.
We next show that $r(\g)$ is bounded from above as stated.
To show that the sequence is bounded, we 
employ the Taylor expansion of $\ln(q(\g))$. First, we write
\[-\ln (q(\g)) =-\ln (1-\beta(\g))\qquad\hbox{with }
\beta(\g) =2\eta\g - \eta L \g^2. \]
Noting that $\beta(\g)=1-q(\g)\in(0,1)$, we then use 
the fact $\ln(1-x)=-\sum_{k=1}^\infty \frac{x^k}{k}$ for $|x|<1$, and obtain
\[-\ln (q(\g)) = \sum_{k=1}^\infty \frac{\beta^k(\g)}{k}
\le \sum_{k=1}^\infty \beta^{k}(\g)=\frac{\beta(\g)}{1-\beta(\g)}
=\frac{\beta(\g)}{q(\g)}.\]
Using $\beta(\g)\le 2\eta\g$, we further obtain
\[\frac{-\ln (q(\g))}{\g}\le \frac{2\eta}{q(\g)}.\]
The function $q(\g)$ is convex over $\Real$ and it attains its minimum at
$\g^*=\frac{1}{L}$ with the minimum value $q^*=1-\frac{\eta}{L}$.
The minimum value satisfies $q^*>0$ when $L>\eta$.
Thus, when $\eta<L$, we have $q(\g)\ge 1-\frac{\eta}{L}$,
implying that
\[\frac{-\ln (q(\g))}{\g}\le \frac{2\eta L}{L-\eta}.\]
The relation for the limit is obtained by applying
L'H\^{o}pital's rule, as follows:
\begin{align*}
 \lim_{\g \to 0}  \, \frac{-\ln(1-2\eta \g + \eta L \g^2)}{\g}
=  \lim_{\g \to 0}\, \frac{2\eta - 2\eta L \g}{1-2\eta \g + \eta L \g^2} 
= 2\eta.
\end{align*}
\end{proof}

\begin{proposition}[Cascading steplength stochastic approximation (CSA) scheme]
\label{prop:cascadingscheme} 
Let Assumptions~\ref{assum:convex} and~\ref{assum:bounded_error} hold. Also,
let $f$ be differentiable over the set $X$ with Lipschitz gradients with 
constant $L>0$ and strongly convex with constant $\eta >0$, where $L>\eta$. 
Assume that the set $X$ is compact and let $D=\max_{x,y\in X}\|x-y\|$. 
Let the sequence $\{x_k\}$ be generated by algorithm~(\ref{eqn:algorithm}) and 
cascading steplength scheme with $\g_0\in\left(0,\frac{2}{L}\right)$ and 
$\theta\in(0,1)$. Then, $\{x_k\}$ converges almost
surely to the unique optimal solution of problem~(\ref{eqn:problem}).
\end{proposition}
\begin{proof} The result will follow from 
Proposition~\ref{prop:lipschitzgrad} provided we verify that
Assumption~\ref{assum:step_error} holds, i.e., 
$\sum_{k=0}^\infty \g_k = \infty$ and $\sum_{k=0}^\infty \g_k^2 < \infty$. 
According to Phase~II$_t$ of the cascading scheme, we have
$\g_k=\g_t$ for $k=\bar K_{t-1}+1,\ldots, \bar K_t$ with 
$\g_t=\theta^t\g_0$ and $\bar K_t=\bar K_{t-1}+K_t$. Therefore
\[\sum_{k=0}^\infty \g_k = \g_0 \sum_{j=0}^\infty K_j\theta^j,\qquad
 \sum_{k=0}^\infty \g_k^2 =\g_0^2 \sum_{j=0}^\infty K_j\theta^{2j}.\]
Thus, we need to show  
\[\sum_{j=0}^\infty K_j\theta^j=\infty,\qquad
\sum_{j=0}^\infty K_j\theta^{2j}<\infty.\]

From the definition of $K_t$ 
in~\eqref{epoch:Kt} we have
\begin{equation}\label{eqn:rel1}
q_t^{K_t} 2^t\left(\prod_{j=0}^{t-1}q_j^{K_j}\right) D^2 
> \frac{\g^2_t\nu^2}{1-q_t},\end{equation}
while $K_t+1$ satisfies
\begin{equation}\label{eqn:rel2}
q_t^{K_t+1} 2^t\left(\prod_{j=0}^{t-1}q_j^{K_j}\right) D^2 
\le \frac{\g^2_t\nu^2}{1-q_t}.\end{equation}
Relation~\eqref{eqn:rel2} and the fact $\g_t=\theta^t\g_0$
(see Phase~II$_t$ of the cascading scheme) yield
\begin{align*}
q_{t}^{K_t+1} \left(\prod_{j=0}^{t-1}q_j^{K_j}\right)  \le 
\frac{\g_0^2\left(\frac{\theta^{2}}{2}\right)^t\nu^2}{D^2(1-q_{t})}  
\implies \left(\prod_{j=0}^{t}q_j^{\tilde K_j}\right)  \le 
q_{t}^{K_t+1} \left(\prod_{j=0}^{t-1}q_j^{K_j}\right)  \le
\frac{\g_0^2\left(\frac{\theta^{2}}{2}\right)^t\nu^2}{D^2(1-q_{t})},
\end{align*}
where $\tilde K_j = K_j +1.$ Consequently,  by taking logarithms and noting 
that $q_j\in (0,1)$ for all $j$ (since $\g_0\theta^j\in(0,2/ L)$ by the choice
of $\g_0$ and $\theta\in(0,1)$) we have
\begin{align*}
\sum_{j=0}^t \tilde K_j \ln (q_j)  \leq \ln
\left(\frac{\g_0^2\left(\frac{\theta^{2}}{2}\right)^t\nu^2}
{D^2(1-q_{t})}\right)
	\implies 
 \sum_{j=0}^t \tilde K_j (-\ln (q_j))  \geq -\ln
\left(\frac{\g_0^2\left(\frac{\theta^{2}}{2}\right)^t\nu^2}
{D^2(1-q_{t})}\right).
\end{align*}
Therefore, by multiplying and dividing by $\g_0\theta^j$, we obtain
\begin{align*}
\g_0\sum_{j=0}^t \tilde K_j \theta^j\,\left(\frac{-\ln (q_j)}{\g_0\theta^j}
\right)
& \geq -\ln
\left(\frac{\g_0^2\left(\frac{\theta^{2}}{2}\right)^t\nu^2}
{D^2(1-q_{t})}\right).
\end{align*}
Note that $q_j=1-2\eta\g_0\theta^j+\eta L (\g_0\theta^j)^2$ with 
$\g_0\in (0,2/ l)$ and $\theta\in(0,1)$. Thus, by Lemma~\ref{conv-ln} we have
$\frac{-\ln(q_j)}{\g_0\theta^j}\le 2\eta L /(L-\eta)$, implying
\begin{align*}
\frac{2\g_0\eta L}{L-\eta}\ \sum_{j=0}^t \tilde K_j \theta^j 
& \geq -\ln
\left(\frac{\g_0^2\left(\frac{\theta^{2}}{2}\right)^t\nu^2}
{D^2(1-q_{t})}\right).
\end{align*}
Taking limits on both sides, we have that 
\begin{align*}
\frac{2\g_0\eta L}{L-\eta}\ \sum_{j=0}^\infty \tilde K_j \theta^j 
& \geq \lim_{t \to \infty} 
-\ln\left(\frac{\g_0^2\left(\frac{\theta^{2}}{2}\right)^t\nu^2}{D^2(1-q_{t})}
\right).
\end{align*}
The limit on the right can be simplified by substituting 
$q_t=1-2\eta
		\gamma_0 \theta^t + \eta L \gamma_0^2 \theta^{2t}$,
	leading to   
\begin{align*}
 \quad -\lim_{t \to \infty} \ln 
\left(\frac{\g_0^2\left(\frac{\theta^{2}}{2}\right)^t\nu^2}{1-q_t}\right) 
 = - \lim_{t \to \infty} \ln 
\left(\frac{\g_0^2\left(\frac{\theta}{2}\right)^t\nu^2}{D^2(2\eta
		\g_0\theta^t  - \eta L \g_0^2 \theta^{t})}\right)  = +\infty,
\end{align*}
where we also use $\theta\in(0,1)$.
Hence, $\sum_{j=0}^t \tilde K_j \theta^j = +\infty.$
Since $\tilde K_j=K_j+1$ and $\theta\in(0,1)$, it follows that
$$\infty = \sum_{j=0}^{\infty} \tilde K_j \theta^j = \sum_{j=0}^\infty K_j
\theta^j + \sum_{j=0}^{\infty} \theta^j = \sum_{j=0}^\infty K_j
\theta^j + \frac{1}{1-\theta}, $$
implying that $\sum_{j=0}^{\infty} K_j \theta^j = \infty.$

It remains to show that $\sum_{t=0}^{\infty}{K_t\theta^{2t}} < \infty.$ 
From \eqref{eqn:rel1} and the fact $q_j\in(0,1)$ for all $j$, we have that 
\begin{align*}
\frac{\g_0^2\left(\frac{\theta^{2}}{2}\right)^t\nu^2}{D^2(1-q_{t})} \leq
\prod_{j=0}^t q_j^{K_j} \leq q_t^{K_t}.
\end{align*}
This allows for obtaining an upper bound on $K_t$, given by 
\begin{align}\label{eqn:boundKt}
K_t \leq \frac{ \ln\left(
\frac{\g_0^2\left(\frac{\theta^{2}}{2}\right)^t\nu^2}{D^2(1-q_{t})}\right)}{
			\ln q_t}. 
\end{align}
The desired result will follow by the Cauchy root test, if we show that 
$$ \lim_{t \to \infty} (K_t \theta^{2t})^{1/t}  < 1. $$
By noting that $(K_t \theta^{2t})^{1/t} = \theta^2 (K_t)^{1/t}$, it
suffices to use the upper bound on $K_t$ in~\eqref{eqn:boundKt}.
We proceed to analyze this bound, for which 
by letting $\beta(\g)=1-q(\g)$ and recalling that $q_t=q(\g_t)$ we have
\[
\ln\left(
\frac{\g_0^2\left(\frac{\theta^{2}}{2}\right)^t\nu^2}{D^2(1-q_{t})}\right)
=t \ln \frac{\theta}{2} 
+ \ln \left(\frac{\g_0^2\nu^2}{D^2}\right) -\ln (\beta(\g_t)).\] 
Thus, 
\begin{align*}
K_t^{1/t}  \leq   \left( \frac{t \ln \frac{\theta}{2} 
+ \ln \left(\frac{\g_0^2\nu^2}{D^2}\right) -\ln (\beta(\g_t))}
{\ln (q_t) }\right)^{1/t}
= \left( \frac{- t \ln \frac{\theta}{2} 
- \ln \left(\frac{\g_0^2\nu^2}{D^2}\right) +\ln (\beta(\g_t))}
{-\ln (q_t) }\right)^{1/t}.
\end{align*}
Noting that $\beta(\g)\in (0,1)$ for all $\g$ when $\eta<L$, 
we have $\ln(\beta(\g))<0$, implying
\begin{align}\label{eqn:rel3}
K_t^{1/t}  \leq  \left( \frac{- t \ln \frac{\theta}{2} 
- \ln \left(\frac{\g_0^2\nu^2}{D^2}\right)}
{-\ln (q_t) }\right)^{1/t}.
\end{align}
Since $\beta(\g)\in (0,1)$, the denominator can be 
expanded in Taylor series as follows:
\begin{align*}
-\ln(q_t)=-\ln (1-\beta(\g_t))= \sum_{k=1}^\infty\frac{\beta^k(\g_t)}{k}
\ge \beta(\g_t).\end{align*}
Furthermore, since $\beta(\g_t)= \eta \g_t (2- L\g_t)$
and $\g_t=\g_0\theta^t$ with $\theta\in(0,1)$,  
we have $\g_0\theta^t\le 1$ for $t$ large enough,
implying $\beta(\g_t)\ge\eta\g_0\theta^t$. Thus, 
\begin{align}\label{eqn:rel4}
-\ln(q_t)\ge \eta\g_0\theta^t\qquad\hbox{for $t$ large enough}.\end{align}
By combining \eqref{eqn:rel3} and \eqref{eqn:rel4}, we obtain
for $t$ large enough,
\begin{align*}
K_t^{1/t}
& \le t^{1/t}\left( \frac{- \ln \frac{\theta}{2} 
- \frac{1}{t}\ln\left(\frac{\g_0^2\nu^2}{D^2}\right)}
    { \eta \g_0\theta^t }\right)^{1/t}
=\frac{ t^{1/t}}{\theta (\eta\g_0)^{1/t}} \left(- \ln \frac{\theta}{2} 
- \frac{1}{t} \ln\left(\frac{\g_0^2\nu^2}{D^2}\right) \right)^{1/t}. 
 \end{align*}
By recalling that $\lim_{t \to \infty} {t^{1/t}} = 1$
and $\lim_{t \to \infty} {c^{1/t}} = 1$ for any $c>0$, it
follows that
\begin{align*}
\lim_{t\to\infty} K_t^{1/t}
& \le \frac{1}{\theta}
\,\lim_{t\to\infty}\left(- \ln \frac{\theta}{2} 
- \frac{1}{t} \ln\left(\frac{\g_0^2\nu^2}{D^2}\right) \right)^{1/t}. 
 \end{align*}
We next examine the limit on the right hand side. 
Letting $a = -\ln (\theta/2)$ and $b = -\ln
		\left(\frac{\g_0^2\nu^2}{D^2} \right)$,
we can write
$$ \lim_{t \to \infty} \left(a+\frac{b}{t}\right)^{1/t} 
= \lim_{t \to \infty}a^{1/t} \left(1+\frac{b}{at}\right)^{1/t} 
=   \lim_{t \to \infty}
a^{1/t} \lim_{t \to \infty}\left(1+\frac{b}{at}\right)^{1/t} = 1.$$
Therefore, $
\lim_{t\to\infty} K_t^{1/t} \le \frac{1}{\theta},$
implying that
\[\lim_{t\to \infty} (K_t \theta^{2t})^{1/t} \le \theta <1,\]
As a consequence, the Cauchy-root test is satisfied and
$\sum_{t=0}^{\infty} K_t \theta^{2t} < \infty.$
\end{proof}

\section{Addressing nondifferentiability through Local Randomized Smoothing}
\label{sec:localrand}
In this section, we develop a smoothing approach for solving stochastic
optimization problem with nonsmooth integrands. In Section~\ref{sec:approx},  
given a nondifferentiable function $f(x)$,  we introduce a smooth approximation
for $f(x)$, denoted by ${\hat f}(x)$ by using local random perturbations. 
In Section \ref{sec:uniform}, we derive Lipschitz constants for the gradients 
associated with this smooth approximation when the smoothing is introduced via 
a uniform distribution. Finally, in Section~\ref{sec:rnd_SA}, the convergence 
theory of stochastic approximation schemes is examined in this modified regime.

\subsection{Differentiable Approximation}\label{sec:approx}
We let $f$ be nondifferentiable and consider its 
approximation $\hat f$, defined by
\begin{equation}\label{eqn:aproxf}
\hat f(x) \triangleq \EXP{f(x+z)},
\end{equation}
where the expectation  is with respect to $z\in\mathbb{R}^n$, a
random vector with a compact support.  Suppose that $z\in\mathbb{R}^n$
is a random vector with a probability distribution over the
$n$-dimensional ball centered at the origin and with radius $\e$. For
the function $\hat f$ to be well defined, we need to enlarge the
underlying set $X$ so that $f(x+z)$ is defined for every $x\in X$. In
particular, for a set $X\subseteq\mbR^n$ and $\e>0$, we let $X_\e$ be
the set defined by: \[X_\e=\{y\mid y=x+z,\ x\in X,\ z\in\mbR^n,\
						   \|z\|\le \e\}.\]

We discuss our local smoothing technique under the assumption that the
function $f$ has uniformly bounded subgradients over the set $X_\e$,
		 given as follows. 
\begin{assumption}\label{assum:bounded_subgradients}
The subgradients of $f$ over $X_\e$ are uniformly bounded, i.e.,
there is a scalar $C>0$ such that $\|g\|  \le C$ for all $g\in\partial f(x)$ 
and $x \in X_\e$.
\end{assumption}

Assumption~\ref{assum:bounded_subgradients} is satisfied, for
example, when $X$ is bounded.  In the sequel, we let $\EXP{g(x+z)}$
denote the vector-valued integral of an element from the set of
subdifferentials, which is given by
\begin{equation}\label{eqn:expsubdiff}
\EXP{g(x+z))}=\left\{\bar g=\int_{\mbR^n}g(x+z)p{_u}(z )dz
\,\Big|\, g(x+z)\in\partial f(x+z)\, a.s.\right\}.
\end{equation}
The following lemma presents properties of the randomized
technique~(\ref{eqn:aproxf}) with an arbitrary local random distribution 
over a ball. It states that, under the boundedness of the subgradients of $f$, 
the set $\EXP{g(x+z)}$ defined above is a singleton. In particular, the lemma 
shows that $\hat f$ is convex and differentiable approximation of $f$.

\begin{lemma}\label{lemma:approxf}
Let $z\in\mathbb{R}^n$ be a random vector with the density distribution 
support contained in the $n$-dimensional ball centered at
the origin and with a radius $\e$, and let $\EXP{z}=0$. 
Let $X\subseteq \mathbb{R}^n$ be a convex set and let the function $f(x)$ be 
defined and convex on the set $X_\e$, where $\e>0$ is the parameter
characterizing the distribution of $z$.
Also, let Assumption~\ref{assum:bounded_subgradients} hold. Then, for the
function $\hat f$ given in~\eqref{eqn:aproxf}, we have:
\begin{itemize}
    \item[(a)] $\hat f$ is convex and differentiable over $X$,
with gradient
\[\nabla \hat f(x) = \EXP{g(x+z)}\hbox{ for all }x\in X,\]
where the vector $\EXP{g(x+z)}$  is as defined in \eqref{eqn:expsubdiff}.
Furthermore, $\|\nabla \hat f(x)\|\le C$ for all $x\in X$.
    \item[(b)] $f(x) \le \hat f(x) \le f(x)+ \e C$ for all $x\in X$.
\end{itemize}
\end{lemma}
\begin{proof}
(a) \ For the convexity and differentiability of $\hat f$ see the
proof\footnote{There, the vector $z$ has a normal zero-mean
distribution. Furthermore, the proof is applicable to a convex
function defined over $\mbR^n$. However, the analysis can be
extended in a straightforward way to the case when $f$ is defined
over an open convex set $\mathcal{D}\subset\mbR^n$, since the
directional derivative $f'(x; d)$ is finite for each
$x\in\mathcal{D}$ and for any direction $d\in\mbR^n$ (Theorem~23.1
in \cite{Rockafellar70}).} of Lemma~3.3(a) in \cite{DeFarias08}. The
gradient boundedness follows by
Assumption~\ref{assum:bounded_subgradients},
relation~\eqref{eqn:expsubdiff}, and $\nabla \hat f(x) =
\EXP{\partial f(x+z)}$.

\noindent (b) \ By definition of random vector $z$, it has zero
mean, i.e., $\EXP{x+z}=x$, so that $f(\EXP{x+z})=f(x)$. Therefore,
by Jensen's inequality and the definition of $\hat f$, we have
\[f(x)=f(\EXP{x+z}) \leq \EXP{f(x+z)}=\hat f(x)\qquad\hbox{for all }x\in X.\]

To show relation $\hat f(x)\le f(x)+\e C$, we use the subgradient
inequality for $f$, which in particular implies that, for every
$\bar x \in X_\e$ and $g\in \partial f(\bar x)$, we have
\[f(\bar x) \le f(x) +\|g\|\,\|x-\bar x\|\qquad\hbox{for all }x\in X_\e.\]
Since $\bar x\in X_\e$, we have $\bar x=x+z$ for some $x\in X$ and
$z$ with $\|z\|\le\e$. Using this and the subgradient boundedness,
from the preceding relation we obtain
\[f(x+z) \leq f(x) +C\e\qquad\hbox{for all }x\in X.\]
Thus, by taking the expectation, we get
$\hat f(x) =\EXP{f(x+z)}\le f(x) +\e C$ for all $x\in X.$
\end{proof}

\subsection{Smoothing via random variables with uniform distributions}
\label{sec:uniform}
 In this subsection, we consider a local smoothing technique wherein $z$
 is generated via a uniform distribution. Other distributions may also work
 such as normal, considered in~\cite{DeFarias08}.
 However, distributions with finite support seem more appropriate 
 for capturing local behavior of a function, as well as to deal with
 the problems where the function itself has a restricted domain. Our choice
 to work with a uniform distribution is due to the uniform distribution lending
 itself readily for computation of resulting Lipschitz
 constant and for assessment of the growth of the Lipschitz constant with the
 size of the problem.
 
 The key result of this
 section is an examination of the Lipschitz continuity of the gradients
 of the smooth approximation, particularly in terms of the rate that
 such a constant grows with problem size. 
 
 Suppose 
 $z\in\mathbb{R}^n$ is a random vector with uniform distribution over
 the $n$-dimensional ball centered at the origin and with a radius $\e$,
 i.e., $z$ has the following probability density function:
\begin{equation}\label{eqn:zuniform}
p{_u}(z) = \left\{\begin{array}{ll} \frac{1}{c_n \varepsilon^n}
&\hbox{for }\|z\| \le\e,\cr \hbox{} &\hbox{}\cr 0
&\hbox{otherwise,}\end{array}\right.
\end{equation}
where $c_{n} = \dfrac{\pi^\frac{n}{2}}{\Gamma(\frac{n}{2} + 1)}$,
and $\Gamma$ is the gamma function given by
\begin{eqnarray*}
\Gamma\left(\frac{n}{2}+1\right)= \left\{ \begin{array}{ll}
\left(\frac{n}{2}\right)! &\hbox{if $n$ is even,}\cr
\hbox{}&\hbox{}\cr \sqrt{\pi}\,\frac{n!!}{2^{(n+1)/2}} &\hbox{if $n$
is odd}.
\end{array}\right.
\end{eqnarray*}
The following lemma shows that
$\hat f$ is convex and differentiable approximation of $f$ with
Lipschitz gradients, where the Lipschitz constant for $\nabla \hat f$ 
is related to the norm bound for the subgradients of $f$.
\begin{lemma}\label{lemma:approxf-uni}
Let $z\in\mathbb{R}^n$ be a random vector with uniform 
density distribution {with zero mean} over a $n$-dimensional ball centered at
the origin and with a radius $\e$. Let $X\subseteq \mathbb{R}^n$ be a 
convex set and let the function $f(x)$ be defined
and convex on the set $X_\e$, where $\e>0$ is the parameter
characterizing the distribution of $z$.
Also, let Assumption~\ref{assum:bounded_subgradients} hold. Then, for the
function $\hat f$ given in~\eqref{eqn:aproxf}, we have
 \[\| \nabla \hat f(x) -\nabla \hat f(y)\| \le 
\kappa \dfrac{ n!!}{(n-1)!!} \,\dfrac{C}{\e}\|x-y\| \qquad 
\hbox{for all } x,y\in X,\]
 where $\kappa=
\frac{2}{\pi}$ if $n$ is even, and otherwise $\kappa = 1$.
\end{lemma}
\begin{proof}
\ From Lemma~\ref{lemma:approxf}(a) and
relation~\eqref{eqn:expsubdiff}, for any $x \in X$, there is a
vector $g(z+x)$ such that $g(z+x)\in\partial f(x+z)$ a.s.\ and
\[\nabla \hat f(x)= \int_{\mathbb{R}^n}g(x+z)p_u(z )dz
=\int_{\mathbb{R}^n}g(v)p(v-x)dv,\]
where the last equality follows by letting $v=x+z$. Therefore,
for any $x,y\in X$, 
\begin{eqnarray}\label{eqn:grad_diff}
\|\nabla \hat f(x) -\nabla \hat f(y)\|
&=&\left\|\int_{X_{\e}}(p{_u}(z -x)-p_u(z -y))g(z)dz\right\|
\cr
&\leq & \int_{X_{\e}}|p_u(z -x)-p_u(z -y)|\|g(z)\| dz\cr
&\leq&  C \int_{X_\e}|p_u(z -x)-p_u(z -y)|dz,
\end{eqnarray}
where the last inequality follows by using the boundedness of the
subgradients of $f$ over $X_\e$.

Now, we let $x,y\in X$ be arbitrary but fixed, and we 
estimate $\int_{X_\e}|p{_u}(z -x) - p{_u}(z -y)| dz$
in~\eqref{eqn:grad_diff}. For this we consider the cases where 
$\|x-y\|>2\e$ and $\|x-y\|\le 2\e$.

\noindent{\it Case 1} ($\|x-y\|> 2\e$): For every $z$ with
$\|z - x \| \le \e$, we have $\|z -y\| > \e$, implying that
$p{_u}(z -y)=0$, so that $\int_{\| z - x \|\le \e} |p{_u}(z
-x) - p{_u}(z -y)|dz=1$. Likewise, for every $z$ with $\|z - y \|
\le \e$, we have $p{_u}(z -x)=0$, implying 
$$\int_{\| z - y \| \le \e}|p{_u}(z -x) - p{_u}(z -y)|dz=1.$$ 
Therefore, 
\begin{align*}\label{eqn:case1}
\int_{X_\e}{|p_u(z -x) - p_u(z -y)|}dz 
&= \int_{\| z -x \| \le \e} |p_u(z -x) - p_u(z -y)|dz  
+\int_{\| z - y \| \le \e}|p_u(z -x) - p_u(z -y)|dz \cr
&= 2.
\end{align*}
Since $2 < \| x-y \|/\e$, it follows that
\begin{equation}\label{eqn:case1_est}
\int_{X_\e}{|p_u(z -x) - p_u(z-y)|}dz
\le\frac{\|x-y\|}{\e}.\end{equation} It can be further seen
that $\kappa \frac{n!!}{(n-1)!!}\ge1$ for all $n\ge1$,
which combined with~\eqref{eqn:case1_est} and~\eqref{eqn:grad_diff}
yields the result.

\noindent{\it Case  2} ($\|x-y\|\le 2\e$): We decompose the integral
in~\eqref{eqn:grad_diff} over several regions, as follows:
\begin{equation*}
\begin{split}
&\int_{X_\e}{|p_u(z -x)-p_u(z -y)| }dz\cr
&=\int_{\|z-x\|\le\e\ \&\ \|z -y \|\le\e}|p_u(z -x) -p_u(z -y)|dz 
+\int_{\|z-x\|\le\e\ \&\ \|z- y\|\ge\e}|p_u(z-x) - p_u(z -y)|dz\cr 
&\quad +\int_{\|z-x\|\ge\e\ \&\
\|z-y\|\le \e}|p_u(z -x) - p_u(z -y)|dz
+\int_{\|z-x\|\ge\e\ \& \ \|z-y\|\ge\e}|p_u(z -x) - p_u(z-y)|dz.
\end{split}
\end{equation*}
The first and the last integrals are zero, since 
$p_u(z-x)=p_u(z -y)$ for $z$ in the integration region there.
Furthermore, in the other two integrals, the supports of $p_u(z-x)$ and 
$p_u(z -y)$ do not intersect, so that we have
$|p_u(z -x)-p_u(z -y)|=1/(c_n\e^n)$ for $z$ in the integration region there. 
Using this and the symmetry of these integrals, 
by letting $S=\{z \in \mbR^n \mid\| z - x \| \le \e\hbox{ and } \| z - y \| 
\geq \e\}$, we obtain
\begin{equation}\label{eqn:volout1}
\int_{X_\e} |p_u(z -x) - p_u(z -y)| dz
=\frac{2}{c_n\e^n}\, V_S,\end{equation} 
where $V_S$ denotes the volume of the set $S$.

Now we want to find an upper bound for $V_S$ in terms of $\| y - x\|$. 
Let $V_{cap}(d)$ denote the volume of the spherical cap with the
distance $d$ from the center of the sphere. Therefore,
\begin{equation}\label{eqn:volout2}
V_S=c_n\e^n-2V_{cap}\left(\frac{\|x-y \|}{2}\right).\end{equation}
The volume of the $n$-dimensional spherical cap with distance $d$
from the center of the sphere can be calculated in terms of the
volumes of $(n-1)$-dimensional spheres, as follows:
\[V_{cap}(d)
=\int_{d}^{\e}c_{n-1}\left(\sqrt{\e^2-\rho^2}\right)^{n-1}d\rho
\qquad\hbox{for }d \in [0,\e],\] with
$c_n=\dfrac{\pi^{n/2}}{\Gamma(\frac{n}{2}+1)}$ for $n\ge1$. We have
for $d \in [0,\e]$,
\begin{equation*}
\begin{split}
&V_{cap}'(d)= -c_{n-1}(\e^2-d^2)^{\frac{n-1}{2}} \leq 0,\cr
&V_{cap}''(d)= (n-1)c_{n-1}d(\e^2-d^2)^{\frac{n-3}{2}} \geq 0,
\end{split}
\end{equation*}
where $V_{cap}'$ and $V_{cap}''$ denote the first and the second
derivative, respectively, with respect to $d$. Hence, $V_{cap}(d)$
is convex over $[0,\varepsilon]$, and by the subgradient inequality
we have
\[V_{cap}(0)+V_{cap}'(0)\,d\le V_{cap}(d)\qquad\hbox{for }d\in[0,\e].\]
Since $V_{cap}(0)=\frac{1}{2}c_n\e^n$ and
$V_{cap}'(0)=-c_{n-1}\e^{n-1} $, it follows
\begin{equation}\label{eqn:vdlower}
\frac{1}{2}c_n\e^n-c_{n-1}\e^{n-1}d\le V_{cap}(d)
\qquad\hbox{for }d\in[0,\e].\end{equation} 
Noting that $\|x-y\|/2\le \e$ (since $\|x-y\|\le 2\e$), we can let 
$d=\|x-y\|/2\le \e$ in~\eqref{eqn:vdlower}. By doing so and 
using~\eqref{eqn:volout2}, we obtain
\[V_S=c_n\e^n-2V_{cap}\left(\frac{\|x-y \|}{2}\right)
\le 2 c_{n-1}\e^{n-1}\frac{\|x-y\|}{2}.\] 
Finally, substituting the preceding relation in~\eqref{eqn:volout1}, we have
\[\int_{X_\e}|p_u(z -x)-p_u(z -y)|dz 
\le\frac{2c_{n-1}}{c_n}\,\frac{\|x-y\|}{\e}.\]
Since $c_n=\frac{\pi^{n/2}}{\Gamma(\frac{n}{2}+1)}$, it can be seen
that
\begin{equation}\label{rate-cn}
\frac{2c_{n-1}}{c_n}=\kappa \frac{n!!}{(n-1)!!},
\end{equation}
with $\kappa=\frac{2}{\pi}$ if $n$ is even, and otherwise
$\kappa=1$. Thus, we have
\begin{equation}\label{eqn:case2_est}
\int_{\mbR^n}|p{_u}(z -x)-p{_u}(z -y)|dz \le \kappa
\frac{n!!}{(n-1)!!} \,\frac{\|x-y\|}{\e}.\end{equation} 
By combining~\eqref{eqn:case2_est} with~\eqref{eqn:grad_diff}, we
obtain the desired result.
\end{proof}

It can be seen that the Lipschitz constant 
$\kappa\frac{n!!}{(n-1)!!}\,\frac{C}{\e}$ established in
Lemma~\ref{lemma:approxf} for the differentiable approximation $\hat
f$ grows at the rate of $\sqrt{n}$ with the number $n$ of the
variables, i.e.,
\[\lim_{n\to\infty} \ \frac{\kappa \frac{n!!}{(n-1)!!}}
{\sqrt{n}}=\sqrt{\frac{\pi}{2}}.\] This growth rate is worse than
the growth rate $\sqrt{\ln(n+1)}$ obtained in~\cite{DeFarias08} for
the global smoothing approximation, which uses a normally
distributed perturbation vector $z$. However, it should be
emphasized that the smoothing technique in~\cite{DeFarias08}
requires the function $f$ to be defined over the entire space
since $z$ is drawn from a normal distribution, which is a somewhat
stringent requirement. Our proposed local smoothing technique removes such 
a requirement, but suffers from a worse growth rate.
\subsection{Convergence analysis of the algorithm with local smoothing}
\label{sec:rnd_SA}
In this section, we apply the stochastic approximation scheme presented
in Section~\ref{sec:formulation} to the smooth approximation $\hat f$ of a
nondifferentiable function $f$. First, we consider the case when $f$ is
convex but deterministic and then, we consider the case when $f$ is given as
the expectation of a convex function.

\subsubsection{Deterministic nondifferentiable optimization} 
\label{sec:deterministic}
We apply the local smoothing technique to the minimization of a convex
but not necessarily differentiable function $f$.  In particular, suppose
we want to minimize such a function $f$ over some set $X$. We may first
approximate $f$ by a differentiable function $\hat f$ and then minimize
$\hat f$ over $f$.  In this case, by taking the minimum over $x\in X$ in
the relation in Lemma~\ref{lemma:approxf}(b), we see that $f^*\le \hat
f^*\le f^*+\e C$.  Thus, we may overestimate the optimal value $f^*$ of
the original problem by at most $\e C$, where $C$ is a bound on
subgradient norms of $f$. So we consider the following optimization
problem
\begin{equation}\label{eqn:deterministic_problem}
\min_ {x \in X} \left\{ \hat{f}(x)\right\}, \hbox{ where } \hat{f}(x)
	\triangleq \EXP{f(x+z)}.
\end{equation}
We may solve the problem by considering the
method~\eqref{eqn:algorithm}, which takes the following form
\begin{equation}\label{eqn:algo_approxf}
\begin{split}
x_{k+1} & =\Pi_X[x_k-\g_k(\nabla \hat f(x_k)+w_k)]\qquad\hbox{for
}k\ge0,\cr w_k&=g_k-\nabla \hat f(x_k)\qquad\hbox{with }g_k\in
\partial f(x_k+z_k),
\end{split}
\end{equation}
where $\{z_k\}$ is an i.i.d.~sequence of random variables with
uniform distribution over the $n$-dimensional sphere centered at the
origin and with the radius $\e>0$.

We have the following result.

\begin{proposition}\label{prop:approxfresult}
Let $f$ be defined and convex over some open convex set
$\mathcal{D}\subseteq\mbR^n$. Let $X$ be a closed convex set and let
$\e>0$ be such that $X_\e\subset\mathcal{D}$, where $\e$ is the
parameter of the distribution of the random vector $z$ as given
in~\eqref{eqn:zuniform}. Let Assumptions~\ref{assum:step_error}(a)
and~\ref{assum:bounded_subgradients} hold. Also, assume that 
problem~\eqref{eqn:deterministic_problem} has a solution. Then, the sequence
$\{x_k\}$ generated by method~(\ref{eqn:algo_approxf}) converges
almost surely to some random optimal solution of the problem.
\end{proposition}

\begin{proof}
We show that the conditions of Proposition~\ref{prop:lipschitzgrad}
are satisfied. In particular, under the given assumptions,
the set $X_\e$ is convex and closed (Corollary 9.1.2 in \cite{Rockafellar70}).
Furthermore, the function $F(x,z)=f(x+z)$ is convex and finite
on some open set containing the set $X_\e$ for any
$z\in \Omega=\{\xi\mid \|\xi\|\le \e\}$. Since $z$ is a random  variable
with uniform distribution on
the sphere $\Omega$, we see that $\EXP{F(x,z)}=\EXP{f(x+z)}$ is finite for
every $x\in X$. Thus, Assumption~\ref{assum:convex} is satisfied.
Since $f$ has bounded subgradients on $X_\e$ and $x_k\in X\subset X_\e$,
we have $\|g_k\|\le C$. By Lemma~\ref{lemma:approxf}(a), 
the gradients $\nabla\hat f(x)$ over $X$ are also bounded
uniformly by $C$. Hence,
\[\|w_k\|\le\|g_k\|+\|\nabla \hat f(x_k)\|\le 2C,\]
implying that $\EXP{\|w_k\|^2\mid\sF_k}\le 4C^2$. In view of this, and
$\sum_{k=0}^\infty \g_k^2<\infty$ (Assumption~\ref{assum:step_error}(a)),
it follows that $\sum_{k=0}^\infty \g_k^2\EXP{\|w_k\|^2\mid\sF_k}<\infty$,
thus showing that Assumption~\ref{assum:step_error}(b) is satisfied.
By Lemma~\ref{lemma:approxf}, the function $\hat f$ is differentiable
with Lipschitz gradients over $X$. Thus, the conditions of
Proposition~\ref{prop:lipschitzgrad} are satisfied and the result follows.
\end{proof}

\subsubsection{Stochastic nondifferentiable optimization}\label{sec:stochopt}
In this section, we apply our local smoothing technique to a
nondifferentiable stochastic problem of the
form~\eqref{eqn:problem}. Essentially, this amounts to putting the
results of Sections~\ref{sec:formulation} and~\ref{sec:rnd_SA}
together. We thus consider the following problem:
\begin{equation}\label{eqn:stoch_problem_init}
\begin{split}
&\begin{array}{cc}
  \hbox{minimize } & \hat f(x)\cr
  \hbox{subject to } & x\in X
 \end{array}\cr
&\hbox{ where } \hat f(x)=\EXP{f(x+z)},\quad f(x)=\EXP{F(x,\xi)},
\end{split}\end{equation}
$F$ is the function as described in section~\ref{sec:formulation}, and
$\hat f$ is a smooth approximation of $f$ with $z$ having a uniform
density $p_u$ as discussed in Section~\ref{sec:localrand}. In view of
Lemma~\ref{lemma:approxf}(a), we know that $\e C$ is an upper bound for
the difference between the optimal value $f^*=\min_{x\in X}f(x)$ and
$\hat f^*=\min_{x\in X} \hat f(x)$, under appropriate conditions to be
stated shortly. Under these conditions, we are interested in solving the
approximate problem in~\eqref{eqn:stoch_problem_init}.

Note that
\[\hat f(x)=\EXP{f(x+z)}=\EXP{\EXP{F(x+z,\xi)\mid \xi}},\]
where the inner expectation is conditioned on $\xi$ and is with
respect to $z$ while the outer expectation is with respect to
$\xi$. We note that the variables $\xi$ and $z$ are independent,
and by exchanging the order of the expectations, we obtain:
\[\hat f(x)=\EXP{\hat F(x,\xi)},\qquad\hbox{with }
\hat F(x,\xi)=\EXP{F(x+z,\xi)}.\] Thus, the problem
in~\eqref{eqn:stoch_problem_init} is equivalent to
\begin{equation}\label{eqn:stoch_problem}
\begin{split}
&\begin{array}{ll}
  \hbox{minimize } & \hat f(x), \hbox{ where }
  \hat{f}(x) =\EXP{\hat F(x,\xi)},\hat F(x,\xi)=\EXP{F(x+z,\xi)}\cr
  \hbox{subject to } & x\in X
 \end{array}\cr
\end{split}\end{equation}

In the following lemma, we provide some conditions ensuring the
differentiability of $\hat F$ with respect to $x$, as well as some
other properties of $\hat F$. The lemma can be viewed as an
immediate extension of Lemma~\ref{lemma:approxf} to the collection
of functions $F(\cdot,\xi)$.

\begin{lemma}\label{lemma:approxstochF}
Let the set $X$ and function $F:\mathcal{D}\times\Omega\to\mbR$
satisfy Assumption~\ref{assum:convex}. Let the parameter $\e$ that
characterizes the distribution of $z$ be such that $X_\e\subset
\mathcal{D}$. In addition, assume that the subdifferential set
$\partial_x F(x,\xi)$ is uniformly bounded over the set
$X_\e\times\Omega$, i.e.,  there is a constant $C$ such that
\[\|s\|\le C\quad\hbox{for all $s\in\partial_x F(x,\xi)$,
and all $x\in X_\e$ and $\xi\in\Omega$}.\] Then, for the
function $\hat F:\mathcal{D}\times\Omega\to\mbR$ given by $\hat
F(x,\xi)=\EXP{F(x+z,\xi)}$, we have:
\begin{itemize}
    \item[(a)]
For every $\xi\in\Omega$, the function $\hat
F(\cdot,\xi)$ is convex and differentiable with respect
to $x$ at every $x\in X$, and the  gradient $\nabla_x \hat F(x,\xi)$
is given by
\[\nabla \hat F(x,\xi) = \EXP{\partial F(x+z,\xi)}
\qquad\hbox{for all }x\in X.\] Furthermore, $\|\nabla_x \hat
F(x,\xi)\|\le C$ for all $x\in X$ and $\xi\in\Omega$.
    \item[(b)]
$F(x,\xi) \le \hat F(x,\xi) \le F(x,\xi)+ \e C$ for all
$x\in X$ and $\xi\in\Omega$.
    \item[(c)]
$\|\nabla_x \hat F(x,\xi) -\nabla_x \hat
F(y,\xi)\| \le \kappa \dfrac{ n!!}{(n-1)!!}
\,\dfrac{C}{\e}\|x-y\|$ for all $x,y\in X$ and
$\xi\in\Omega$, where $\kappa= \frac{2}{\pi}$ if $n$ is
even, and otherwise $\kappa = 1$.
\end{itemize}
\end{lemma}

\begin{proof}
Under the given assumptions, each of the functions
$F(\cdot,\xi)$ for $\xi\in\Omega$ satisfies the
conditions of Lemma~\ref{lemma:approxf}. Thus, the results follow by
applying the lemma to each of the functions
$F(\cdot,\xi)$ for $\xi\in\Omega$.
\end{proof}

In the light of Lemma~\ref{lemma:approxf}, the optimal value $\hat
f^*$ of the approximate problem in~\eqref{eqn:stoch_problem} is an
overestimate of the optimal value $f^*$ of the original
problem~\eqref{eqn:problem} within the error $\e C$. In particular,
by taking the expectation with respect to $\xi$ in the relation of
Lemma~\ref{lemma:approxf}(b), we obtain
\[f^*\le \hat f^*\le f^*+\e C.\]

This motivates solving approximate
problem~\eqref{eqn:stoch_problem}. Since for every $\xi
\in\Omega$, the function $\hat F(\cdot,\xi)$ is convex
and differentiable over the set $X$, the function $\hat
f(x)=\EXP{\hat F(x,\xi)}$ is also convex and differentiable over the
set $X$ (see~\cite{Bertsekas73}). Thus, the objective function $\hat
f$ in~\eqref{eqn:stoch_problem} is differentiable. To solve the
problem, we consider the method in~\eqref{eqn:algorithm}, which
takes the following form:
\begin{equation}\label{eqn:algo_approxstoch}
\begin{split}
x_{k+1}&=\Pi_X[x_k-\g_k(\nabla \hat f(x_k)+w_k)]\qquad\hbox{for
}k\ge0,\cr w_k&=s_k-\nabla \hat f(x_k)\qquad\hbox{with } s_k\in
\partial_x F(x_k+z_k,\xi_k).
\end{split}
\end{equation}

We have the following convergence result for the method.
\begin{proposition}\label{prop:approxstochFresult}
Let the assumptions of Lemma~\ref{lemma:approxstochF} hold, and let
Assumption~\ref{assum:step_error} hold. Then, the sequence $\{x_k\}$
generated by method~(\ref{eqn:algo_approxstoch}) converges almost
surely to some optimal solution of
problem~\eqref{eqn:stoch_problem}.
\end{proposition}

\begin{proof}
It suffices to show that the conditions of
Proposition~\ref{prop:lipschitzgrad} are satisfied for the set $X$,
and the functions $\hat F(x,\xi)$ and $\hat f(x)$. The result will
then follow from Proposition~\ref{prop:lipschitzgrad}.

We first verify that $\hat F(x,\xi)$ satisfies
Assumption~\ref{assum:convex} and that $\hat f(x)$ has Lipschitz
gradients over $X$. Under the given assumptions,
Lemma~\ref{lemma:approxstochF} holds. By
Lemma~\ref{lemma:approxstochF}(a)--(b), the function $\hat F(x,\xi)$
satisfies Assumption~1. Furthermore, by
Lemma~\ref{lemma:approxstochF}(a) and (c), the function $\hat
F(x,\xi)$ is differentiable and with Lipschitz gradients for every
$\xi\in\Omega$. Hence, $\hat f(x)=\EXP{\hat
    F(x,\xi)}$ is also
differentiable with the gradient given by $\nabla \hat
f(x)=\EXP{\nabla_x \hat F(x,\xi)}$(see~\cite{Bertsekas73}). To see
that the gradients $\nabla\hat f$ are Lipschitz continuous, we take
the expectation in the relation of
Lemma~\ref{lemma:approxstochF}(c), and we obtain for all $x,y\in X$,
\[\EXP{\|\nabla_x \hat F(x,\xi) -\nabla_x \hat F(y,\xi)\|}
\le \kappa \dfrac{ n!!}{(n-1)!!} \,\dfrac{C}{\e}\|x-y\|,\] where
$\kappa= \frac{2}{\pi}$ if $n$ is even, and otherwise $\kappa = 1$.
Using Jensen's inequality, we further have for all $x,y\in X$,
\[\| \EXP{\nabla_x \hat F(x,\xi)}-\EXP{\nabla_x \hat F(y,\xi)}\|
\le \kappa \dfrac{ n!!}{(n-1)!!} \,\dfrac{C}{\e}\|x-y\|.\] Since
$\nabla \hat f(x)=\EXP{\nabla_x \hat F(x,\xi)}$, it follows that
$\nabla \hat f(x)$ is Lipschitz over the set $X$. Thus, the
objective function $\hat f$ satisfies the conditions of
Proposition~\ref{prop:lipschitzgrad}.

We now show that Assumption~\ref{assum:step_error}(b) is satisfied.
In view of the assumption that $\sum_{k=0}^\infty \g_k^2<\infty$
(Assumption~\ref{assum:step_error}(a)), it suffices to show that
$\|w_k\|$ is uniformly bounded. By the definition of $w_k$
in~\eqref{eqn:algo_approxstoch}, we have for all $k$,
\[\|w_k\|\le\|s_k\|+\|\nabla \hat f(x_k)\|
\qquad\hbox{with }s_k\in\partial_x F(x_k+z_k,\xi_k),\] where $x_k\in
X$ and $\|z_k\|\le\e$ for all $k$. Thus, $x_k+z_k\in X_\e$ for all
$k$. By the assumptions of Lemma~\ref{lemma:approxstochF}, the
subdifferential set $\partial_x F(x,\xi)$ is uniformly bounded over
$X_\e\times\Omega$, implying that
\begin{equation}\label{eqn:wkbound}
\|w_k\|\le C+\|\nabla \hat f(x_k)\|\qquad\hbox{for all }k\ge0.
\end{equation}
We next prove that the gradients $\nabla\hat f(x)$ are uniformly
bounded over the set $X$. Taking the expectation in the relation
$\|\nabla_x \hat F(x,\xi)\|\le C$ valid for any $x\in
X$ and $\xi\in\Omega$ (Lemma~\ref{lemma:approxstochF}(a)),
and using Jensen's inequality, we obtain
\[\| \EXP{\nabla_x \hat F(x,\xi)}\|
\le \EXP{\|\nabla_x \hat F(x,\xi)\|}\le C\qquad\hbox{for $x\in
X$}.\] Since $\nabla \hat f(x)=\EXP{\nabla_x \hat F(x,\xi)}$, we see
that $\|\nabla \hat f(x)\|\le C$ for $x\in X$. This and
relation~\eqref{eqn:wkbound} yields
\[\|w_k\|\le 2 C\qquad\hbox{for all }k\ge0.\]
thus showing that $\|w_k\|$ is uniformly bounded.
\end{proof}

 \section{Numerical results}\label{sec:numerical}
In this section, we present computational results of applying our
adaptive and smoothing schemes to \us{three} test problems. Sections
\ref{sec:utility_problem}, \ref{sec:BMG_problem} and \ref{sec:snum}
\us{consider} a stochastic utility problem (see \cite{Nemirovski09}), a
bilinear matrix game and a stochastic network utility maximization
problem, respectively. In all of these examples, we \us{compare} the
performance of the recursive steplength SA scheme (RSA) and the
cascading steplength SA scheme (CSA) with a standard implementation of
stochastic approximation.  The standard SA scheme, where the steplength
sequence is \us{chosen} to be a harmonic \us{sequence} is referred to as
the HSA scheme and is employed as a benchmark.  \us{For each example, we
provide this comparison for 9 problems of varying size and problem
parameters apart from figures illustrating the difference between
theoretical bounds and the obtained results.}   Notably, the first two
problems are nonsmooth convex problems, prompting us to work with a
regularized strongly convex form. In Section \ref{sec:discussion}, we
discuss the sensitivity of the schemes to changes in parameters.
Throughout Section \ref{sec:numerical},  we use $N, n, \eta$ and $\e$,
		   to denote the no. of iterations, the problem dimension, the
		   strong convexity parameter, and the size of the uniform
		   distribution employed for smoothing, respectively.
 
\subsection{Examples}
\subsubsection{A stochastic utility problem}\label{sec:utility_problem}
Consider the following optimization problem,
\begin{align}\label{eqn:utility_problem}
\min_{x \in X} \left
\{f(x)=\EXP{\phi\left(\sum_{i=1}^{n}\left(\frac{i}{n}+\xi_i\right)x_i\right)}
\right\},
\end{align}
where $X=\{x \in R^n | x \geq 0, \sum_{i=1}^{n}x_i=1\}$, $\xi_i $	are
independent and normally distributed random variables with mean zero and
variance one. The function $\phi(\cdot)$ is a piecewise linear convex
function given by $\phi(t)=\max_{1\le i \le m}\{v_i+s_it\},$	 where
$v_i$ and $s_i$ are constants between zero and one, and
$F(x,\xi)=\phi(\sum_{i=1}^{n}(\frac{i}{n}+\xi_i)x_i))$. To apply our
schemes, we require strong convexity of function $f$. Therefore, we
regularize $f$ by adding the term $\frac{\eta }{2}\|x\|^2$ to $f$ where
$\eta >0$ is the strong convexity parameter. We now apply the randomized
smoothing technique discussed in Section \ref{sec:rnd_SA}. Smoothed
regularized problem given by
\begin{align}\label{eqn:inexact_utility_problem}
\min_{x \in X}\left\{\hat f(x)\triangleq
\EXP{\phi(\sum_{i=1}^{n}(\frac{i}{n}+\xi_i)(x_i+z_i))+\frac{\eta}{2}\|x+z\|^2}
\right\},
\end{align}
where $z \in \mathbb{R}^n$ is the uniform distribution on a ball with
radius $\epsilon$ with independent elements $z_i$, $1 \leq i \leq n$.
{We let $x^*$ denote an optimal solution of problem
(\ref{eqn:utility_problem}) and $x_{\epsilon, \eta}^*$ be} the unique
optimal solution of problem (\ref{eqn:inexact_utility_problem}). To find
{optimal} solutions, we use an SAA method~\cite{shap03sampling}
which leads to linear and a quadratic program for solving problem
(\ref{eqn:utility_problem}) and problem
(\ref{eqn:inexact_utility_problem}), respectively. 

Table \ref{tab:utility_errors} shows the results of parametric analysis
of the simulation of our schemes on problem \eqref{eqn:inexact_utility_problem}. The
table is partitioned into three parts, each corresponding to a variation
of parameters $n$, $N$,  $\eta$, respectively. In each
part, one parameter has been assigned three increasing values while the
other parameters are kept fixed, allowing us to ascertain   the impact of each
parameter on the performance of the schemes. We generated 50
trajectories of the RSA and CSA scheme for a given $n, N, \eta,
\epsilon.$ Over these realizations, we computed the means and 90$\%$
confidence intervals. The baseline parameters are chosen as $n=20$,
$N=4000$, $\e=0.5$, and $\eta=0.5$ as a reference for each group.  Note
that in Table~\ref{tab:utility_errors}, the confidence intervals employ
the logarithm of the error. 
Recall that we have a theoretical upper bound on the error
$\EXP{\|x_k -x_{\e,\eta}^*\|^2}$, as given by \eqref{rate_nu_est} and
\eqref{constant-rate} for the RSA and CSA schemes. Additionally, we
{obtain an empirical error bound based on using the scheme in 
practice. 
{\bf Insights:} 
We observe that the confidence intervals of both the CSA and the RSA
schemes are relatively invariant to changes in problem dimension.
Furthermore, RSA appears to have provide slightly tighter intervals in
comparison with CSA. Expectedly, increasing $N$ leads to significant
improvement in these intervals while larger values of $\eta$ lead to
less accurate solutions (with respect to the unregularized problem) but
tighter bounds. Moreover, the CSA schemes in particular give better
confidence bounds than RSA when $\eta$ is larger. 
\begin{table}[htb] 
\vspace{-0.05in} 
\tiny 
\centering 
\begin{tabular}{|c|c|c|c|c|c||c||c||c||c|} 
\hline 
-&P$(i)$ & $n$ & $N$ & $\epsilon$ & $\eta$ & HSA - $90\%$ CI & RSA - $90\%$ CI
& CSA - $90\%$ CI&$\|x_{\epsilon, \eta}^*-x^*\|^2$ \\ 
\hline 
\hline
$n$ &1 &  10& 4000& $5.0$e$-1$ & $5.0$e$-1$& [$1.00 $e${+0}$,$1.01 $e${+0}$]
&  [$1.58 $e${-3}$,$1.96 $e${-3}$]
&  [$1.47 $e${-3} $, $1.93 $e${-3}$] & $3.28 $e${-2}$ \\

\hbox{ }& 2 &20& 4000& $5.0$e$-1$ & $5.0$e$-1$& [$1.03 $e${+0}$,$1.04 $e${+0}$] & [$1.74 $e${-3}$,	
$2.21 $e${-3}$]&  [	$1.49 $e${-3} $, $1.88 $e${-3}$] & $1.84 $e${-2}$ \\

\hbox{ }&3 & 40& 4000& $5.0$e$-1$ & $5.0$e$-1$& [$1.03 $e${+0}$,$1.04 $e${+0}$] & [$2.21 $e${-3}$,	
$2.54 $e${-3}$]&  [$2.24 $e${-3} $, $2.74 $e${-3}$] & $6.49 $e${-2}$ \\ 
\hline 
$N$ & 4 &   20& 1000& $5.0$e$-1$ & $5.0$e$-1$& [$1.05 $e${+0}$,$1.05 $e${+0}$] & [$3.76 $e${-3}$,	$4.74	 
$e${-3}$] & [$4.67 $e${-3} $, $5.96 $e${-3}$] & $1.84 $e${-2}$ \\

\hbox{ }& 5 & 20& 2000& $5.0$e$-1$ & $5.0$e$-1$& [$1.04 $e${+0}$,$1.05 $e${+0}$] & [$2.86 $e${-3}$,	
$3.63 $e${-3}$]& [$2.78 $e${-3} $, $3.57 $e${-3}$] & $1.84 $e${-2}$ \\

\hbox{ }& 6 &  20& 4000& $5.0$e$-1$ & $5.0$e$-1$& [$1.03 $e${+0}$,$1.04 $e${+0}$] & [$1.74 $e${-3}$,	
$2.21 $e${-3}$]& [$1.49 $e${-3} $, $1.88 $e${-3}$] & $1.84 $e${-2}$  \\ 
\hline 

$\eta$ &7 &  20& 4000& $5.0$e$-1$ & $2.5$e$-2$& [$1.13 $e${+0}$,$1.13 $e${+0}$] & [$2.77 $e${-3}$,	
$3.48 $e${-3}$]&  [$2.73 $e${-3} $, $3.51 $e${-3}$] & $9.63 $e${-3}$ \\

\hbox{ }&8 &  20& 4000& $5.0$e$-1$ & $5.0$e$-1$& [$1.03 $e${+0}$,$1.04 $e${+0}$] & [$1.74 $e${-3}$,	
$2.21 $e${-3}$]& [$1.49 $e${-3} $, $1.88 $e${-3}$] & $1.84 $e${-2}$ \\

\hbox{ }&9 & 20& 4000& $5.0$e$-1$ & $1.0$e$+0$& [$0.83 $e${+0}$,$0.84 $e${+0}$]  & [$9.70 $e${-4}$, 
$1.21 $e${-3}$]&  [$1.07 $e${-3} $, $1.30 $e${-3}$] & $4.52 $e${-2}$ \\ \hline
\end{tabular} 
\caption{Stochastic utility problem: HSA, RSA, CSA} 
\label{tab:utility_errors} 
\vspace{-0.2in} 
\end{table}	
 
\subsubsection{A bilinear matrix game problem}\label{sec:BMG_problem}
We consider a bilinear matrix game,
\begin{align}\label{equ:BMG_prob}
\min_{x \in X} \max_{y \in Y} y^T A x,
\end{align}
where $X=Y=\{x \in \mathbb{R}^n: \sum_{i=1}^n x_i=1, x \geq0\}.$ 
Furthermore, $A$ is a symmetric matrix whose entries are
\begin{align}\label{equ:matrix A}
A_{ij}=\frac{i+j-1}{2n-1}\quad 1\leq i,j \leq n.
\end{align}
{Problem~\eqref{equ:BMG_prob}} a saddle point problem. Solving saddle point 
problems by SA algorithm has been discussed extensively (cf.~\cite{Nedich09}). 
The gradient and its sampled variant to be employed in algorithm
\eqref{eqn:algorithm} {are given by}:
\begin{align}\label{equ:BMG-subgradient}
g(x,y)=\left(
			\begin{array}{ccc}
A^Ty \\
-Ax  \end{array} \right), \qquad 
G(x,y,\xi)=\left( \begin{array}{ccc}
A_{\cdot,l(y,\xi_1)} \\
-A_{l(x,\xi_2),\cdot} \end{array} \right),
\end{align}
respectively 
where $l(y,\xi_1)$ and $\l(x,\xi_2)$ are random integers between 1 and $n$ 
with probabilities 
\[\frac{y_q-\min (0,y_1,\ldots,y_n)}{\sum_{j=1}^n (y_j -\min
		(0,y_1,\ldots,y_n))}, \quad   1 \leq q \leq n, \qquad 
 \frac{x_p-\min (0,x_1,\ldots,x_n)}
{\sum_{i=1}^n (x_i -\min (0,x_1,\ldots,x_n))}, \quad   1 \leq  p \leq n, \]
 respectively for arbitrary vectors $x$ and $y$. We generate these
 random variables through two independent random variables $\xi_1$ and
 $\xi_2$ which are uniformly distributed in $[0,1]$. Now, for any $(x,y)
	\in X \times Y$, since
$\min (0,x_1,\ldots,x_n)=\min (0,y_1,\ldots,y_n)=0$,
and $ \sum_{i=1}^n x_i =\sum_{j=1}^n y_j=1$, we have
\[\EXP{G(x,y,(l(y,\xi_1),l(x,\xi_2)))}= \left( \begin{array}{ccc}
A^Ty \\
-Ax  \end{array} \right)=g(x,y),\]
implying
that $w_k$ has zero-mean, i.e.,
$\EXP{w_k\mid\sF_k}=0$ for all $k\ge0.$
To analyze the behavior of the upper bound of error arising from RSA and
CSA, we need a strongly convex function. This is obtained by
adding a regularization term
$\frac{\eta}{2}\|x\|^2-\frac{\eta}{2}\|y\|^2$ to the function $y^TAx$
which makes it a strongly convex function with respect to $x$ and a
strongly concave function with respect to $y$. To apply the randomized
technique in Section~\ref{sec:localrand}, we consider an
($2n$)-dimensional ball with radius $\epsilon$ uniformly distributed.
We use the following SA algorithm to find the
solution to an approximate solution of
(\ref{equ:BMG_prob}):
\begin{equation}\label{eqn:algo_BMG}
\begin{split}
x_{k+1} & 
=\Pi_X[x_k-\g_k(G(x_k+\zeta_1^k,y_k+\zeta_2^k,l(y_k+\zeta_2^k,\xi_1^k))
+\eta(x_k+\zeta_1^k))] \qquad\hbox{for all
}k\ge0 ,
\cr 
y_{k+1} 
& =\Pi_Y[y_k+\g_k(G(x_k+\zeta_1^k,y_k+\zeta_2^k,l(x_k+\zeta_1^k,\xi_2^k))
-\eta(y_k+\zeta_2^k	))] 
\qquad\hbox{for all
}k\ge0, 
\end{split}\end{equation}
where $\zeta_1 \in \mathbb{R}^n$ and $\zeta_2 \in \mathbb{R}^m$ are 
random vectors with uniform distribution in the ($n+m$)-dimensional ball with 
radius $\epsilon$. 

From the structure of $A$ in (\ref{equ:matrix A}), it is observed that
the optimal solution of problem (\ref{equ:BMG_prob}) is obtained for
$x^*=[1,0,\ldots,0]^T$ and $y^*=[0,\ldots,0,1]^T$. This result can also be
obtained quite simply by using a  linear programming reformulation. The
regularized problem cannot be analyzed as easily and its solution can be
obtained by using QP duality and SAA techniques.

Table \ref{tab:BMG_errors} presents the results of simulations for RSA
and CSA schemes. Similar to the Table \ref{tab:utility_errors}, there
are three parts in the Table \ref{tab:BMG_errors} for the parameters. For
this problem, $\|x_{\epsilon, \eta}^*-x^*\|^2$ is very small and shows
that the optimal solution of the approximate problem is very close to
the optimal solution of problem (\ref{equ:BMG_prob}). We set $n=20$,
	$N=4000$, $\e=0.2$, and $\eta=0.01$ as the reference setting. Figure
	\ref{fig:BMG_mean} shows the theoretical upper bounds and the mean
	of samples of simulation for RSA and CSA schemes.

\begin{table}[htb] 
\vspace{-0.05in} 
\tiny 
\centering 
\begin{tabular}{|c|c|c|c|c|c||c||c||c||c|} 
\hline 
-&P$(i)$ & $n$ & $N$ & $\epsilon$ & $\eta$ & HSA - $90\%$ CI & RSA - $90\%$ CI
& CSA - $90\%$ CI&$\|x_{\epsilon, \eta}^*-x^*\|^2$ \\ 
\hline 
\hline
$n$ &1 &  10& 4000& $2.0$e$-1$ & $1.0$e$-2$& [$1.92$e$+0$,	
$1.92$e$+0$] & [$8.00$e$-12$,	
$8.00$e$-12$]&  [$2.00$e$-12$, $2.00$e$-12$] & $0.00$e${-12}$ \\

\hbox{ }& 2 &20& 4000& $2.0$e$-1$ & $1.0$e$-2$& [$1.92$e$+0$,	
$1.92$e$+0$]  & [$8.00$e$-12$, $9.00$e$-12$]&  
[$5.50$e$-10$, $5.76$e$-10$] & $0.00$e${-12}$ \\

\hbox{ }&3 & 40& 4000& $2.0$e$-1$ & $1.0$e$-2$& [$1.92$e$+0$,	
$1.92$e$+0$]  & [$9.82$e$-2$, $9.82$e$-2$]&  
[$3.55$e$-9$, $3.70$e$-9$] & $0.00$e${-12}$ \\  
\hline 
\hline
$N$ & 4 &   20& 1000& $2.0$e$-1$ & $1.0$e$-2$& [$1.92$e$+0$,	
$1.92$e$+0$]  & [$2.79$e$-1$, $2.79$e$-1$]&  
[$1.12$e$-1$, $1.12$e$-1$] & $0.00$e${-12}$ \\

\hbox{ }& 5 & 20& 2000& $2.0$e$-1$ & $1.0$e$-2$& [$1.93$e$+0$,	
$1.93$e$+0$]  & [$1.07$e$-1$, $1.07$e$-1$]&  
[$5.37$e$-10$, $5.77$e$-10$] & $0.00$e${-12}$ \\

\hbox{ }& 6 &  20& 4000& $2.0$e$-1$ & $1.0$e$-2$& [$1.92$e$+0$,	
$1.92$e$+0$] & [$8.00$e$-12$, $9.00$e$-12$]&
 [$5.50$e$-10$, $5.76$e$-10$] & $0.00$e${-12}$ \\  
\hline 
\hline

$\eta$ &7 &  20& 4000& $2.0$e$-1$ & $5.0$e$-3$& [$1.96$e$+0$,	
$1.96$e$+0$] & [$1.13$e$-1$,	$1.13$e$-1$]&  
[$-1.15$e$-10$, $2.51$e$-10$] & $0.00$e${-12}$ \\

\hbox{ }&8 &  20& 4000& $2.0$e$-1$ & $1.0$e$-2$& [$1.92$e$+0$,	
$1.92$e$+0$] & [$8.00$e$-12$, $9.00$e$-12$]& 
[$5.50$e$-10$, $5.76$e$-10$] & $0.00$e${-12}$	 \\

\hbox{ }&9 & 20& 4000& $2.0$e$-1$ & $2.0$e$-2$& [$1.84$e$+0$,	
$1.84$e$+0$]  & [$1.07$e$-10$,	$1.46$e$-10$]& 
[$3.29$e$-9$, $3.55$e$-9$] & $0.00$e${-12}$ \\ 
 \hline
\end{tabular} 
\caption{Bilinear matrix game problem: HSA, RSA, CSA} 
\label{tab:BMG_errors} 
\vspace{-0.2in} 
\end{table}	

{\bf Insights:} Unlike in the stochastic utility problem, in this
instance, the true optimal solution is obtained within the $N$ gradient
steps for most of the test problems. However, it should be remarked that
the CSA appears to find solutions faster than RSA, in at least three of
the problems (P$(i)$: 3, 5 and 7).  	

 \subsubsection{A stochastic network utility problem}\label{sec:snum}
 In this \us{example}, we consider a spatial network and consider the
 associated \us{network utility maximization problem}
 (See~\cite{Kelly98,Srikant04}). Suppose that there are $n$
 users and $L_1$ links.  \us{The overall network maximization
problem is characterized by an objective that is a sum of user-specific
concave	utilities less a congestion cost, which is given by a function of
	aggregate flow over a link}. \us{Let $x_i$ denote the $i$th user's flow
	rate while $F_i(x_;\xi)$ denotes its utility function, defined by}
\begin{align*}
F_i(x_i,\xi_i) \triangleq -k_i(\xi_i)\log(1+x_i),
\end{align*} 
where $k_i(\xi_i)$ is an uncertain parameter.  Suppose that $A$ denotes the
adjacency matrix that captures the set of links traversed by the
traffic. More precisely, for every link $l \in {\mathcal L}$ and user $i$, we have
$A_{li} =1$ if link $l$ carries flow of user $i$ and $A_{li}=0$
otherwise.  The congestion cost is given by $c(x)=\|Ax\|^2.$
\us{The total cost at the network level us then given by} 
\[F(x,\xi)=-\sum_{i=1}^{N}k_i(\xi_i)\log(1+x_i)+ \|Ax\|^2.\]
Therefore
\[\nabla F(x,\xi)= \left( \begin{array}{ccc}
-\frac{k_1}{1+x_1}  \\
\vdots  \\
-\frac{k_N}{1+x_N} \end{array} \right)+2A^TAx.\]
We assume that the user traffic rates are restricted by a capacity
constraint $Ax \leq C$. Since the objective function $F$ is smooth,
		   there is no requirement to introduce an additional smoothing. 

Table \ref{tab:network_errors} shows the results of simulations for HSA,
	  RSA, and CSA scheme. Here, we assume that
	  $C_3=(0.10,0.15,0.20,0.10,0.15,0.20,0.20,0.15,0.25)=0.75C_2=0.5C_1$
	  and \us{$x$ is constrained to be nonnegative.} We also assume that
	  $k_i(\xi_i)$ is drawn from uniform distribution $Uni(0.2,1)$ for
	  every user. The confidence intervals for the normed error between the
	  terminating iterate and the optimal solution are reported for each
	  problem.

{\bf Insights:} We observe that both RSA and CSA schemes perform
favorably in comparison with the HSA scheme. Importantly, neither scheme
appears to deteriorate from a confidence interval standpoint when 
the problem size  grows. Similar to the earlier examples, CSA appears to
have slightly tighter confidence intervals in the empirical tests that
we carried out. 

\begin{table}[htb] 
\vspace{-0.05in} 
\tiny 
\centering 
\begin{tabular}{|c|c|c|c|c||c||c||c|} 
\hline 
-&P$(i)$ &  $n$ & $N$ & $C$ & HSA - $90\%$ CI & RSA - $90\%$ CI
& CSA - $90\%$ CI \\ 
\hline 
\hline
$C$ &1 &   5& 4000 & $C_1$ & [$1.58 $e${-2}$,$1.89 $e${-2}$]
&  [$5.57 $e${-3}$,$6.81 $e${-3}$]
&  [$3.65 $e${-3} $, $4.55 $e${-3}$]  \\

\hbox{ }& 2 & 5& 4000 & $C_2$ & [$1.16 $e${-2}$,$1.38 $e${-2}$]
&  [$4.47 $e${-3}$,$5.86 $e${-3}$]
&  [$3.62 $e${-3} $, $4.52 $e${-3}$]  \\

\hbox{ }&3 &  5& 4000 & $C_3$ & [$9.08 $e${-3}$,$1.09 $e${-2}$]
&  [$4.30 $e${-3}$,$5.32 $e${-3}$]
&  [$3.62 $e${-3} $, $4.52 $e${-3}$]  \\ 
\hline 
$n$ & 4 &    5& 4000 & $C_3$ & [$9.08 $e${-3}$,$1.09 $e${-2}$]
&  [$4.30 $e${-3}$,$5.32 $e${-3}$]
&  [$3.62 $e${-3} $, $4.52 $e${-3}$]  \\

\hbox{ }& 5 &  10 & 4000 & $C_3$ & [$1.09 $e${-2}$,$1.31 $e${-2}$]
&  [$4.80 $e${-3}$,$5.94 $e${-3}$]
&  [$4.09 $e${-3} $, $5.08 $e${-3}$]  \\

\hbox{ }& 6 &   15& 4000 & $C_3$ & [$1.04 $e${-2}$,$1.24 $e${-2}$]
&  [$5.21 $e${-3}$,$6.36 $e${-3}$]
&  [$3.76 $e${-3} $, $4.63 $e${-3}$]  \\ 
\hline 

$N$ &7 &   5& 1000 & $C_3$ & [$8.98 $e${-3}$,$1.07 $e${-2}$]
&  [$6.63 $e${-3}$,$7.93 $e${-3}$]
&  [$5.36 $e${-3} $, $6.43 $e${-3}$] \\

\hbox{ }&8 &   5& 2000 & $C_3$ & [$9.70 $e${-3}$,$1.16 $e${-2}$]
&  [$5.65 $e${-3}$,$6.88 $e${-3}$]
&  [$5.32 $e${-3} $, $6.50 $e${-3}$]  \\

\hbox{ }&9 &  5& 4000 & $C_3$ & [$9.08 $e${-3}$,$1.09 $e${-2}$]
&  [$4.30 $e${-3}$,$5.32 $e${-3}$]
&  [$3.62 $e${-3} $, $4.52 $e${-3}$] \\ \hline
\end{tabular} 
\caption{Stochastic network utility problem: HSA, RSA, CSA} 
\label{tab:network_errors} 
\vspace{-0.2in} 
\end{table}	
\subsection{Interpretation of numerical results}\label{sec:discussion}  
\begin{figure}[htb]
 \centering
 \subfloat[Utility Problem]
{\label{fig:utility_mean}\includegraphics[width=2in,height=2in]
{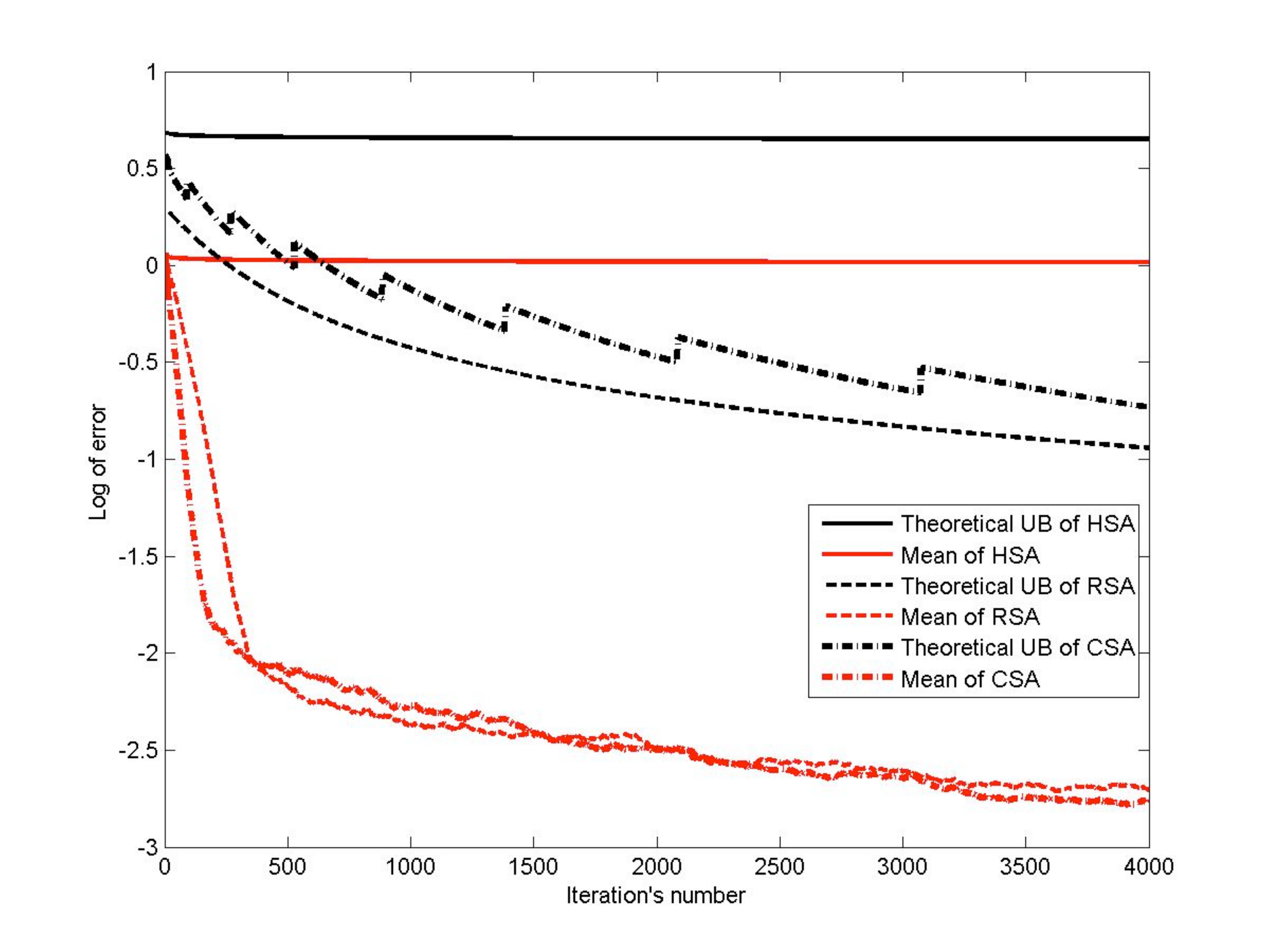}}
 \subfloat[Bimatrix
 Game]{\label{fig:BMG_mean}\includegraphics[width=2in,height=2in]
{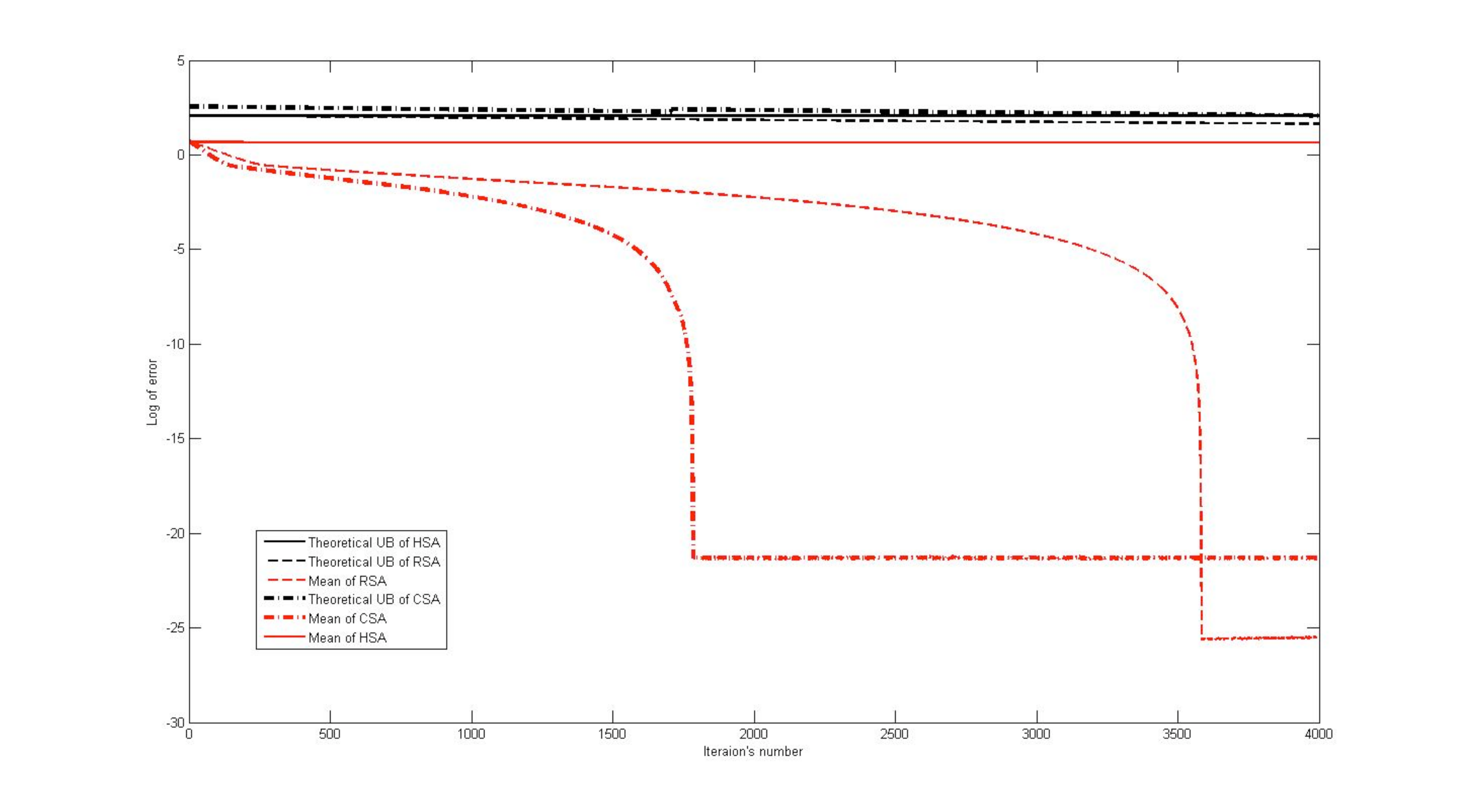}}
 \subfloat[Network Utility Problem
 Game]{\label{fig:NUM_mean}\includegraphics[width=2in,height=2in]
{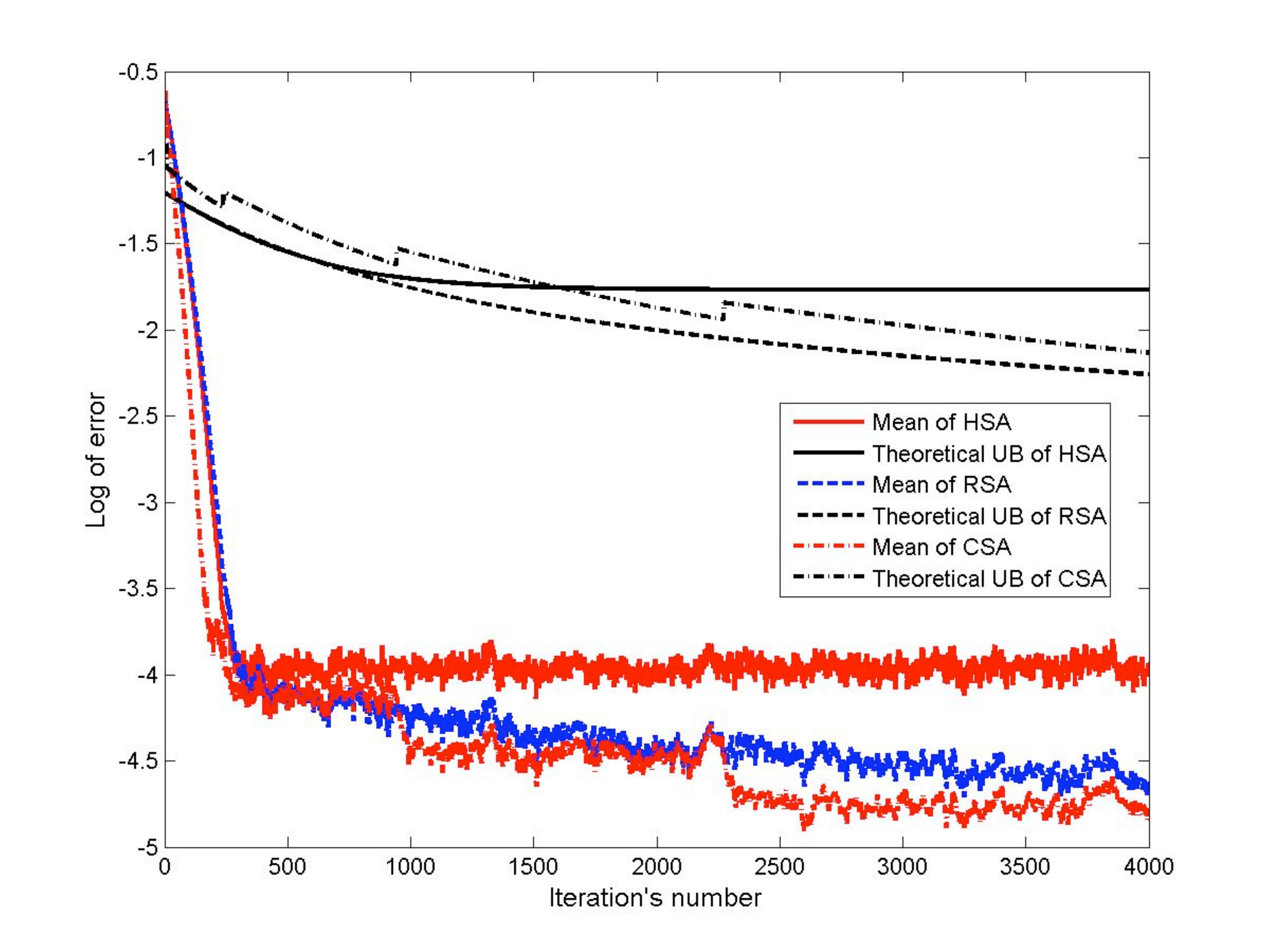}}
\caption{Theoretical and empirical error bounds for RSA and CSA
schemes.}
\end{figure}
\us{In this section, we interpret the numerical results obtained in the
previous subsections, focusing on a comparison between the theoretical
	and empirical results and the sensitivity of the schemes to the
	algorithm parameters.}

\subsubsection{Theoretical and empirical trajectories}
\us{In Figures~\ref{fig:utility_mean},~\ref{fig:BMG_mean} and
~\ref{fig:NUM_mean}, we provide schematics of the trajectories
associated with the theoretically obtained upper bounds and the
empirical means. Several observations can be immediately made. In the
context of the stochastic utility problem and the network utility
maximization problem, we observe that the RSA scheme displays uniformly
better theoretical bounds, in comparison with CSA. It is also worth
emphasizing that the ``jumps'' seen in the theoretical error bound
trajectories of CSA correspond to junctures where the steplengths drop. 
In fact, the cascading nature is also apparent in the empirical
trajectories of the network utility maximization game in
Fig~\ref{fig:NUM_mean}, albeit in a less obvious fashion. We observe
that the overall empirical behavior of both schemes is similar in terms
of the final errors for the utility and network utility maximization
problems while in the context of the bimatrix game, the CSA scheme
performs significantly better for a subset of problems. }

 \begin{figure}[htb]
 \centering{ 
 \subfloat[HSA]{\label{fig:HSA1}\includegraphics[width=2.1in,height=2.1in]
{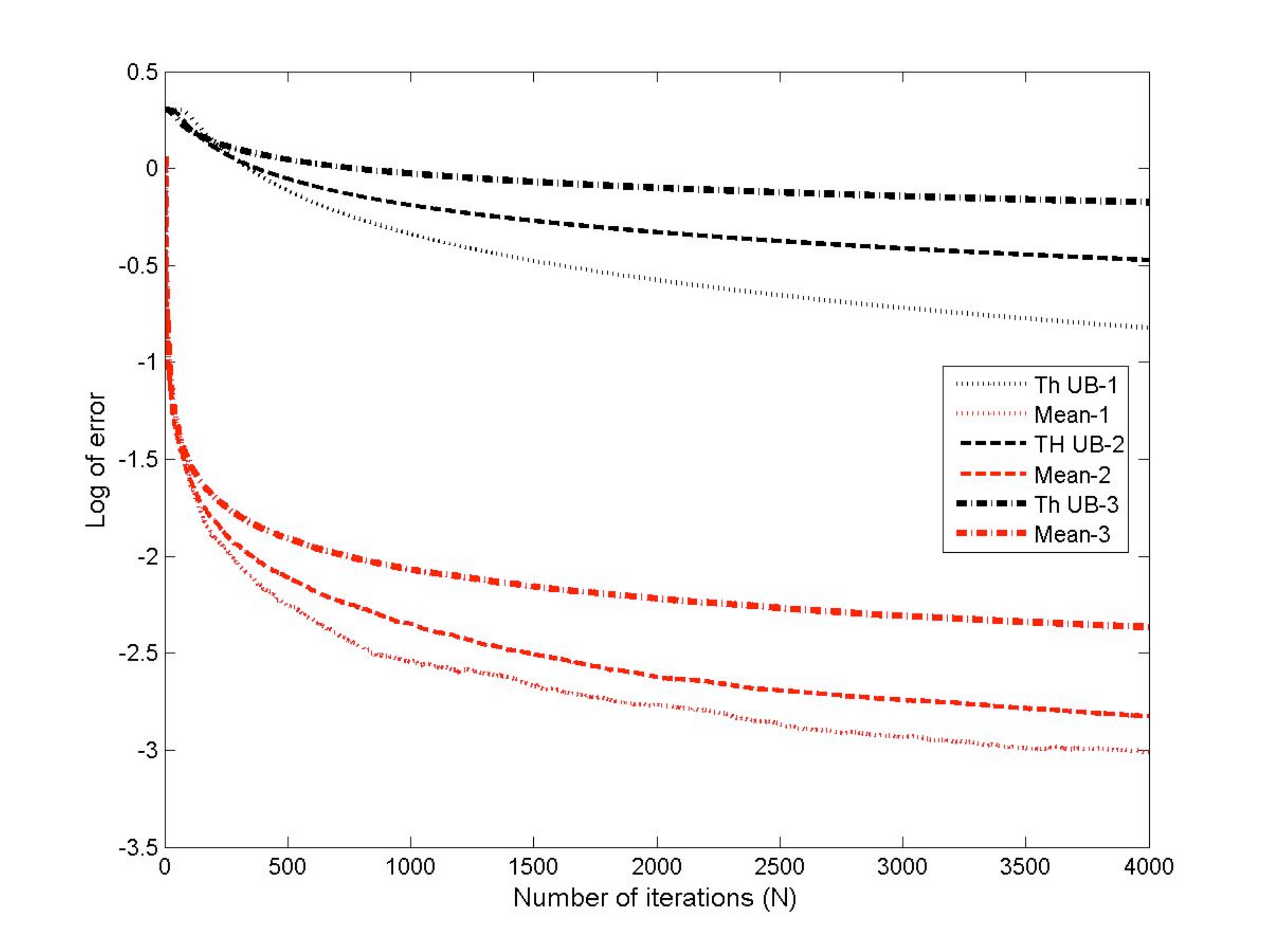}}
 \subfloat[RSA]{\label{fig:RSA1}\includegraphics[width=2.1in,height=2.1in]
{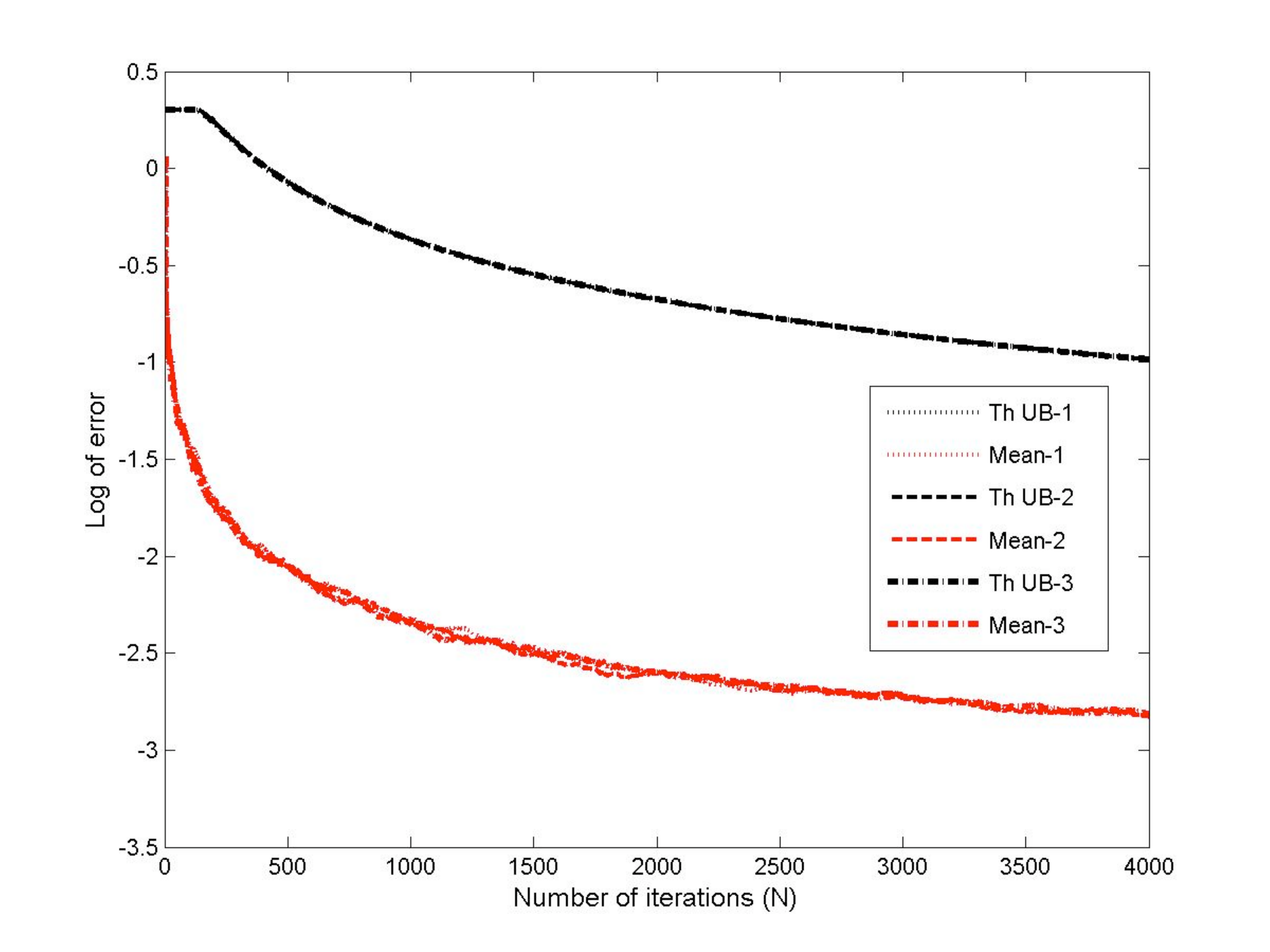}}
 \subfloat[CSA]{\label{fig:CSA1}\includegraphics[width=2.2in,height=2.1in]
{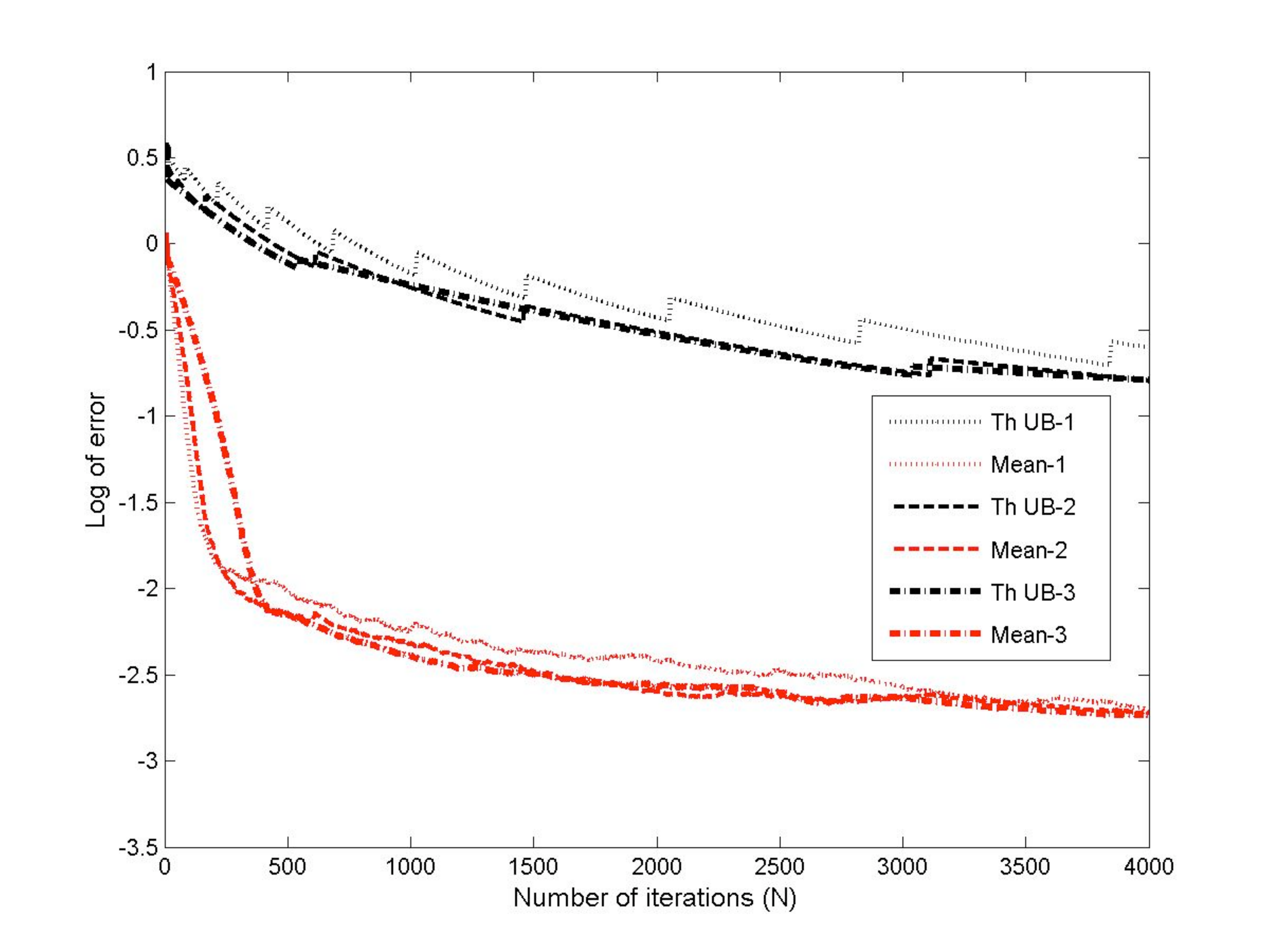}}
}
 \caption{The stochastic utility problem: HSA, RSA, CSA}
 \label{fig:Dis_utility}
\end{figure}
\subsubsection{Sensitivity to algorithm parameters}
Finally, in this section, we discuss the sensitivity of each scheme to
algorithm parameters and provide a comparison with a  standard stochastic
approximation scheme  where we assume 
that the stepsize is $\g_k=
\frac{\alpha}{k}$ for $k \geq 1$ and $\alpha >0$. In HSA, we intend to
examine the effect of choosing different values of $\alpha$ on the performance 
of the SA algorithm. In the RSA scheme, we have a choice of the first stepsize
$\gamma_0^{RSA}$ and also parameter $c$ in the inequality of Proposition
\ref{prop:rec_convergence}. We set $c = 0.5$ and examine the impact of
changing $\gamma_0^{RSA}$. Finally, the CSA scheme  performs differently with 
different choices of the cascading parameter $0<\theta<1$. We consider three
different values for each of $\alpha$, $\gamma_0^{RSA}$, and $\theta$
and present simulations for HSA, RSA and CSA in the
case of the stochastic utility problem.  The reference
setting is specified by $n=20$, $N=4000$, $\epsilon=0.5$, and
$\eta=0.5$. Now suppose $\alpha$, $\gamma_0^{RSA}$, and
$\theta$ are set as follows: \[\alpha=1, 0.5, \hbox{ and } 0.25; \quad \gamma_0^{RSA}=1,
0.5,\hbox{ and } 0.25; \quad \theta=0.75, 0.5,\hbox{ and } 0.25.  \]  
Figure~\ref{fig:Dis_utility} shows the simulations for the specified
parameters. Note that ``Th. UB'' shows the corresponding theoretical upper 
bound of each scheme and "Mean" shows the mean of error
$\|z_k-z_{\epsilon,\eta}^*\|^2$ where $z=(x,y)$.

Figure~\ref{fig:HSA1} shows the harmonic scheme with $\alpha=$1, 0.5,
and 0.25 corresponding to labels 1, 2, and 3 in the legend. This shows
that the performance of HSA is \us{extremely} sensitive to the choice of $\alpha$ and
HSA implementations with a larger $\alpha$ performed better for the stochastic 
utility problem. \us{Furthermore, the error on termination of HSA schemes
can vary by nearly a factor of 10 for the problems that we tested.}
	\us{The update rules in the RSA schemes rely on $\eta$ and $L$ with
		$\gamma_0^{RSA}$ being the sole user input.} Yet, when examining
the sensitivity of the RSA scheme to the choice of $\gamma_0^{RSA}$ (see
		Figure \ref{fig:RSA1} with $\gamma_0^{RSA}=$1, 0.5, and 0.25 corresponding to
labels 1, 2, and 3), we observe that  the performance is relatively
insensitive to the choice of initial stepsize. In effect, the modeler
can be relatively less concerned about such parameters when attempting
to solve this class of problems. Importantly, both theoretical and numerical 
aspect of RSA have almost the same performance for three values of 
$\gamma_0^{RSA}$.  \us{Finally, a concern in the implementation of CSA
schemes is the choice of $\theta$, the cascading parameter where
	$\theta \in (0,1)$.} Figure~\ref{fig:CSA1} shows the simulation of
	the  cascading scheme with $\theta=$0.75, 0.50, and 0.25
	corresponding to labels 1, 2, and 3.  Theoretically, we observe that
	smaller values of $\theta$ (more aggressive reductions in stepsize)
	lead to slightly superior theoretical bounds but not significantly
	so. \us{However, the results are far more muted when conducting an
	empirical examination. In particular, we observe that the CSA
	scheme appears to be relatively insensitive to diversity in the
	choice of $\theta.$ The relative robustness of the RSA and CSA
	schemes to the choice of parameters is seen as a crucial advantage of such schemes. }
\section{Concluding remarks}\label{sec:conclusions}
This paper  is motivated by two shortcomings associated with standard
stochastic approximation procedures for stochastic convex programs.
First, standard implementations of such schemes provide little guidance
in specifying  parameters that may prove crucial in practical
performance. Furthermore, direct extensions to nonsmooth regimes of such
schemes is not immediate. Accordingly, this paper makes two sets of
contributions.  First, we develop two sets of adaptive steplength
schemes and provide the associated global convergence theory. Of these,
the former, a recursive steplength scheme (RSA), specifies the
steplength at a particular iteration using the previous steplength and
certain problem parameters.  The second scheme, called a cascading
steplength scheme (CSA), differs significantly and is essentially a
sequence of constant steplength schemes in which the steplength is
reduced at specific points in time.  The second set of contributions
extends these techniques to settings where the objective is not
necessarily differentiable.  Through the use of a local smoothing method
that perturbs the problem through a uniformly distributed random
variable, we propose a stochastic gradient scheme. Notably, Lipschitz
bounds are obtained {for the gradients and their growth} with problem
size is found to be modest.  Locally smoothed variants of the RSA and
CSA scheme were seen to perform well on two classes of nonsmooth
stochastic optimization problems and implementations were seen to be
relatively insensitive to problem parameters.

\bibliographystyle{IEEEtran}
\bibliography{local_rand}

\end{document}